\documentclass[12pt]{amsart}
\usepackage{amssymb}
\usepackage{amscd}
\usepackage{amsmath}
\theoremstyle{plain}
\newtheorem{Thm}{Theorem}[subsection]
\newtheorem{Thm*}{Theorem}[section]
\newtheorem{Thm'}[Thm]{"Theorem"}
\newtheorem{Cor}[Thm]{Corollary}

\newtheorem{Conj}[Thm]{Conjecture}
\newtheorem{Prop}[Thm]{Proposition}
\newtheorem{Lem}[Thm]{Lemma}
\newtheorem{Cl}[Thm]{Claim}

\theoremstyle{definition}

\newtheorem{Rem}[Thm]{Remark}

\newtheorem{Emp}[Thm]{}

\newtheorem{Q}[Thm]{Question}

\numberwithin{equation}{section}
\newcommand{\nc}{\newcommand}
\nc{\lm}{\lambda}
\newcommand{\pp}{\boxtimes}
\newcommand{\ov}{\overline}

\newcommand{\B}[1]{\mathbb#1}
\newcommand{\cal}[1]{\mathcal{#1}}
\newcommand{\C}[1]{\cal#1}

\newcommand{\isom}{\overset {\thicksim}{\to}}

\newcommand{\om}{\omega}

\newcommand{\lra}{\longrightarrow}
\newcommand{\lla}{\longleftarrow}
\newcommand{\hra}{\hookrightarrow}
\newcommand{\dra}{\dashrightarrow}
\newcommand{\wt}{\widetilde}
\newcommand{\Gm}{\Gamma}

\newcommand{\lan}{\langle}
\newcommand{\Var}{\operatorname{Var}}
\newcommand{\ran}{\rangle}

\renewcommand{\P}{\mathbf P}
\newcommand{\T}{\Upsilon}
\renewcommand{\Q}{\mathbf Q}
\newcommand{\I}{\mathbf I}

\newcommand{\gm}{\gamma}
\newcommand{\dt}{\delta}
\newcommand{\Dt}{\Delta}
\newcommand{\La}{\Lambda}
\newcommand{\bs}{\backslash}

\newcommand{\un}[1]{\underline{#1}}

\newcommand{\al}{\alpha}

\newcommand{\la}{\lambda}

\newcommand{\form}[1]{(\ref{Eq:#1})}

\newcommand{\rl}[1]{Lemma \ref{L:#1}}
\newcommand{\rcn}[1]{Conjecture \ref{C:#1}}

\newcommand{\rcl}[1]{Claim \ref{C:#1}}
\newcommand{\rp}[1]{Proposition \ref{P:#1}}

\newcommand{\re}[1]{\ref{E:#1}}
\newcommand{\rco}[1]{Corollary \ref{C:#1}}

\newcommand{\rt}[1] {Theorem \ref{T:#1}}
\newcommand{\rtt}[1] {"Theorem" \ref{T:#1}}

\newcommand{\sm}{\smallsetminus}

\newcommand{\sgn}{\operatorname{sgn}}
\newcommand{\indlim}{\operatorname{indlim}}
\newcommand{\pr}{\operatorname{pr}}
\newcommand{\gr}{\operatorname{gr}}

\newcommand{\Pro}{\operatorname{Pro}}

\newcommand{\Spec}{\operatorname{Spec}}
\newcommand{\map}{\operatorname{map}}

\newcommand{\Mor}{\operatorname{Mor}}
\newcommand{\Ext}{\operatorname{Ext}}

\newcommand{\Ho}{\operatorname{Ho}}

\newcommand{\Ad}{\operatorname{Ad}}

\newcommand{\Gal}{\operatorname{Gal}}
\newcommand{\Tr}{\operatorname{Tr}}
\newcommand{\Av}{\operatorname{Av}}
\newcommand{\Sch}{\operatorname{Sch}}
\newcommand{\IndSch}{\operatorname{IndSch}}
\newcommand{\Par}{\operatorname{Par}}
\newcommand{\ev}{\operatorname{ev}}
\newcommand{\Fr}{\operatorname{Fr}}

\newcommand{\Lie}{\operatorname{Lie}}

\newcommand{\Id}{\operatorname{Id}}
\newcommand{\Fl}{\operatorname{Fl}}
\newcommand{\Spr}{\operatorname{Spr}}
\newcommand{\rk}{\operatorname{rk}}
\newcommand{\Rep}{\operatorname{Rep}}

\newcommand{\Fib}{\operatorname{Fib}}
\newcommand{\Tor}{\operatorname{Tor}}
\newcommand{\St}{\operatorname{St}}

\newcommand{\Fun}{\operatorname{Fun}}
\newcommand{\End}{\operatorname{End}}

\newcommand{\Dist}{\operatorname{Dist}}

\newcommand{\Ob}{\operatorname{Ob}}
\newcommand{\Alg}{\operatorname{Alg}}
\newcommand{\Set}{\operatorname{Set}}
\newcommand{\Ch}{\operatorname{Ch}}

\newcommand{\Hom}{\operatorname{Hom}}
\newcommand{\Cat}{\operatorname{Cat}}
\newcommand{\RHom}{\operatorname{RHom}}

\newcommand{\colim}{\operatorname{colim}}
\newcommand{\Irr}{\operatorname{Irr}}
\newcommand{\pt}{\operatorname{pt}}
\newcommand{\fq}{\B{F}_q}
\newcommand{\ql}{\B{Q}_l}

\newcommand{\qlbar}{\overline{\ql}}
\newcommand{\fqbar}{\overline{\fq}}
\newcommand{\wh}{\widehat}

\newcommand{\ASt}{\operatorname{ASt}}
\newcommand{\ASch}{\operatorname{ASch}}
\newcommand{\AISt}{\operatorname{AISt}}
\newcommand{\AISch}{\operatorname{AISch}}
\newcommand{\Corr}{\operatorname{Corr}}

\newcommand{\Sp}{\operatorname{Sp}^{ft}}
\newcommand{\Art}{\operatorname{Art}^{ft}}

\begin{document}

\title[A categorical approach to the stable
center conjecture]%
{A categorical approach to the stable center conjecture}

\author{Roman Bezrukavnikov}
\address{Department of Mathematics\\
Massachusetts Institute of Technology\\
77 Massachusetts Avenue\\
Cambridge, MA 02139, USA} \email{bezrukav@math.mit.edu}

\author{David Kazhdan}
\address{Institute of Mathematics\\
The Hebrew University of Jerusalem\\
Givat-Ram, Jerusalem,  91904\\
Israel} \email{kazhdan@math.huji.ac.il}

\author{Yakov Varshavsky}
\address{Institute of Mathematics\\
The Hebrew University of Jerusalem\\
Givat-Ram, Jerusalem,  91904\\
Israel} \email{vyakov@math.huji.ac.il}

\thanks{The project have received funding from ERC under grant agreement No 669655.
In addition, this research was supported by the BSF grant 2012365, R.B. was supported by the
NSF grant DMS-1102434, D.K. was supported by the ISF grant 1691/10, and Y.V. was supported by
the ISF grants 598/09 and 1017/13.}
\begin{abstract}
Let $G$ be a connected reductive group over a local
non-archimedean field $F$. The stable center conjecture provides
an intrinsic decomposition of the set of equivalence classes of smooth irreducible
representations of $G(F)$, which is only slightly coarser than the
conjectural decomposition into $L$-packets. In this work we
propose a way to verify this conjecture for depth zero
representations. As an illustration of our method, we show that
the Bernstein projector to the depth zero spectrum is stable.
\end{abstract}
\maketitle

\centerline{\em To G\'erard Laumon on his 60th birthday}
\tableofcontents

\section*{Introduction}

\noindent{\bf The stable center conjecture.}
Let $G$ be a connected reductive group over a local non-archimedean field
$F$, let $R(G)$ be the category of smooth complex representations of $G(F)$, and let
$Z_G$ be the Bernstein center of $G(F)$, which is by definition
the center of the category $R(G)$. Then $Z_G$ is a commutative algebra over $\B{C}$.

Every $z\in Z_G$ defines an invariant distribution $\nu_z$ on $G(F)$, and we denote by
$Z^{st}_G$ the set of all $z\in Z_G$ such that the distribution $\nu_z$ is stable.
The stable center conjecture asserts that $Z^{st}_G$ is a unital subalgebra of $Z_G$.

This conjecture is closely related to the local Langlands conjecture. Recall that the local Langlands conjecture asserts that
the set of equivalence classes of smooth irreducible representations $\Irr(G)$ of $G(F)$ decomposes as a disjoint union of so-called $L$-packets.
By definition, we have a natural homomorphism $z\mapsto f_z$ from $Z_G$ to the algebra of functions $\Fun(\Irr(G),\B{C})$. A more precise version of the stable center conjecture asserts that $Z^{st}_G$ consists of all $z\in Z_G$ such that the function $f_z$
is constant on each $L$-packet.

In other words, the local Langlands conjecture allows a more precise formulation of the stable center conjecture.
Conversely, if $Z^{st}_G\subset Z_G$ is known to be a subalgebra, then we can decompose $\Irr(G)$ by characters of $Z^{st}_G$,
and conjecturally this decomposition is only slightly coarser than the decomposition by $L$-packets.
Thus, the stable center conjecture can be thought both as
a supporting evidence and as a step in the proof of the local Langlands conjecture.

As follows from results of Bernstein and Moy--Prasad, the category $R(G)$ decomposes as a direct sum $R(G)=R(G)^0\oplus
R(G)^{>0}$, where $R(G)^0$ (resp.  $R(G)^{>0}$) consists of those representations $\pi$, all of whose irreducible subquotients
have depth zero (resp. positive depth).  Therefore the Bernstein center $Z_G$ decomposes as a direct sum of centers
$Z_G=Z_G^0\oplus Z_G^{>0}$. In particular, we have an embedding $Z_G^0\hra Z_G$, which identifies
the unit element of $Z_G^0$ with the projector to the
depth zero spectrum $z^0\in Z_G$. Set
$Z^{st,0}_G:=Z_G^0\cap Z^{st}_G$.

The depth zero stable center conjecture asserts that $Z^{st,0}_G\subset Z_G^0$ is a unital subalgebra. In particular,
it predicts the stability of the projector $z^0\in Z_G$.

The main goal of this work is to outline an approach to a proof of the depth zero stable center conjecture.
As an illustration of our method, we prove an explicit formula for the Bernstein projector $z^0$,
and deduce its stability. More precisely, we do it when $G$ is a split semisimple simply connected group, and $F$ is a local field of a positive but not very small characteristic.

\vskip 8truept\noindent{\bf Our approach.}
Our strategy is to construct explicitly many elements $z$ of $Z_G^{0}\subset Z_G$, whose span is a subalgebra, and to prove that these elements are stable and generate all of $Z_G^{st,0}$. Here by "explicitly", we mean to describe both the invariant distribution $\nu_z$ on $G(F)$ and the function $f_z$ on $\Irr(G)$.

To carry out our strategy, we construct first a categorical analog $\C{Z}(LG)$ of the Bernstein center $Z_G$. Then we observe that a version of the Grothendieck "sheaf-function correspondence" associates
to each Frobenius equivariant object $\C{F}\in \C{Z}(LG)$ an element of the Bernstein center
$[\C{F}]\in Z_G$. Thus, to construct elements of $Z_G$, it suffices to construct Frobenius-equivariant objects of $\C{Z}(LG)$.

In order to construct elements of $\C{Z}(LG)$, we construct first a categorical analog $\C{Z}_{\I^+}(LG)$ of $Z^0_G$ and a categorical analog $\C{A}:\C{Z}_{\I^+}(LG)\to\C{Z}(LG)$ of the embedding $Z^0_G\hra Z_G$. Then we apply $\C{A}$ to monodromic analogs of Gaitsgory central sheaves.

Roughly speaking, we define $\C{A}$ to be the composition
of the averaging functor $\Av_{\Fl}$, where $\Fl$ is the affine flag variety of $G$,
and the functor of "derived $\wt{W}$-skew-invariants", where $\wt{W}$ is the affine Weyl group of $G$.
This construction is motivated by an analogous finite-dimensional result, proven in \cite{BKV}.
However, in the affine case one has to overcome many technical difficulties.



\vskip 8truept\noindent{\bf Bernstein projector to the depth zero spectrum.}
To illustrate our method, we provide a geometric construction of the Bernstein projector
$z^0\in Z_G$. More precisely, we construct a class $\lan A\ran $ in the Grothendieck group version of $\C{Z}(LG)$ and
show that the corresponding element of $Z_G$ is $z^0$. Then we show that the restriction $\nu_{z^0}|_{G^{rss}(F)}$
is locally constant and prove an explicit  formula, which we now describe.

Let $I^+$ be the pro-unipotent radical of the Iwahori subgroup of $G(F)$, let $\mu^{I^+}$ be the Haar measure on $G(F)$
normalized by the condition that $\mu^{I^+}(I^+)=1$, and let $\phi_{z^0}\in C^{\infty}(G(F))$ be such that
$\nu_{z^0}|_{G^{rss}(F)}=\phi_{z^0}\mu^{I^+}$.

For each $\gm\in G^{rss}(F)$, we denote by $\Fl_{\gm}$ be the corresponding affine Springer fiber. The affine Weyl group
$\wt{W}$ acts on each homology group $H_i(\Fl_{\gm})=H^{-i}(\Fl_{\gm},\B{D}_{\Fl_{\gm}})$, where $\B{D}_{\Fl_{\gm}}$ is the dualizing sheaf.
Consider the $\Tor$-groups $\Tor^{\wt{W}}_j(H_i(\Fl_{\gm}),\sgn)$, where by $\sgn$ we denote the sign-character of $\wt{W}$. Each $\Tor^{\wt{W}}_j(H_i(\Fl_{\gm}),\sgn)$ is a finite-dimensional $\qlbar$-vector space, equipped with an action of the Frobenius element. One of the main results in this paper is the following identity
\begin{equation} \label{Eq:main}
\phi_{z^0}(\gm)=\sum_{i,j}(-1)^{i+j}
\Tr(\Fr,\Tor_j^{\wt{W}}(H_i(\Fl_{\gm}),\sgn)).
\end{equation}
Using formula \form{main} and a group version
of a theorem of Yun \cite{Yun}, we show that $\nu_{z^0}|_{G^{rss}(F)}$ is stable. Note that the proof of Yun is global,
while all the other arguments are purely local.

\vskip 8truept\noindent{\bf Remark.}
Though $\infty$-categories are not needed for the construction of $\lan A\ran $, we need them
in order to prove the formula \form{main}. Moreover, the structure of formula \form{main} indicates why
the $\infty$-categories appears here. The shape of the formula suggests a possibility to write the right hand side of \form{main} as
the trace of Frobenius on the "derived skew-coinvariants" $R\Gm(\Fl,\B{D}_{\Fl_{\gm}})_{\wt{W},sgn}$. However, the functor of
"derived skew-coinvariants" is defined as a homotopy colimit, and it can not be defined in the framework of derived categories.
Therefore one has to pass to stable $\infty$-categories.

\vskip 8truept\noindent{\bf Plan of the paper.}
In Sections 1  we study derived categories
of constructible sheaves on a certain class of ind-schemes and ind-stacks, which we call admissible.
This class includes some infinite-dimensional ind-stacks, which are not algebraic. We also construct a certain
geometric $2$-category, whose $\infty$-version is used later.

In Section 2, we apply the formalism of Section 1 to the case of loop groups $LG$ and related spaces
in order to construct a categorical analog of the Hecke algebra.

Section 3 deals with the stable center conjecture. Namely, we formulate and discuss the stable center conjecture in subsection 3.1,
categorify various objects from subsection 3.1 in subsections 3.2-3.3, and describe our (conjectural) approach
to the depth zero stable center conjecture in subsections 3.4-3.5.

The results in subsections 3.2-3.3 are given without complete proofs, and  details will appear in the forthcoming paper
\cite{BKV2}. To emphasise  this fact, we write {\em "Theorem"} instead of {\em Theorem}.

In Section 4 we implement the strategy of Section 3 in the case of the unit element. Namely,
in subsections 4.1-4.3, we construct a $K$-group analog of the projector to the depth zero spectrum, following the strategy of subsection 3.3.
Then, in subsection 4.4, we provide a geometric construction of
$z^0\in Z_G$, formulate the formula \form{main} for $\nu_{z^0}|_{G^{rss}(F)}$
and deduce its stability.

Finally, in Section 5 we prove our formula \form{main}.

\vskip 8truept\noindent{\bf Conventions on categories.}
(a) Our approach is based on "categorification". Since our constructions
involve homotopy limits, the derived categories are not suitable for our purposes.
Instead, we use the language of stable $\infty$-categories (see
\cite{Lur2}). However, to make the exposition accessible to a wider mathematical audience,
we use them as little as possible.

Namely, we use $\infty$-categories only in two places. The first place we need $\infty$-categories is in subsections 3.2-3.3,
where we categorify various objects from subsection 3.1. On the other hand, since the results in
3.2-3.3 are given without complete proofs, here we use $\infty$-categories mostly
as a "black box".

The second use is in the proof of formula \form{main} in subsections 5.2-5.3.
The main part of the arguments (the exception is the proof of \rcl{functor}, carried out in subsection 5.3)
uses only very basic properties of $\infty$-categories, mainly the notion of homotopy colimits.
But even in our proof of \rcl{functor}, we only use the notion of $\infty$-categories, and avoid the usage
of more complicated notions like monoidal $\infty$-categories, or $(\infty,2)$-categories.

(b) For every scheme or algebraic stack $X$ of
finite type over an algebraically closed field, we use the
existence of a stable $\infty$-category $\C{D}(X)$, called the derived $\infty$-category,
whose homotopy category is the bounded derived category of constructible sheaves $D(X)=D^b_c(X,\qlbar)$.
We also use the fact that the six functor formalism exists in this setting (see, for example, \cite{LZ1,LZ2}).

(c) We say that two objects of an $\infty$-category are {\em
equivalent} and write $A\cong B$, if they become isomorphic in the
homotopy category. In particular, by writing
$\C{F}\cong\C{G}$ for two objects of the $\infty$-category
$\C{D}(X)$, we indicate that they are isomorphic in the derived
category $D(X)$.

(d) Contrary to the common use, in this work we do not assume that monoidal categories
have units. Moreover, even when the monoidal categories do have units, we do not assume
that the monoidal functors preserve them.

(e) For a category (resp. $\infty$-category) $C$, we denote its pro-category by $\Pro C$.

(f) A category $C$ gives rise to an $\infty$-category, which we also denote by $C$.


\vskip 8truept\noindent{\bf Related works.} Our work was influenced by a
paper of Vogan on the local Langlands conjectures (\cite{Vo}).
In the process of writing this paper we have learned that
versions of the stable Bernstein center and the stable center conjecture
were also considered by Haines \cite{Ha} and Scholze--Shin
\cite{SS}.

Another approach to harmonic analysis on $p$-adic groups via $l$-adic sheaves
was proposed in recent papers by Lusztig \cite{Lu4}, \cite{Lu5}. It is different
from ours: for example, {\em loc. cit.} deals with characters of irreducible
representations rather than elements of Bernstein center. However, we expect
the two constructions to be related.

In the finite-dimensional setting related questions for $D$-modules were also studied by
Ben-Zvi--Nadler \cite{BN} and Bezrukavnikov--Finkelberg--Ostrik \cite{BFO}.

\vskip 8truept\noindent{\bf Acknowledgements.}
We thank Nick Rozenblyum for numerous stimulating conversations about $\infty$-categories.
In particular, he explained to us how to define the "correct"
notion of the categorical center. We also thank Michael Temkin for discussions about non-Noetherian schemes and stacks, and the referee for his numerous remarks
about the first version of the paper.


The final revision of the paper was done while the third author visited the MSRI. He thanks this institution for excellent atmosphere
and financial support.


\section{Constructible sheaves on admissible ind-schemes and ind-stacks}

\subsection{Admissible morphisms}
Let $\Sch_k$ be the category of quasi-compact and quasi-separated schemes over $k$. We usually write $\lim$ instead of $\projlim$ and $\colim$ instead of $\indlim$.


\begin{Lem} \label{L:lim}
Let $\{X_i\}_{i}$ be a projective system in  $\Sch_k$, indexed by a filtered partially ordered set $I$, such that the transition maps $X_{i}\to X_j, i>j$ are affine.

(a) Then there exists a projective limit $X\cong\lim_ i X_i$ in $\Sch_k$.

(b) Assume that we are given $i'\in I$, a morphism $Y_{i'}\to X_{i'}$ and a finitely presented morphism $Z_{i'}\to X_{i'}$ in $\Sch_k$. Set $Y_i:=X_i\times_{X_{i'}}Y_{i'}$ and $Z_i:=X_i\times_{X_{i'}}Z_{i'}$ for $i\geq i'$, and also set
$Y:=X\times_{X_{i'}}Y_{i'}$ and $Z:=X\times_{X_{i'}} Z_{i'}$. Then the natural map $\colim_{i\geq i'}\Hom_{X_i}(Y_i,Z_i)\to\Hom_X(Y,Z)$ is a bijection.

(c) For every $i'\in I$ and a finite presented morphism $Y\to X_{i'}$ in $\Sch_k$, the map \\
$\colim_{i\geq i'}\Hom_{X_{i'}}(X_i,Y)\to\Hom_{X_{i'}}(X,Y)$ is a bijection.

(d) For every finitely presented morphism $f:Y\to X$ in $\Sch_k$ there exists
$i$ and a finitely presented morphism  $f':Y'\to X_i$ such that
$f$ is isomorphic to the morphism $f'\times_{X_i}X:Y'\times_{X_i} X\to X$, induced by $f'$.

(e) The morphism $f'$ in (d) is essentially unique, that is,
if $f'':Y''\to X_j$ is another morphism such that $f$ is isomorphic to $Y''\times_{X_j} X\to X$, induced by $f''$,
then there exists $r\geq i,j$ such that $f'\times_{X_i} X_r\cong f''\times_{X_j} X_r$.

(f) In the notation of (d) assume that the map $f$ is smooth (resp. closed embedding, resp. affine). Then
there exists $j\geq i$ such that $Y_i\times_{X_i} X_j\to X_j$ has the same property.
\end{Lem}
\begin{proof}
This is standard. Namely, (a) is easy; (b) and (d) are proven in \cite[Thm. 8.8.2]{EGA}; (c) and (e) follow immediately from  (b).
Finally, the assertion (f) for smooth morphisms  is proven in \cite[Prop. 17.7.8]{EGA}, while the assertion for closed embeddings
and affine morphisms is proven in \cite[Thm. 8.10.5]{EGA}.
\end{proof}

\begin{Emp} \label{E:unip}
{\bf Unipotent morphisms.}
(a)  Let $\Var_k$ be a category of separated schemes of finite type over $k$. Fix a prime number $l$, different from the characteristic of $k$.
Then for every $X\in\Var_k$ one can consider its bounded derived category of constructible $\qlbar$-sheaves $D^b_c(X,\qlbar)$, which we denote by
$D^b_c(X)$ or simply by $D(X)$. Every morphism $f:X\to Y$ in $\Var_k$ induce functors $f^*,f^!,f_*,f_!$.

(b) Let $k$ be separably closed. We say that $X\in \Var_k$ is {\em acyclic},
if the canonical map $\ql\to R\Gm(X,\ql)$ is an isomorphism. In particular, $X=\B{A}^n$ is acyclic.

(c) We call a finitely presented morphism $f:X\to Y$ in $\Sch_k$ {\em unipotent}, if it is smooth, and all geometric fibers of $f$ are acyclic. Notice that the assertion \rl{lim} (f) for smooth morphisms immediately implies the assertion for unipotent ones.
\end{Emp}

\begin{Lem} \label{L:ff}
Let $f:X\to Y$ be a smooth morphism in $\Var_k$. Then $f$ is unipotent if and only if the counit map $f_!f^!\to\Id$ is an isomorphism and if and only if the functor $f^!:D(Y)\to D(X)$ is fully faithful. In this case, $\C{F}\in D(X)$ belongs to the essential image of $f^!$ if and only if $\C{F}\cong f^!f_!\C{F}$.
\end{Lem}

\begin{proof}
Since $f_!$ is a left adjoint of $f^!$, we conclude that $f^!:D(Y)\to D(X)$ is fully faithful if and only if
the counit map $f_!f^!\C{F}\to\C{F}$ is an isomorphism for every $\C{F}\in D(Y)$. If $f$ is smooth of relative dimension
$n$, then $f^!\cong f^*[ 2n](n)$, hence it follows from the projection formula that
\[
f_!f^!\C{F}\cong f_!f^*\C{F}[ 2n](n)\cong f_!(f^*\C{F}\otimes f^!\qlbar)\cong \C{F}\otimes f_!f^!\qlbar.
\]

Thus, $f^!$ is fully faithful if and only if the counit map $f_!f^!\qlbar\to\qlbar$ is an isomorphism. Using isomorphism $f^!\cong f^*[ 2n](n)$ once more  together with the base change isomorphism, we see that this happens if and only if for each geometric fiber $X_{\ov{y}}$ of $f$ the counit map $R\Gm_c(X_{\ov{y}},\B{D}_{X_{\ov{y}}})\to\qlbar$ is an isomorphism. Hence, by the Verdier duality, this happens if and only if each $X_{\ov{y}}$ is acyclic, that is, $f$ is unipotent.

Finally, if $f^!$ is fully faithful and $\C{F}\cong f^!\C{G}$, then $f^!f_!\C{F}\cong f^!f_!(f^!\C{G})\cong f^!(f_!f^!\C{G})\cong f^!\C{G}\cong \C{F}$, while the opposite assertion is obvious.
\end{proof}

\begin{Emp} \label{E:admmor}
{\bf Admissible morphisms.}
(a) We call a morphism  $f:X\to Y$ in $\Sch_k$ {\em admissible}, if there exists a projective system
$\{X_i\}_{i\in I}$ over $Y$, indexed by a filtered partially ordered set $I$ such that each $X_i\to Y$ is finitely presented,
all the transition maps $X_{i}\to X_j,i>j$ are affine unipotent, and $X\cong\lim_i X_i$.
We call an isomorphism $X\cong \lim_i X_i$ an {\em admissible presentation} of $f$.

(b) We call a morphism $f:X\to Y$ in $\Sch_k$ {\em pro-unipotent}, if it has an admissible presentation $X\cong \lim_i X_i$ such that
$X_i\to Y$ is unipotent for some $i$ (or, equivalently, for all sufficiently large $i$). Observe that for an admissible morphism $f:X\to Y$ with an admissible presentation $X\cong \lim_i X_i$, each projection $X\to X_i$ is pro-unipotent.

(c) We call $X\in \Sch_k$ {\em admissible}, if the map $X\to \Spec k$ is admissible.
In other words, $X$ is admissible if and only if there exists a projective system
$\{X_i\}_{i\in I}$ in $\Var_k$ such that all the transition maps $X_{i}\to X_j$ are affine unipotent, and $X\cong\lim_i X_i$.
In this case we say that $X\cong \lim_i X_i$ is an {\em admissible presentation} of $X$.
We denote the category of admissible schemes by $\ASch_k$.
\end{Emp}

\begin{Lem} \label{L:admmor}
The composition of admissible morphisms is admissible.
\end{Lem}

\begin{proof}
Let $f:X\to Y$ and $g:Y\to Z$ be two admissible morphisms with admissible presentations $X\cong\lim_{i\in I} X_i$ and $Y\cong\lim_{j\in J} Y_j$,  respectively. We are going to construct an admissible presentation of $h=g\circ f:X\to Z$.
To avoid discussion about ordinals, we only consider the case $I=J=\B{N}$, which is sufficient for this work.

Since $f_0:X_0\to Y$ is of finite presentation, there exists $n_0\in\B{N}$ and a morphism of finite presentation
$f'_0:X'_0\to Y_{n_0}$ such that $f_0\cong f'_0\times_{Y_{n_0}}Y$ (by \rl{lim} (d)). In particular,
$X_0\cong\lim_{i\geq n_0}(X'_0\times_{Y_{n_0}}Y_i)$. Next since $f_1:X_1\to X_0$ is unipotent and affine, there exists $n_1>n_0$ and a unipotent affine
morphism $f'_1:X'_1\to X'_0\times_{Y_{n_0}}Y_{n_1}$ such that $f_1\cong f'_1\times_{Y_{n_1}}Y$
(by \rl{lim} (f) and \re{unip} (c)). Continuing this process, we construct  an increasing sequence $\{n_i\}_i$ and a sequence of unipotent affine morphisms $f'_i:X'_{i}\to X'_{i-1}\times_{Y_{n_{i-1}}} Y_{n_{i}}$.

Denote by $h_i$ the composition $X'_i\overset{f'_i}{\lra} X'_{i-1}\times_{Y_{n_{i-1}}} Y_{n_{i}}\to X'_{i-1}$. By construction, $h_i$ is  unipotent and affine, while the morphism $X'_i\to X'_0\to Y_{n_0}\to Z$ is of finite presentation. So it remains to show that $X\cong\lim X'_i$.
By assumption, we have $X\cong \lim_i X_i\cong \lim_i (X'_i\times_{Y_{n_i}}Y)$.
Using isomorphism $\lim_i Y_{n_i}\cong Y$,  we conclude that $\lim_i X'_i\cong \lim_i (X'_i\times_{Y_{n_i}}Y)\cong X$.
\end{proof}

\subsection{Constructible sheaves on admissible schemes}

\begin{Emp} \label{E:cat}
{\bf Notation.}
(a) Every $X\in \Sch_k$ defines an under-category $X/\cdot:=X/\Var_k$, whose objects are morphisms $X\to V$ in $\Sch_k$ with $V\in \Var_k$.
To simplify the notation, we sometimes write $V\in X/\cdot$ instead of $(X\to V)\in X/\cdot$.

Notice that since the category $\Var_k$ has fiber products, the
category $X/\cdot$ is co-filtered, that  is, the opposite category $(X/\cdot)^{op}$ is filtered.

(b) Denote by $(X/\cdot)^{sm}\subset X/\cdot$ (resp. $(X/\cdot)^{un}\subset X/\cdot$) the full subcategory of $X/\cdot$ consisting of formally smooth (resp. pro-unipotent) morphisms $X\to V$. For every $V\in X/\cdot$, we denote by $(X/\cdot)/V$ the over-category over $V$.
\end{Emp}

\begin{Emp} \label{E:csh}
{\bf Constructible $\qlbar$-sheaves.}
Recall that  the assignment $V\mapsto D(V)$ for $V\in\Var_k$ gives rise to four functors $f^!,f^*,f_!,f_*$.

(a) For every $X\in \Sch_k$, these functors give rise to four versions of "constructible $\qlbar$-sheaves on $X$", which we denote by $M(X), D(X), \wh{M}(X)$ and $\wh{D}(X)$, respectively.  Namely, we define

\textbullet $\;M(X)$ to be the (homotopy) colimit $\colim^!_{(X\to V)\in (X/\cdot)^{op}} D(V)$, taken  with respect to $!$-pullbacks;

\textbullet $\;D(X)$ to be the colimit $\colim^*_{(X\to V)\in (X/\cdot)^{op}} D(V)$, taken with respect to $*$-pullbacks;

\textbullet $\;\wh{M}(X)$ to be the limit $\lim^!_{(X\to V)\in (X/\cdot)} D(V)$, taken with respect to $!$-pushforwards;

\textbullet $\;\wh{D}(X)$ to be the limit $\lim^*_{(X\to V)\in (X/\cdot)} D(V)$, taken with respect to $*$-pushforwards.

(b) Explicitly, the class of objects of $M(X)$  is a union of classes of objects of $D(V)$, taken over $(X\to V)\in X/\cdot$. Next, for every
$\C{F}\in \Ob D(V)\subset \Ob M(X)$ and $\C{F}'\in \Ob D(V')\subset \Ob M(X)$ the set of morphisms $\Hom_{M(X)}(\C{F},\C{F}')$ is the inductive limit
$\colim_{h,h'}\Hom_{D(U)}(h^!\C{F}, h'^!\C{F}')$, taken over all diagrams $V\overset{h}{\lla}U\overset{h'}{\lra}V'$ in $X/\cdot$. The description of $D(X)$ is similar.

(c) Likewise, objects of $\wh{M}(X)$ are collections $\{\C{F}_{X\to V},\phi_{V,V'}\}$ compatible with compositions,
where $\C{F}_{X\to V}$ is an object of $D(V)$ for each $(X\to V)\in X/\cdot$, and $\phi_{V,V'}$ is an isomorphism $f_!\C{F}_{X\to V}\isom\C{F}_{X\to V'}$ for each morphism $f:V\to V'$ in $X/\cdot$.

(d) We denote by $\B{D}_X\in M(X)$ and $1_X\in D(X)$ the images of the constant sheaf $\qlbar\in D(k):=D(\Spec k)$.

(e) Every morphism $f:X\to Y$ in $\Sch_k$ induces a functor $f_{\cdot}:Y/\cdot\to X/\cdot$, hence induce
functors $f^!:M(Y)\to M(X)$, $f^*:D(Y)\to D(X)$, $\wh{f}_!:\wh{M}(X)\to \wh{M}(Y)$ and  $\wh{f}_*:\wh{D}(X)\to\wh{D}(Y)$.
\end{Emp}

\begin{Emp} \label{E:adm}
{\bf The admissible case.}
Fix $X\in \ASch_k$ with an admissible presentation $X\cong\lim X_i$.

(a) By \rl{lim} (c), the projections
$\{X\to X_i\}_i$ form a co-cofinal system in $X/\cdot$, that is, a cofinal system in $(X/\cdot)^{op}$.
In particular, we have natural equivalences $M(X)\cong \colim^!_i D(X_i)$, $\wh{M}(X)\cong \lim^!_{i} D(X_i)$, and similarly for $D(X)$ and $\wh{D}(X)$. Since $\pi_{i,j}:X_{i}\to X_j$ is unipotent for all $i>j$, it follows from  \rl{ff} that the functors $\pi_{i,j}^!:D(X_j)\to D(X_{i})$ are fully faithful.  Thus the functors $D(X_i)\to M(X)$ are fully faithful as well.

(b) Since each $X\to X_i$ is pro-unipotent, we conclude from \rl{lim} (c) that  the subcategory $(X/\cdot)^{un}\subset X/\cdot$ is co-cofinal and co-filtered. As in (a), for every $V\in (X/\cdot)^{un}$, the induced functor $D(V)\to M(X)$ is fully faithful.

(c) Every morphism $f:V'\to V$ in $(X/\cdot)^{un}$ is unipotent.
Indeed, $f$ is formally smooth, because both $X\to V'$ and $X\to V'\to V$ are formally smooth and surjective.
Next, $f$ is smooth, because $f$ is formally smooth of finite type. Finally, the functor $f^!:D(V)\to D(V')$ is fully faithful,
because both $D(V')\to M(X)$ and $D(V)\to D(V')\to M(X)$ are such (by (b)). Thus $f$ is unipotent by \rl{ff}.

(d) We claim that we have a canonical fully faithful functor $M(X)\hra \wh{M}(X)$ (and similarly, $D(X)\hra\wh{D}(X)$).
Since $M(X)=\colim^!_{V} D(V)$, taken over $V\in (X/\cdot)^{un}$, it is enough to construct a system
of fully faithful embeddings $D(V)\hra \wh{M}(X)$, compatible with $!$-pullbacks.

Fix $V\in (X/\cdot)^{un}$. By the remark of (b), we have $\wh{M}(X)=\lim^!_{V'} D(V')$, taken over $V'\in (X/\cdot)^{un}/V$. Thus it suffices
to construct a system of fully faithful embeddings $D(V)\hra D(V')$, compatible with $!$-pushforwards. To each morphism $f:V'\to V$ in $(X/\cdot)^{un}$,
we associate the functor $f^!:D(V)\to D(V')$, which is fully faithful by (c) and \rl{ff}.
To show the compatibility with $!$-pushforwards,  we have to show that for every morphism $g:V''\to V'$ in $(X/\cdot)^{un}$, the natural morphism $g_!(f\circ g)^!=g_!g^!f^!\to f^!$ is an isomorphism. But this follows from the fact that $g$ is unipotent (see (c) and
\rl{ff}).

(e) By (d) and \re{csh} (e), for every morphism $f:X\to Y$ between admissible schemes,
we have a functor $f_!: M(X)\hra \wh{M}(X)\overset{\wh{f}_!}{\lra}\wh{M}(Y)$. Moreover, if $f$ is admissible then, by \rl{pf} below, $f$ induces a functor $f_!:M(X)\to M(Y)$.
\end{Emp}

\begin{Lem} \label{L:pf}
If $f:X\to Y$ is an admissible morphism between admissible schemes, then
the image $f_!(M(X))\subset \wh{M}(Y)$ lies in $M(Y)\subset \wh{M}(Y)$.
\end{Lem}

\begin{proof}
Let $X\cong\lim_i X_i$ be an admissible presentation of $f$, and let $f_i:X_i\to Y$ be the induced morphism. Then $M(X)=\colim_i M(X_i)$,
and it suffices to show that $(f_i)_!(M(X_i))\subset M(Y)\subset \wh{M}(Y)$. Thus we may assume that $f$ is finitely presented.

Let $Y\cong\lim_j Y_j$ be an admissible presentation of $Y$. Then, by \rl{lim} (d), there exists $j\in J$ and a morphism
$f_j:X_j\to Y_j$ in $\Var_k$ such that $f\cong f_j\times_{Y_j}Y$. For every $i\geq j$, we set $X_i:=X_j\times_{Y_j}Y_i$ and
$f_i:=f_j\times_{Y_j}Y_i:X_i\to Y_i$.

Then  $X\cong\lim_{i\geq j} X_i$ is an admissible presentation of $X$, and it follows
from the smooth base change that for every $\C{F}\in D(X_i)\subset M(X)$, we have a natural isomorphism
$f_!(\C{F})\cong (f_i)_!(\C{F})\in D(Y_i)\subset M(Y)$.
\end{proof}

\begin{Emp}
{\bf Remark.} Categories $D(X), M(X)$ and $\wh{M}(X)$ are the categorical analogs of the spaces of
locally constant functions, locally constant measures, and all measures, respectively.
This analogy works perfectly well for admissible schemes, since in this case we have an embedding
$M(X)\hra \wh{M}(X)$. Also $f^*$ is analog of the pullback of functions, while
$\wh{f}_!$ is an analog of the pushforward of measures (compare \re{shfun1} below).
\end{Emp}

\begin{Emp}
{\bf Morphisms of finite presentation.}
Let $f:X\to Y$ be a morphism in $\ASch_k$ of finite presentation. In this case, we have four additional  functors $f^*:M(Y)\to M(X)$, $f_*:M(X)\to M(Y)$, $f^!: D(Y)\to D(X)$ and $f_!:D(X)\to D(Y)$, constructed as follows.

Since $f$ is of finite presentation, there exist
$U\in (X/\cdot)^{un}$, $V\in (Y/\cdot)^{un}$ and a morphism $g:U\to V$ in $\Var_k$ such that
$f\cong g\times_{V}Y$ (by \rl{lim} (d)). For every $V'\in (Y/\cdot)^{un}/V$ we define $g^*_{V'}:D(V')\to D(U\times_{V}V')$ to be the $*$-pullback.
Since all morphisms in  $(Y/\cdot)^{un}/V$ are smooth (see \re{adm} (c)), functors $g^*_{V'}$ give rise to the functor
$g^*:M(Y)=\colim^!_{V'/V} D(V')\to \colim^!_{V'/V} D(U\times_{V}V')=M(X)$.

It remains to show that the above functor only depends on $f$ rather than $g$. In other words, we claim that
there exists a natural isomorphism of functors $g^*\cong g'^*$ for every other morphism $g':U'\to V'$ such that $f\cong g'\times_{V'}Y$.
This is clear when $g'$ is a base change of $g$, therefore the assertion follows from
the fact that  $g$ and $g'$ have isomorphic base changes (by \rl{lim} (e)).

The construction of the other three functors is similar.
\end{Emp}

A notion of a Haar measure in the case of profinite groups can be generalized to profinite sets of the form
$S=\lim S_i$ such that all fibers of all projections $S_j\to S_i$ have the same cardinality. Below we define a geometric analog
of this notion.

\begin{Emp} \label{E:haar}
{\bf (Generalized) Haar measures}. Let $X,Y\in\ASch_k$. (a) For every $V\in (X/\cdot)^{sm}$, we denote by $\mu^{X\to V}\in M(X)$ the image of  $1_V\in D(V)$ (see \re{csh} (d)).  By a {\em Haar measure of $X$}, we mean any object of $M(X)$ of the form $\mu^{X\to V}[2n](n)$, where $V\in (X/\cdot)^{sm}$ and $n\in\B{Z}$. In particular, a Haar measure always exists.

(b) If $X$ is connected, then a Haar measure of $X$ is unique up to a transformation  $\C{F}\mapsto\C{F}[2n](n)$. Indeed, we claim that
for every $V,V'\in (X/\cdot)^{sm}$, we have an isomorphism
$\mu^{X\to V'}[2\dim V'](\dim V')\cong\mu^{X\to V}[2\dim V](\dim V)$. Since $X/\cdot$ is co-filtered, we can assume that $V'\in (X/\cdot)/V$.
Then $f:V'\to V$ is smooth of relative dimension $\dim V'-\dim V$, and the assertion follows from the isomorphism
\[
f^!(1_V)\cong 1_{V'}[2(\dim V'- \dim V)](\dim V'-\dim V).
\]

(c) If $f:X\to Y$ is finitely presented, then for every Haar measure $\C{F}\in M(Y)$
its pullback $f^*(\C{F})$ is a Haar measure on $X$. Indeed, $f$ is obtained as a base change of a certain morphism $g:U\to V$ in $\Var_k$ (by \rl{lim} (d)), so the assertion follows from the fact that $g^*(1_{V})\cong 1_{U}$.
Moreover, if $X$ and $Y$ are connected, then it follows from (b) that for every  Haar measure $\C{F}\in M(X)$ there exists a unique Haar measure $\C{F}'\in M(Y)$ such that $\C{F}\cong f^*(\C{F}')$.

(d) If $f:X\to Y$ is formally smooth, then for every Haar measure $\C{F}\in M(Y)$ on $Y$,
its pullback $f^!(\C{F})$ is a Haar measure on $X$. Indeed, for every $V\in (Y/\cdot)^{sm}$, we have
$V\in  (X/\cdot)^{sm}$ and $f^!(\mu^{Y\to V})=\mu^{X\to V}$.
\end{Emp}

\begin{Emp} \label{E:tensor}
{\bf Tensor Product.}  Let $X,Y\in \ASch_k$.

(a) We have a natural functor $\otimes: D(X)\times M(X)\to M(X)$. Namely, for every $V\in (X/\cdot)^{sm}$ we have a natural functor
\[\otimes_V:D(V)\times D(V)\to D(V)\hra M(X):(\C{A},\C{F})\mapsto \C{A}\otimes\C{F}.\]
Moreover, for every
$V'\in (X/\cdot)^{sm}/V$, the projection  $f:V'\to V$ is smooth, hence we have a natural isomorphism
$f^*\C{A}\otimes f^!\C{F}\cong f^!(\C{A}\otimes\C{F})$. Thus the functors $\{\otimes_V\}_{V\in (X/\cdot)}$ give rise to the functor  $\otimes: D(X)\times M(X)\to M(X)$.

(b) Let $f:X\to Y$ be a formally smooth morphism. Then for every $\C{A}\in D(Y)$ and $\C{F}\in M(Y)$
we have a natural isomorphism $f^*\C{A}\otimes f^!\C{F}\cong f^!(\C{A}\otimes\C{F})$, which follows from the corresponding isomorphism for $\Var_k$.

(c) Let $f:X\to Y$ be of finite presentation, thus functors $f^*:M(Y)\to M(X)$ and
$f_!:D(X)\to D(Y)$ are defined. Then we have natural
isomorphisms $f^*(\C{A}\otimes\C{F})\cong f^*\C{A}\otimes f^*\C{F}$ and
$f_!(\C{A}\otimes f^*\C{F})\cong f_!\C{A}\otimes\C{F}$,  which follow from the corresponding
isomorphisms for $\Var_k$.

(d) By (a), each $\C{F}\in M(X)$ defines a functor $\cdot\otimes\C{F}:D(X)\to M(X)$. Moreover, since each functor $\cdot\otimes 1_V[2n](n):D(V)\to  D(V)$ is an isomorphism, the functor $\cdot\otimes\C{F}$ is an isomorphism, if $\C{F}$ is a Haar measure.
\end{Emp}

\subsection{The case of ind-schemes}

\begin{Emp} \label{E:admind}
{\bf Admissible ind-schemes.} (a) We say that a functor $X:\Sch^{op}_k\to \Set$ is an {\em ind-scheme} over $k$ and write $X\in \IndSch_k$, if  there exists
an inductive system $\{X_i\}_{i}\in \Sch_k$ such that all the transition maps $X_i\to X_{j}$ are
finitely presented closed embeddings, and $X\cong\colim_i X_i$, that is,
$X(\cdot)\cong \colim_i\Hom_{\Sch_k}(\cdot, X_i)$. In this case, we will say that  $X\cong\colim_i X_i$ is a presentation of $X$.

(b) Let $Y\in \IndSch_k$ and $X\in \Sch_k$. We say that a morphism $f:X\to Y$ is {\em admissible} (resp. {\em finitely presented}), if there exists a presentation $Y\cong\colim_i Y_i$ such that $f$ is induced by an admissible
(resp. finitely presented) morphism $f:X\to Y_i$. Notice that this notion is independent of the presentation of $Y$.

(c) Let $Y\in \IndSch_k$ and $X\in \Sch_k$. We say that $X\subset Y$ is an {\em fp-closed subscheme}, if the inclusion  $X\hra Y$ is a finitely presented closed embedding. By a {\em closed fp-neighborhood of $y\in Y$}, we mean an fp-closed subscheme $X\subset Y$, containing $y$.

(d) A morphism  $f:X\to Y$ in $\IndSch_k$ is called {\em admissible} (resp. {\em finitely presented}), if
for every fp-closed subscheme $Z\subset X$ the restriction $f|_Z:Z\to Y$ is admissible (resp. finitely presented).
In particular, we say that $X\in \IndSch_k$ is admissible and write $X\in\AISch_k$, if the structure map $X\to \Spec k$ is admissible.

(e) Notice that $X\in \IndSch_k$ with presentation $X\cong\colim_i X_i$ is admissible if and only if each $X_i\in\Sch_k$ is admissible. Indeed, the "only if" assertion follows from definition. Conversely, if each $X_i$ is admissible, then every fp-closed
subscheme $Y$ of $X$ is an fp-closed subscheme of some $X_i$, thus $Y$ is admissible by \rl{admmor}.

(f) A morphism between ind-schemes $f:X\to Y$ is called {\em schematic} (resp. and {\em formally smooth}), if for every fp-closed subscheme $Z\subset Y$ the pre-image $f^{-1}(Z)\subset X$ is a scheme (resp. and the induced morphism
$f^{-1}(Z)\to Z$ is formally smooth).
\end{Emp}

\begin{Emp} \label{E:consind}
{\bf Constructible sheaves.} Let $X\in \AISch_k$.

(a) We denote by $M(X)$ (resp. $D(X)$) the inductive limit $\colim_{Y} M(Y)$
(resp. $\colim_Y D(Y)$), where $Y$ runs over the set of fp-closed subschemes of $X$, and the limit is taken with respect to
fully faithful functors $i_*:M(Y)\to M(Y')$ (resp. $i_*:D(Y)\to D(Y')$), corresponding to fp-closed embeddings $i:Y\to Y'$.
In particular, $M(X)\cong \colim_i M(X_i)$ for each presentation $X\cong\colim_i X_i$ (and similarly, for $D(X)$).


(b) For every fp-closed subscheme $Y\subset X$, we denote by $\dt_Y\in M(X)$ the extension by zero of $\B{D}_Y\in M(Y)$ (see \re{csh} (d)). If $Y'\subset Y$ is an fp-closed subscheme, then we have a natural morphism $\dt_{Y'}\to\dt_Y$, induced by the counit morphism $i_!i^!\to\Id$.

(c) We set $\wt{M}(X):=\lim^*_Y M(Y)$ (resp. $\wt{D}(X):=\lim^*_Y D(Y)$), where $Y$ is as in (a), and the transition maps are $*$-pullbacks.
Arguing as in \re{adm} (e) and using the fact that for every fp-closed embedding $i:Y\to Y'$ of schemes the counit morphism $i^*i_*\to \Id$ is an isomorphism, we conclude that  the $*$-pullbacks induce a fully faithful embedding $M(X)\hra\wt{M}(X)$ (resp. $D(X)\hra \wt{D}(X)$).

This embedding is a categorification of the embedding of the space of smooth measures (resp. functions) with compact support into the space of all smooth measures (resp. functions).

(d) For every schematic morphism $f:X\to Y$ in $\AISch_k$, we have pullback functors $f^!:M(Y)\to M(X)$ and $f^*:D(Y)\to D(X)$, while for every admissible morphism $f:X\to Y$ in $\AISch_k$, we have push-forward functors
$f_!:M(X)\to M(Y)$ and $f_*:D(X)\to D(Y)$, whose constructions formally follow from the corresponding functors for schemes.

(e)  For a finitely presented morphism $f:X\to Y$ in $\AISch_k$, we also have functors $f_!:D(X)\to D(Y)$ and
$f^*:\wt{M}(Y)\to\wt{M}(X)$. If, in addition, $f$ is schematic, we also have a functor  $f^*:M(Y)\to M(X)$.
Finally, for a formally smooth schematic morphism $f:X\to Y$ we have a functor $f^!:\wt{M}(Y)\to\wt{M}(X)$.

(f) We denote by $1_X\in \wt{D}(X)$ the projective system $\{1_Y\}_Y$, where $Y$ is as in (c).
\end{Emp}

\begin{Emp} \label{E:haarind}
{\bf Haar measures.} Let $X\in \AISch_k$.

(a) Note that we have a natural functor $\otimes:\wt{D}(X)\otimes \wt{M}(X)\to \wt{M}(X)$,
which sends $\C{A}=\{\C{A}_Y\}_Y\in \wt{D}(X)$ and $\C{F}=\{\C{F}_{Y}\}_{Y}\in \wt{M}(X)$ to
$\{\C{A}_Y\otimes\C{F}_Y\}_Y\in \wt{M}(X)$. Moreover, this functor restricts to functors
$\otimes:D(X)\otimes \wt{M}(X)\to M(X)$ and $\otimes:\wt{D}(X)\otimes M(X)\to M(X)$.

(b) By a {\em Haar measure} on $X$, we mean an element $\C{F}\in\wt{M}(X)\in \lim^*_Y M(Y)$ such that
$\C{F}_Y\in M(Y)$ is a Haar measure for each $Y$. Note that a Haar measure on $X$ always exists (by \re{haar} (c)).
Moreover,  if $X$ is connected, then a Haar measure on $X$ is unique up to a change $\C{F}\mapsto \C{F}[2n](n)$.
Also, if $\C{F}$ is a Haar measure on $X$,
then $\cdot\otimes\C{F}_Y:D(Y)\to M(Y)$ is an equivalence for all $Y$, thus  $\cdot\otimes\C{F}:D(X)\to M(X)$ is an equivalence.

(c) Let $X$ be connected. Then, by \re{haar} (c), for every fp-closed connected subscheme $Y\subset X$ there exists a unique Haar measure $\mu^Y\in \wt{M}(X)$, whose $*$-pullback to $Y$ is $\B{D}_Y$.

(d) Let  $\C{F}\in\wt{M}(Y)$ be a Haar measure on $Y$. If $f:X\to Y$  is finitely presented, then $f^*\C{F}\in \wt{M}(X)$ is a Haar measure (by \re{haar} (c)). Similarly, if $f:X\to Y$ is schematic and formally smooth, then $f^!\C{F}\in \wt{M}(X)$ is a Haar measure (by \re{haar} (d)).
\end{Emp}

\begin{Emp}
{\bf Remark.}
The main reason why we introduced category $\wt{M}(X)$ and Haar measures was to find a way to identify the category of "smooth  functions" $D(X)$ with
the category of "smooth measures" $M(X)$.
\end{Emp}

\subsection{Generalization to (ind-)stacks}


\begin{Emp} \label{E:st}
{\bf Definitions.} (a) Let $\St_k$ is the $2$-category of stacks over $\Spec k$ (see \cite[3.1]{LMB}), and let
$\Art_k\subset \St_k$ be the full  subcategory of consisting of Artin stacks of finite type over $k$.
Note that the $2$-category $\St_k$ is stable under all small (2-)limits.

(b) Denote by $\St'_k\subset \St_k$ the full 2-subcategory
consisting of $X\in\St_k$, which can be represented by a filtered projective limit $X\cong\lim_i X_i$, where $X_i\in\Art_k$ for all $i$.
In this subsection we generalize notions and results defined above from $\Sch_k$ to $\St'_k$.

(c) We say that a morphism $f:Y\to X$ in $\St'_k$ is of {\em finite presentation}, if $f$ is equivalent to a pullback of a morphism
$f':Y'\to X'$ in $\Art_k$. In this case, we will say that $f':Y'\to X'$ is a presentation of $f$.

(d) We say that a finitely presented morphism $f:X\to Y$ in $\St'_k$ is smooth, resp. closed embedding, resp. representable, if there is a presentation $f':X'\to Y'$ of $f$, satisfying these properties.
\end{Emp}

\begin{Emp} \label{E:examples}
{\bf Examples.}
(a) Rydh \cite{Ry} showed that all quasi-compact and quasi-separated DM-stacks and
many Artin stacks belong to $\St'_k$. In particular, $\Sch_k\subset\St'_k$ (see \cite{TT}).
Furthermore, in these cases the transition maps can be made affine. Moreover, it follows from
\rl{lim} and its extension to stacks (see \cite[App. B]{Ry}) that in these cases our definitions
are equivalent to the standard ones.

(b) Assume that $X\in\Sch_k$ is equipped with an action of a quasi-compact group scheme $G$ over $k$ such that
$X$ has a $G$-equivariant presentation $X\cong\lim_i X_i$, where all $X_i\in \Var_k$ and all transition maps
are affine. Then the quotient stack $X/G$ belongs to $\St'_k$.

\begin{proof}[Proof of (b)]
By assumption, $X/G\cong \lim_i(X_i/G)$. Since $\St'_k$ is stable under filtered limits (see \rl{stlim} (a) below), it is
enough to show that each $X_i/G$ belongs to $\St'_k$. Thus we may assume that $X\in \Var_k$.
By \cite{Pe}, $G$ can be written as a filtered projective limit $G\cong\lim_i G_i$  of group schemes of finite type.
We claim that there exists $j$ such that the action of $G$ on $X$ factors through $G_j$.
Indeed, the action map $G\times X=\lim_j (G_j\times X)\to X$ factors through $G_j\times X$ (by \rl{lim} (c)).
Now the assertion follows from the equivalence $X/G\cong\lim_{i\geq j} X/G_i$.
\end{proof}
\end{Emp}

\begin{Lem} \label{L:stlim}
Let $\{X_i\}_{i\in I}$ be a filtered projective system in $\St'_k$. Then
all assertions (a)-(f) of \rl{lim} hold for $\St'_k$.
\end{Lem}

\begin{proof} To avoid discussion about ordinals, we assume that $I=\B{N}$, which is sufficient for this work.

{\bf Step 1.} First we prove assertion (c) in the case when $X_i\in\Art_k$  for all $i$. Let $\Sp_k$ be the category of algebraic spaces of finite type over $k$.
We are going to reduce the problem to the case when $X_i\in \Sp_k$ for each $i$.

Let $\Dt_{\leq 2}$ be the full subcategory of the simplicial category $\Dt$ with objects $[0],[1],[2]$.  Choose a smooth surjective morphism $p_i:U_i\to X_i$ with $U_i\in \Var_k$. Then $p_i$ gives rise to the functor $\un{U}_i:\Dt_{\leq 2}^{op}\to\Sp_k$ such that
$\un{U}_i[0]=U_i$, $\un{U}_i[1]=U\times_{X_i} U$ and $\un{U}_i[2]= U\times_{X_i} U\times_{X_i} U$
with natural transition maps. Then the "sheaf condition" on $\St_k$ means that $X_i\cong \colim_{[j]\in \Dt_{\leq 2}^{op}}\un{U}_i[j]$ (compare
\cite[proof of Prop 4.18]{LMB}), thus $X\cong\lim_{i}\colim_{[j]\in\Dt_{\leq 2}^{op}} \un{U}_i[j]$.

Moreover, by induction, we can assume that the $p_i$'s come from a morphism of projective systems $\{U_i\}_i\to \{X_i\}_i$.
Then functors $\un{U}_i:\Dt_{\leq 2}^{op}\to\Sp_k$ would also form a projective system, thus we can form a limit
$\un{U}:=\lim_i\un{U}_i:\Dt_{\leq 2}^{op}\to\St'_k$. Explicitly, we have $\un{U}[j]=\lim_i\un{U}_i[j]$ for each $j$.

Since finite colimits commute with filtered limits, we have
$X\cong\colim_{[j]\in\Dt_{\leq 2}^{op}}\un{U}[j]$.
Therefore we have an isomorphism $\Hom_{X_{i'}}(X,Y)\cong \lim_{[j]\in\Dt_{\leq 2}^{op}}\Hom_{X_{i'}}(\un{U}[j],Y)$, and similarly for $\Hom_{X_{i'}}(X_i,Y)$.
Hence it suffices to show that the natural map
\[
\colim_i \Hom_{X_{i'}}(\un{U}_i[j],Y)\to\Hom_{X_{i'}}(\un{U}[j],Y)
\]
is an equivalence for all $j$. Since $\un{U}_i[j]\in\Sp_k$  for all $i$ and $j$, we thus reduce the problem to the case when each $X_i\in\Sp_k$.

Iterating this argument (but slightly easier), we reduce the problem first to the case when the $X_i$'s are schemes,
then to the case of quasi-affine schemes, and finally to the case of affine schemes. In this case, $X$ is an affine scheme,
all the transition maps are affine, and the assertion is shown, for example,
in \cite[Prop 4.15]{LMB} (compare also \cite[Prop. B.1]{Ry}).

{\bf Step 2.} Proof of (a). For each $i$, choose a projective system $\{X'_{ij}\}_{j\in I_j}$
such that $X_i\cong\lim_j X'_{ij}$. Again, we only treat the case when $I_j=\B{N}$ for all $j$.

It suffices to construct a projective system of projective systems  $\{\{X_{ij}\}_j\}_i$ in $\Art_k$ such that
$X_i\cong\lim_j X_{ij}$, and the map $X_{i+1}\to X_i$ is induced by the projection
$\{X_{(i+1)j}\}_j\to \{X_{ij}\}_j$ for each $i$. Indeed, this implies that $X\cong\lim_{i,j} X_{ij}\cong\lim_{i} X_{ii}$.

We construct the $X_{ij}$'s by induction on $i$. First, we set $\{X_{0j}\}_j:=\{X'_{0j}\}_j$.
Next, arguing by induction and using the particular case of (c) proven above, we find an increasing sequence $\{r_j\}_j$ such that the map $\lim_j X'_{1j}\cong X_1\to X_0\to X_{0j}$ factors though $X'_{1r_j}$, and the compositions $X'_{1r_j}\to X'_{1r_{j-1}}\to X_{0(j-1)}$
and $X'_{1r_j}\to X_{0j}\to X_{0(j-1)}$ are isomorphic for all $j>0$. We now set $\{X_{1j}\}_j:=\{X'_{1r_j}\}_j$ and continue by induction on $i$.

{\bf Step 3.} Now we are ready to deduce the general case of (c) from the already shown particular case.  Choose $X_{ij}$ as in the proof of (a).
By definition, a finitely presented morphism $Y\to X_{i'}$ is a pullback of a morphism $Y'\to X'$ in $\Art_k$. Then
$\Hom_{X_{i'}}(X,Y)\cong\Hom_{X'}(X,Y')$, and $\Hom_{X_{i'}}(X_i,Y)\cong\Hom_{X'}(X_i,Y')$ for $i\geq i'$.

Since $X_{i'}\cong\lim_j X_{i' j}$, the projection $X_{i'}\to X'$ factors through some $X_{i'j'}$ (by the particular case of (c)). Replacing $Y'\to X'$ by its pullback under $X_{i'j'}\to X'$, we may assume that $X'=X_{i'j'}$. Since $X\cong\lim_{i\geq i',j\geq j'} X_{ij}\cong\lim_{i\geq i',j'} X_{ii}$
and also $X_i\cong\lim_{j\geq j'} X_{ij}$ for $i\geq i'$, both $\Hom_{X'}(X,Y')$ and $\colim\Hom_{X'}(X_i,Y')$ are equivalent to
$\colim_{i\geq i',j'}\Hom_{X'}(X_{ii},Y')$ by the particular case of (c).

{\bf Step 4.} Now (b) and (d)  formally follow from (c) and definitions, (e) follows from (b), while (f) follows from (c),(e), and the observation
that all classes of morphisms are stable under pullbacks.  For example, assume that
$f_i:Y_i\to X_i$ is finitely presented, and its pullback $f:Y\to X$ is smooth. Then there exists a smooth morphism $f':Y'\to X'$
in $\Art_k$, whose pullback to $X$ is  $f$. Using (c) and increasing $i$, we may assume that the morphism $X\to X'$
factors through $X_i$. Denote by $f'_i:Y'_i\to X_i$ the pullback of $f'$ to $X_i$. Using (e), there exists $j\geq i$ such that
$f_i\times_{X_i} X_j\cong f'_i\times_{X_i}X_j$. Then $f_i\times_{X_i} X_j$ is smooth.
\end{proof}

\begin{Emp} \label{E:allim}
{\bf Remark.} By \rl{stlim} (a), the subcategory $\St'_k\subset \St_k$ is stable under all filtered limits.
Since $\Art_k$ is stable under fiber products, we conclude that $\St'_k\subset\St_k$ is stable under all
small limits.
\end{Emp}

\begin{Emp} \label{E:stacks}
{\bf Admissible (ind)-stacks.}
(a) Recall that to every $X\in \Art_k$ one can associate its bounded derived category of constructible $\qlbar$-sheaves $D(X)=D^b_c(X)$ and bounded from above category $D^-(X)=D^-_c(X)$.

Every morphism $f:X\to Y$ induce functors $f^*,f^!:D^?(Y)\to D^?(X)$ (where $D^?$ is either $D$ or $D^-$) and $f_!:D^-(X)\to D^-(Y)$,
satisfying all basic properties of morphisms in $\Var_k$ including the  base change isomorphism
(see \cite{LZ1,LZ2}). Furthermore, $f_!$ induces functor $D(X)\to D(Y)$, if, in addition, $f$ is representable.

(b) Mimicking \re{unip}, we say that a morphism $f:X\to Y$ in $\Art_k$ is unipotent, if $f$ is smooth,  and all the geometric fibers of $f$ are acyclic. Then \rl{ff} holds without any changes. Using \rl{stlim}, we can now generalize all the notions defined earlier from $\Sch_k$ by $\St'_k$, replacing $\Var_k$ by $\Art_k$ in
all places, but we do not require that the transition morphisms are affine.

For example,  we say that $X\in \St'_k$ is an {\em admissible stack over $k$}, if $X\cong \lim_i X_i$, where
$X_i\in\Art_k$ for each $i$ and each morphism $X_{i}\to X_j,i>j$ is unipotent.
We denote the category of admissible stacks by $\ASt_k$.

Similarly, we define admissible morphisms, categories $M(X)$, $M^-(X)$ and $D(X)$,  pullback functors
$f^*$ and $f^!$, Haar measures and tensor products. Likewise, to each admissible morphism $f:X\to Y$ in
$\St'_k$ we associate the push-forward functor $f_!:M^-(X)\to M^-(Y)$.

(c) Finally, we define the category of admissible ind-stacks $\AISt_k$ and extend results from admissible ind-schemes to this setting.
\end{Emp}

\begin{Emp} \label{E:notst}
{\bf Notation.}
(a) We call an admissible morphism $f:X\to Y$ in $\ASt_k$ {\em strongly representable}, if $f$
has an admissible presentation $X\cong\lim_i X_i$, such that all maps
$X_i\to Y$ and $X_i\to X_j,i>j$ are representable.

(b) In the situation of (a), we have $f_!(M(X))\subset M(Y)\subset M^-(Y)$. Indeed, arguing as in \rl{pf},
we may assume that $f$ is representable and finitely presented. Note that in order to check that an object of $M^-(Y)$ belongs to $M(Y)$
we can check it on geometric fibers. Thus, by the base change, we may assume that $Y=\Spec k$,
in which case the assertion is standard.

(c) We call an admissible morphism
$f:X\to Y$ in $\AISt_k$ {\em strongly representable}, if for every pair of fp-closed substacks $X'\subset X$ and
$Y'\subset Y$ such that $f|_{X'}:X'\to Y$ factors through $Y'$, the induced morphism
$f|_{X'}:X'\to Y'$ is strongly representable. In this case, we have $f_!(M(X))\subset M(Y)$ (by (b)).

(d) We call a morphism $f:X\to Y$ in $\AISt_k$ {\em stacky}, if $f^{-1}(X')$ is a stack (rather than an ind-stack)
for every fp-closed substack $X'\subset X$. In this case, we have a functor $f^!:M(Y)\to M(X)$.
\end{Emp}




\begin{Lem} \label{L:unipstack}
(a) Let $G$ be an affine algebraic group over $k$, $U\subset G$ a normal unipotent subgroup, and $H=G/U$.
Assume that $H$ acts on $X\in \Var_k$, and that $G$ acts on $X$ via the projection $G\to H$. Then the natural morphism of quotient
stacks $p:X/G\to X/H$ is unipotent.

(b) In the example \re{examples} (b) assume that presentation $X\cong\lim_{i\in I} X_i$ is admissible and
that $G$ has a presentation $G\cong \lim_{j\in J} G_j$, where each $G_j$ is an algebraic group, and each homomorphism $G_{j}\to G_j'$ is unipotent
(as morphism). Then  the quotient stack $X/G$ is admissible.
\end{Lem}

\begin{proof}
(a) Since $p$ is smooth,  it remains to show that $p^!:D(X/H)\to D(X/G)$ is fully faithful or, equivalently, that
the counit map $p_!p^!\to\Id$ is isomorphism (by \rl{ff}). Property of being an isomorphism can be checked
on geometric fibers. Since each geometric fiber of $p$ is
$\pt/U$, we reduce to the case when $X=\pt=H$, thus $G=U$. Let  $q:\pt\to\pt/U$ be the natural projection. Then $p\circ q=\Id$, thus $q^!p^!\cong\Id$, so it remains to show that $q^!$ is fully faithful. But $q$ is unipotent, because its geometric fiber is $U$, thus the assertion follows from \rl{ff}.

(b) Assume that $I=J=\B{N}$. Arguing as in \re{examples} (b) and renumbering indexes if necessary, we may assume that the action of $G$ on $X_i$ factors through $G_i$ for each $i$. Moreover, by (a), for every $i$, the composition  $X_{i+1}/G_{i+1}\to X_{i}/G_{i+1}\to X_i/G_{i}$ is unipotent.
Since $X/G\cong\lim_{i} X_i/G_{i}$, this implies that $X/G$ is admissible.
\end{proof}





\subsection{The geometric $2$-category} \label{SS:2cat}
Let $C$ be a $2$-category, whose objects are $X\in\Var_k$; for every $X,Y\in\Var_k$, the category of morphisms $\Hom_{C}(X,Y)$ is the category $D(X\times Y)$;
and the composition $D(X\times Y)\times D(Y\times Z)\to D(X\times Z)$ is the map
$(\C{A},\C{B})\mapsto (p_{XZ})_!(p_{XY}^*\C{A}\otimes p_{YZ}^*\C{B})$. In this subsection we construct a similar $2$-category
with objects  $X\in\AISt_k$.

\begin{Emp} \label{E:half}
{\bf Category of half-measures.}
For $X,Y\in \AISt_k$, we are going to define a category $M_X(X\times Y)$ of "measures along $X$ and functions along $Y$".

(a) For every morphism $f\times g:U_1\times V_1\to U_2\times V_2$ in $\Art_k$, we denote by $f^!\times g^*$
the functor $(f\times\Id_{V_1})^!\circ (\Id_{U_2}\times g)^*\cong (\Id_{U_1}\times g)^*\circ (f\times\Id_{V_2})^!:D(U_2\times V_2)\to D(U_1\times V_1)$.

(b) For each $X,Y\in \ASt_k$, we define $M_X(X\times Y)$ to be the inductive limit
$\colim^{!,*}_{U\in (X/\cdot)^{op},V\in (Y/\cdot)^{op}} D(U\times V)$, with respect to functors $f^!\times g^*$, defined in (a).

(c) For $X,Y\in\AISt_k$, we define
$M_X(X\times Y):=\colim_{X',Y'} M_{X'}(X'\times Y')$, where  $X'\subset X$ and $Y'\subset Y$ are
all fp-closed substacks, and the transition maps are pushforwards with respect to closed embeddings.

(d) Also, replacing $D$ by $D^-$ in (a)-(c), we define a category $M^-_X(X\times Y)$.
\end{Emp}

\begin{Emp} \label{E:hm}
{\bf Functoriality of half-measures.} Let $X$ and $Y$ be admissible schemes.

(a) Note that the projection $p_X:X\times Y\to X$ give rise to the pullback functor
$p_X^*:M(X)\to M_X(X\times Y)$, which maps $\C{F}\in D(V)\subset M(X)$ to the image of
$\C{F}\in D(V\times \Spec k)$. Similarly,
$p_Y:X\times Y\to Y$ gives rise to the pushforward functor $(p_Y)_!:M_X(X\times Y)\to D(Y)$,
which maps $\C{F}\in D(U\times V)\subset M_X(X\times Y)$ to $(p_{V})_!\C{F}\in D(V)\subset D(Y)$
for every $U\in (X/\cdot)^{un}$ and $V\in (Y/\cdot)^{un}$.

(b) As in \re{tensor}, we define a functor
$\otimes:M_X(X\times Y)\times M_Y(X\times Y)\to M(X\times Y)$, induced by the tensor product on
$D(U\times V)$ for every $U\in (X/\cdot)^{un}$ and $V\in (Y/\cdot)^{un}$.

(c) By (a) and (b), each $\C{B}\in M_Y(X\times Y)$ defines a functor $M(X)\to M(Y)$, given by the rule
$\C{B}(\C{F})=(p_Y)_!(p_X^*\C{F}\otimes \C{B})$. Moreover, this formula defines a functor $M_Y(X\times Y)\times M(X)\to M(Y)$.

(d) For every $X,Y,Z\in\ASch_k$, we have a natural functor
\[M_Y(X\times Y)\times M_Z(Y\times Z)\to M_Z(X\times Z),\]
defined by the rule $(\C{A},\C{B})\mapsto (p_{XZ})_!(p_{XY}^*\C{A}\otimes p_{YZ}^*\C{B})$, where $p_{XY}$ is the projection $X\times Y\times Z\to X\times Y$, and similarly $p_{YZ}$ and $p_{XZ}$. Here $p_{XY}^*\C{A}\in M_Y(X\times Y\times Z)$,
$p_{YZ}^*\C{B}\in M_Z(X\times Y\times Z)$ (as in (a)), thus $p_{XY}^*\C{A}\otimes p_{YZ}^*\C{B}\in M_{Y\times Z}(X\times Y\times Z)$
(as in (b)), thus $(p_{XZ})_!(p_{XY}^*\C{A}\otimes p_{YZ}^*\C{B})\in M_Z(X\times Z)$ (as in (a)).

(e) Assertions (c)-(d) extend to the ind-schemes almost without changes, but with extra care. For example, in (c), though
$p_X^*\C{F}\notin M_X(X\times Y)$, if $Y$ is not a scheme, we do have
$p_X^*\C{F}\otimes \C{B}\in M(X\times Y)$.
 \end{Emp}

\begin{Emp} \label{E:hm2}
{\bf Extension to pro-categories and ind-stacks.}
(a) By \re{hm} (c)-(e), for every $X,Y,Z\in\AISch_k$, we have functors  $\Pro M_Y(X\times Y)\times \Pro M(X)\to\Pro M(Y)$
and $\Pro M_Y(X\times Y)\times \Pro M_Z(Y\times Z)\to \Pro M_Z(X\times Z)$.

(b) Note that for each $X\in\AISt_k$ there is a natural (fully faithful) embedding
$M^-(X)\hra \Pro M(X)$. Indeed, if $X\in\Art_k$, then the embedding sends $\C{F}\in D^-(X)$ to
$\lim_n (\tau^{\geq -n}\C{F})\in \Pro D(X)$, where $\tau^{\geq -n}$ is the truncation functor.
The extension to $\ASt_k$ and $\AISt_k$ is straightforward.

(c)  We claim that both functors in (a) exist for $X,Y,Z\in\AISt_k$. For example, by definition,
we have a functor $M_Y(X\times Y)\times M(X)\to M^-(Y)\subset\Pro M(Y)$ (see (b)), which extends to the functor
$\Pro M_Y(X\times Y)\times \Pro M(X)\to\Pro M(Y)$.
\end{Emp}

\begin{Emp} \label{E:2cat}
{\bf The geometric $2$-category.} Let $X,Y,Z\in\AISch_k$.

(a) By \re{hm2}, every $\C{B}\in \Pro M_Y(X\times Y)$ defines a functor $M(X)\to\Pro M(Y)$.
We denote by $\Hom_{geom}(X,Y)$ the full subcategory of $\Pro M_Y(X\times Y)$, spanned by
objects $\C{B}$ such that $\C{B}(\C{F})\in M(Y)\subset\Pro M(Y)$ for every $\C{F}\in M(X)$.
By definition, each $\C{B}\in\Hom_{geom}(X,Y)$ induces a functor  $M(X)\to M(Y)$.

(b) Note that the second functor of \re{hm2} (a) induces a functorial associate composition morphism
$\circ:\Hom_{geom}(Y,Z)\times\Hom_{geom}(X,Y)\to \Hom_{geom}(X,Z)$.

(c) We claim that for every $X\in\AISt_k$ there exists $\C{B}_{\Id_X}\in\Hom_{geom}(X,X)$ together with natural isomorphisms $\C{B}_{\Id_X}(\C{F})\cong\C{F}$
for every $\C{F}\in M(X)$, $\C{B}_{\Id_X}\circ\C{B}\cong\C{B}$ for every $\C{B}\in \Hom_{geom}(Y,X)$, and $\C{B}'\circ \C{B}_{\Id_X}\cong\C{B}'$
for every $\C{B}'\in \Hom_{geom}(X,Y)$.

\begin{proof}
Assume first that $X\in\ASt_k$. Then for every $V\in (X/\Art_k)^{un}$, the diagonal morphism $\Dt_V:V\to V\times V$ is representable,
so we can form an object $\C{B}_{\Id_{V}}:=\Dt_{V!}(1_V)\in D(V\times V)\subset M_X(X\times X)$. Moreover, it is not difficult to see
that the $\C{B}_{\Id_{V}}$'s form a filtered projective system, so we can define an object
$\C{B}_{\Id_X}:=\lim_{V}\C{B}_{\Id_{V}}\in\Pro M_X(X\times X)$.

The extension of this construction to $X\in\AISt_k$ and the proof that $\C{B}_{\Id_X}$ satisfies all the required properties is standard.
\end{proof}

(d) By (b) and (c), there exists a $2$-category $M_{geom}$, whose objects are $X\in\AISt_k$,
the category of morphisms is $\Hom_{M_{geom}}(X,Y):=\Hom_{geom}(X,Y)$,
the composition is the map from (b), and the unit in $\Hom_{M_{geom}}(X,X)$ is $\C{B}_{\Id_X}$.
\end{Emp}

\begin{Emp} \label{E:corr}
{\bf The $2$-category of admissible correspondences}.

(a) Let $X,Y\in \AISt_k$. By an {\em admissible correspondence} $(g,f):X\dra Y$ we mean a diagram $X\overset{g}{\leftarrow}Z\overset{f}{\to} Y$,
where $g$ is stacky and formally smooth, and $f$ is admissible and strongly representable (see \re{notst}). For every two admissible correspondences $X{\leftarrow}Z{\to} Y$ and
$Y{\leftarrow}Z'{\to} T$, we can form their composition
$X{\leftarrow}Z\times_Y Z'{\to} T$, which is an admissible correspondence $X\dra T$.

We denote by $\Corr_k$ the $2$-category, whose objects are $X\in \AISt_k$, for every $X,Y\in\AISt_k$, the category $\Mor_{\Corr_k}(X,Y)$ is the category of admissible correspondences $X\dra Y$, and the composition 
is as above.

(b) Every admissible correspondence $(g,f):X\dra Y$ gives rise to the functor
$f_!\circ g^!:M(X)\to M(Z)\to M(Y)$ (see \re{notst}), and similarly to the functor
\[f_!\circ g^!:\Pro M_X(T\times X)\to \Pro M_Z(T\times Z)\to \Pro M_Y(T\times Y)\] for each $T\in\AISt_k$.

(c) We claim that for every $\C{B}\in\Hom_{geom}(T,X)\subset \Pro M_X(T\times X)$ and every $\C{F}\in M(T)$, we have an equality
$(f_!g^!\C{B})(\C{F})=f_!g^!(\C{B}(\C{F}))\in M(Y)\subset \Pro M(Y)$, thus $f_!g^!\C{B}\in \Hom_{geom}(T,Y)$.

It is enough to consider two cases:
$f=\Id_X$, and $g=\Id_Y$. In the first case, the assertion follows from the identity $g^!(p_X^*\C{F}\otimes\C{B})\cong g^*p_X^*\C{F}\otimes g^!\C{B}$,
which holds because $g$ is formally smooth. In the second case, the assertion follows from the projection formula
$f_!(f^* p_X^*\C{F}\otimes\C{B})\cong p_X^*\C{F}\otimes f_!\C{B}$.

(d) By (c), to every admissible correspondence $(g,f):X\dra Y$ one can associate an object $\C{B}_{g,f}:=f_!g^!(\C{B}_{\Id_X})\in \Hom_{geom}(X,Y)$.
Moreover, the assignment $X\mapsto X$, $(g,f)\mapsto\C{B}_{g,f}$ can be upgraded to a functor of $2$-categories $\Corr_k\to M_{geom}$.
\end{Emp}

\section{Categorification of the Hecke algebra}

\subsection{Loop spaces}


\begin{Emp} \label{E:loop}
{\bf Notation.} Set $K:=k((t))$, $\C{O}:=k[[t]]]$, and fix $n\in\B{Z}_{\geq 0}$.

For an affine scheme of finite type $X$ over $K$ (resp. $\C{O}$), we denote by $LX$
(resp. $L^+X$, resp. $L^+_n X$ ) the functor $\Alg_k\to \Set$ from $k$-algebras to sets, defined by the rule $LX(R):=\Hom_{K}(\Spec R((t)),X)$ (resp. $L^+X(R):=\Hom_{\C{O}}(\Spec R[[t]],X)$, resp. $L^+_n X(R):=\Hom_{\C{O}}(\Spec R[t]/(t^{n+1}),X)$).

It is well known that $LX$ (resp. $L^+X$, resp. $L^+_n X$) is represented by an ind-scheme
(resp. scheme) over $k$. Indeed, when $X=\B{A}^m$ the assertion is easy, while the general case follows from the fact
that if $X\hra\B{A}^m$ is a closed embedding over $K$ (resp. $\C{O}$), then $LX\hra L\B{A}^m$
(resp.  $L^+X\hra L^+\B{A}^m$, resp.  $L^+_n X\hra L^+_n\B{A}^m$) is a closed embedding.

We call $LX$ (resp. $L^+X$) the loop ind-scheme (resp. scheme) of $X$. Note that
the natural homomorphism $k[[t]]]\to k[t]/(t^{n+1})$ induces a projection $p_n:L^+X\to L^+_n X$.
\end{Emp}

\begin{Lem} \label{L:nbh}
Let $f:X\to Y$ be a pro-unipotent morphism of affine schemes over $k$, and let $X'\subset X$ be an fp-closed subscheme.  Then for every $x\in X'(k)$ there exist fp-closed neighborhoods $X''\subset X'$ of $x$ and $Y''\subset Y$ of $f(x)$ such that $f$ induces a pro-unipotent morphism $X''\to Y''$.
\end{Lem}

\begin{proof}
Choose an admissible presentation $X\cong\lim_i X_i$ of $f$ such that $X_{i}\to Y$ is unipotent for all $i$. Then there exists $i$ and an fp-closed
subscheme $X'_i\subset X_i$ such that $X'\subset X$ is the preimage of $X'_i\subset X_i$ (by \rl{lim} (f)). Let $x_i\in X'_i$ be the image of $x$. It is enough to find fp-closed neighborhoods $X''_i\subset X'_i$ of $x_i$ and $Y''\subset Y$ of $f(x)$ such that $f_i:X\to Y$ induces a unipotent morphism $X''_i\to Y''$. Indeed, in this case the preimage $X''\subset X'$ of $X''_i\subset X'_i$ is an fp-closed neighbourhood of $x$ and the induced map $X''\to X''_i\to Y''$ is pro-unipotent. Replacing $f$ by $X_i\to Y$, we can thus assume that $f$ is unipotent.

Since $X'\hra X$ and $f:X\to Y$ are finitely presented morphisms between affine schemes over $k$,
there exist a unipotent morphism $f_{\al}:X_{\al}\to Y_{\al}$ of affine schemes of finite type over $k$, a closed embedding
$X'_{\al}\hra X_{\al}$ and a morphism $\pi:Y\to Y_{\al}$ such that $X'\hra X\to Y$ is a pullback of $X'_{\al}\hra X_{\al}\to Y_{\al}$ under $\pi$. Let $x_{\al}\in X'_{\al}$ be the image of $x$, let $X''\subset X'$ be the preimage of
$\{x_{\al}\}\subset X'_{\al}$, and let $Y''\subset Y$ be the preimage of $\{f_{\al}(x_{\al})\}\subset Y_{\al}$.
Then $f$ induces an isomorphism $X''\isom Y''$, thus $X''$ and $Y''$ are the required fp-closed neighbourhoods.
\end{proof}

\begin{Lem} \label{L:smooth}
Let $f:X\to Y$ be a morphism of affine varieties over $K$, let $Lf:LX\to LY$ be the induced morphism of loop ind-schemes, and let $U'\subset LX$ be an fp-closed neighbourhood of $x\in X(K)=LX(k)$.

(a) If $f$ is \'etale at $x$, then for every sufficiently small fp-closed neighborhood $U_x\subset LX$ of $x$ the restriction $Lf|_{U_x}:U_x\to LY$ is an fp-closed embedding, thus $U_x$ can be viewed as an fp-closed neighborhood of $f(x)\in LY$.

(b) If $f$ is smooth at $x$, then  there exist fp-closed neighborhoods $U_x\subset U'\subset LX$ of $x$ and
$V_{f(x)}\subset LY$ of $f(x)$ such that $Lf|_{U_x}$ defines a pro-unipotent morphism $U_x\to V_{f(x)}$.
\end{Lem}

\begin{proof}
(a) It is enough to show that there exists an fp-closed neighborhood $U_x\subset LX$ of $x$
such that $Lf|_{U_x}:U_x\to LY$ is a finitely presented closed embedding, because in this case this will be true
for every smaller neighborhood.

Assume first that $X=Y=\B{A}^n$. Changing coordinates, we may assume that  $x=f(x)=\ov{0}=(0,\ldots,0)$. Composing $f$ with the linear map $(df|_{\ov{0}})^{-1}$, we can assume that the differential $df|_{\ov{0}}=\Id$. Now the assertion is simply a version of the formal inverse function theorem.

To reduce the general case to the particular case, shown above, we argue as follows.

If $f$ is an open embedding, the assertion is easy (or it can be shown by the same arguments as below).
Therefore we can replace $X$ and $Y$ by open affine neighborhoods of $x$ and $f(x)$, respectively, thus assuming that $X$ is a closed subscheme
of $Y\times\B{A}^n$, given by equations $g_1=\ldots=g_n=0$ for certain regular functions $g_i\in K[Y\times\B{A}^n]$.

Choose a closed embedding $Y\hra \B{A}^m$ and a lift $\wt{g}_i\in K[\B{A}^m\times\B{A}^n]$ of each $g_i$. Consider the closed subscheme $\wt{X}\subset\B{A}^m\times\B{A}^n$, given by equations $\wt{g}_1=\ldots=\wt{g}_n=0$. Then the projection
$\pr_{\B{A}^m}:\wt{X}\to \B{A}^m$ is \'etale at $x\in X\subset\wt{X}$.

Consider morphism $\wt{f}=(\pr_{\B{A}^m};\wt{g}_1,\ldots,\wt{g}_n):\B{A}^m\times\B{A}^n\to \B{A}^m\times\B{A}^n$.
Then $\wt{f}$ is \'etale at $x\in\wt{X}\subset \B{A}^m\times\B{A}^n$, and the restriction of $\wt{f}$ to $Y\times \{\ov{0}\}$ is $f$. Thus the assertion for $f$ follows from that for $\wt{f}$, and the proof is complete.

(b) Using (a), we can replace $X$ by a Zariski open neighbourhood of $x$. Thus we can assume that $f:X\to Y$ lifts to an \'etale  morphism
$g:X\to Y\times\B{A}^n$ such that $g(x)=(f(x),0)$. Using (a) again, we can assume that
$X=Y\times\B{A}^n$, $f$ is the projection $Y\times\B{A}^n\to Y$, and $x=(y,0)$. Choose any fp-closed neighbourhood $V\subset LY$ of $y$, and set $U:=V\times L^+\B{A}^n$. Then $U$ is an fp-closed neighborhood of $x$, and $f$ induces a pro-unipotent morphism $\pr:U\to V$. Replacing $U'$ by
$U'\cap U$, we can assume that $U'\subset U$. Now the assertion
follows from \rl{nbh} applied to $\pr:U\to V$.
\end{proof}

\subsection{The categorical Hecke algebra}

\begin{Emp} \label{E:mainnot}
{\bf Notation.} (a) Let $G$ be a connected semisimple and simply
connected group over $k$, and let $LG:=L(G_K)$ be the loop group ind-scheme of $G$ (see \re{loop}).

(b) We fix a Borel subgroup $B\subset G$ and a maximal torus $T\subset B$. Let $\I\subset LG$ be the Iwahori subgroup scheme corresponding to $B$,
that is, $\I$ is the preimage of $B\subset G$ under the projection $p_0:L^+G\to L^+_0G=G$ (see \re{loop}).

(c) Let $\Fl$ be the affine flag variety $LG/\I$, and let $\wt{W}$ be the affine Weyl
group of $G$. For every $w\in\wt{W}$, we denote by $\Fl^{\leq
w}\subset \Fl$ the closure of the $\I$-orbit $Y_w:=\I w\subset \Fl$, and we denote
by $LG^{\leq w}\subset LG$ the preimage of $\Fl^{\leq w}$.


(d) By a standard parahoric subgroup scheme we mean an fp-closed subgroup scheme $\P\subset LG$, containing $\I$.
Denote by $\Par$ the category, corresponding to the partially ordered set of standard parahoric subgroup schemes of $LG$.

(e) For each $\P\in \Par$, we let $\P^+\subset \P$ be the pro-unipotent radical of $\P$.
Explicitly, we have $\I^+=p_0^{-1}(U)\subset L^+G$, where $U$ is the unipotent radical of $B$, and $\P^+\subset \P$ is the largest normal
subgroup scheme, contained in $\I^+$.  We also let
$L_{\P}:=\P/\P^+$ be the "Levi subgroup of $\P$", let $W_{\P}\subset \wt{W}$ be the corresponding parabolic
Weyl subgroup, and set $\rk\P:=\rk (L_{\P})^{der}$ and $\Fl_{\P}:=LG/\P$.

(f) For each $\P\in \Par$ and $n\in \B{N}$, we denote by $\P_n^+\subset\P^+$ the $n$-th
congruence subgroup scheme of $\P^+$. In particular, we have $\P^+_0=\P^+$. Explicitly, for $n>0$, we denote by $V\subset L^+_n U\subset L^+_nG$
the kernel of the projection $L^+_n U\to L^+_{n-1} U$, set $\I^+_n:=p_n^{-1}(V)\subset L^+G$, and let $\P_n^+$ be the largest normal
subgroup scheme of $\P$, contained in $\I^+_n$.
\end{Emp}

\begin{Emp} \label{E:hecke}
{\bf The categorical Hecke algebra.}
(a) For each $w\in\wt{W},\P\in \Par$ and $n\in\B{N}$, the quotient $LG^{\leq w}/\P_n^+$
is a quasi-projective scheme. Namely, it is an $\I/\P_n^+$-torsor over the projective scheme $\Fl^{\leq w}$.

Moreover, for each $m\geq n$, the projection
$LG^{\leq w}/\P_m^+\to LG^{\leq w}/\P_n^+$ is a $\P_n^+/\P_m^+$-torsor, thus an unipotent
morphism. Since  $LG^{\leq w}\cong \lim_n (LG^{\leq w}/\P_n^+)$, we conclude that $LG^{\leq w}$ is an admissible scheme, and $\{LG^{\leq w}/\P_n^+\}_n$ is an admissible presentation of $LG^{\leq w}$.

(b) For every $w'\geq w$ in $\wt{W}$, the inclusion $LG^{\leq w}\hra LG^{\leq w'}$ is a pullback of a closed embedding
$\Fl^{\leq w}\hra \Fl^{\leq w'}$ in $\Var_k$, thus it is finitely presented. Hence $LG=\colim_w LG^{\leq w}$ is an admissible
ind-scheme, thus the construction of \re{consind} provides us with categories $M(LG)$ and $D(LG)$.

Category $D(LG)$ is a categorical counterpart of smooth functions with compact support on $G(F)$,
while $M(LG)$ is a categorical counterpart of the Hecke algebra.
\end{Emp}

\begin{Emp} \label{E:basis}
{\bf Basis of fp-closed neighborhoods.}
For every $\gm\in LG(k)$, the set $\{\gm \I_n^+\}_n$ form a basis of fp-closed neighborhoods of $\gm$.
Indeed, since $\gm \I_n^+\subset LG$ is a preimage of the closed point $[\gm]\in LG/\I_n^+$,
we get that $\gm \I_n^+\subset LG$ is an fp-closed neighborhood of $\gm$. Conversely,
let $Y\subset LG$ is an fp-closed neighborhood of $\gm$. Using isomorphism $LG\cong\lim_n (LG/\I_n^+)$, we conclude that $Y$ is a preimage of some closed subscheme $\ov{Y}\subset LG/\I_n^+$ (by \rl{lim} (f)).
Since $\gm\in Y$, we get $[\gm]\in \ov{Y}$, thus $\gm \I_n^+\subset Y$.
\end{Emp}

\begin{Emp} \label{E:conv}
{\bf The convolution.}
(a) The multiplication map $G\times G\to G$ induces a multiplication map  $m:LG\times LG\to
LG$. Moreover, $m$ can be written as a composition of an isomorphism $(LG)^2\to (LG)^2:(x,y)\mapsto (x,xy)$ and the projection
$p:(LG)^2\to LG:(x,y)\mapsto y$. Since $LG$ is admissible, we conclude that $p$ and hence also $m$ is admissible. Therefore $m$
defines the functor $m_!:M((LG)^2)\to M(LG)$, and we denote by $\ast$ the convolution
$\C{A}\ast\C{B}:=m_!(\C{A}\pp\C{B})$. Since multiplication $m$ is associative, the convolution $\ast$ equips
$M(LG)$ with a structure of a monoidal category (without unit).

(b) Arguing as in (a), we see that the standard action $(g,h)(x):=gxh^{-1}$ of $G^2$ on $G$ induces the action of the group ind-scheme
$(LG)^2$ on $LG$, and thus induces an action of the monoidal category $M((LG)^2)$ on $M(LG)$.

(c) For $\P\in\Par$ and $n\in\B{N}$, let $\dt_{\P^+_n}\in M(LG)$ be the corresponding $\dt$-measure (see \re{consind} (b)), and let
$\pr:LG\to LG/\P_n^+$ be the projection.  Then for every $\C{A}\in M(LG)$
we have a natural isomorphism $\C{A}\ast\dt_{\P_n^+}\cong \pr^!\pr_!\C{A}$.

Indeed, consider the Cartesian diagram
\[
\begin{CD}
LG\times \P_n^+ @>m>> LG\\
@VpVV @V\pr VV\\
LG @>\pr>> LG/\P_n^+,
\end{CD}
\]
where $p$ is the projection. Since all morphisms are formally smooth, we have a canonical isomorphism
$\C{A}\ast\dt_{\P_n^+}= m_!p^!\C{A}\cong\pr^!\pr_!\C{A}$.

(d) Note that $\dt_{\P_n^+}\cong\pr^!\dt_1$, where $\dt_1\in M(LG/\P_n^+)$ is the $\dt$-measure at $[1]$.
Since $\P_n^+$ is pro-unipotent, we conclude from (c) that
$\dt_{\P^+_n}\ast\dt_{\P_n^+}\cong\pr^!\pr_!\dt_{\P_n^+}\cong\pr^!\dt_1\cong\dt_{\P_n^+}$.
\end{Emp}

\begin{Emp} \label{E:equivha}
{\bf The equivariant case.}
(a) Each $\P\in \Par$ acts on $LG$ by the adjoint action, so we can form the quotient stack
$\frac{LG}{\P}$. We claim that $\frac{LG}{\P}$ is an admissible ind-stack, hence we can form a category $M(\frac{LG}{\P})$.

Indeed, note that $\P\cong \lim_n\P/\P_n^+$, and all the transition maps are unipotent.
We say that $w\in\wt{W}$ is $\P$-maximal, if $w\in\wt{W}$ be the longest element in the double coset $W_{\P}wW_{\P}$. For every  $\P$-maximal $w$, the scheme $LG^{\leq w}\subset LG$ is $\P\times\P$-invariant, hence $\Ad\P$-invariant. Since $LG^{\leq w}\cong\lim_n (LG^{\leq w}/\P_n^+)$ is an admissible $\P$-equivariant presentation of $LG^{\leq w}$, the quotient $\frac{LG^{\leq w}}{\P}$ is an admissible stack by \rl{unipstack} (b).
Finally, since the quotient $\frac{LG}{\P}$  is the inductive limit of the $\frac{LG^{\leq w}}{\P}$'s, taken over $\P$-maximal $w$'s, the assertion follows.

(b) Note that the multiplication map $m:LG\times LG\to LG$ induces a diagram $\frac{LG}{\P}\times \frac{LG}{\P} \overset{\pi}{\lla}
\frac{LG\times LG}{\P}\overset{\ov{m}}{\lra} \frac{LG}{\P}$. In particular, $m$ gives rise to the convolution $\ast$ on $M(\frac{LG}{\P})$,
defined by the rule $\C{A}\ast\C{B}:=\ov{m}_!\pi^!(\C{A}\pp\C{B})$. The convolution $\ast$ equips $M(\frac{LG}{\P})$ with a structure of
a monoidal category. Moreover, for every $\P\subset\Q$ in $\Par$, the $!$-pullbacks $M(\frac{LG}{\Q})\to M(\frac{LG}{\P})$ and
$M(\frac{LG}{\P})\to M({LG})$ are monoidal.
\end{Emp}

\begin{Emp} \label{E:biinv}
{\bf Bi-invariant measures.}

(a) By \re{conv} (c) and \rl{ff}, for every $\P\in \Par$ and $n\in \B{N}$ the essential image of the embedding $D(LG/\P_n^+)\hra M(LG)$ consists of all $\C{A}\in  M(LG)$ such that $\C{A}\ast\dt_{\P_n^+}\cong \C{A}$. Similarly, the essential image of the embedding $D(\P_n^+\bs LG)\hra M(LG)$ consists of all $\C{A}\in  M(LG)$ such that $\dt_{\P_n^+}\ast\C{A}\cong \C{A}$.

(b) Since $LG\cong\lim_n (LG/P_n^+)$ and $LG\cong\lim_n (\P_n^+\bs LG)$ are admissible presentations of $LG$, we have $M(LG)\cong \colim_n D(LG/\P_n^+)\cong\colim_n D(\P_n^+\bs LG)$. We conclude from (a) that for every $\C{A}\in M(LG)$ there exists $n\in\B{N}$ such that $\C{A}\ast\dt_{\P_n^+}\cong\C{A}\cong\dt_{\P_n^+}\ast\C{A}$, or, equivalently,
$\dt_{\P_n^+}\ast\C{A}\ast\dt_{\P_n^+}\cong \C{A}$.

(c) Since $\dt_{(\P_n^+)^2}\in M((LG)^2)$ satisfies $\dt_{(\P_n^+)^2}(\C{A})=\dt_{\P_n^+}\ast\C{A}\ast\dt_{\P_n^+}$, we conclude from
(b) that for every $\C{A}\in M(LG)$ there exists $n\in\B{N}$ such that $\dt_{(\P_n^+)^2}(\C{A})\cong\C{A}$.

(d) Set $M_{\P_n^+}(LG):=M(\P_n^+\bs LG/\P_n^+)$. Arguing as in (a), the essential image of the embedding
$M_{\P_n^+}(LG)\hra M(LG)$ consists of all $\C{A}\in M(LG)$ which satisfy $\dt_{\P_n^+}\ast\C{A}\ast\dt_{\P_n^+}\cong\C{A}$.
In particular, $M_{\P_n^+}(LG)\subset M(LG)$ is the monoidal subcategory with unit $\dt_{\P_n^+}$.
It now follows from  (b) that $M(LG)\cong\colim_n  M_{\P_n^+}(LG)$, which is
a categorical analog of the fact that the Hecke algebra is idempotent.

(e) Using the description of $M_{\P_n^+}(LG)\subset M(LG)$ in (d), one sees that the monoidal action of $M((LG)^2)$ on $M(LG)$
(see \re{conv} (b)) induces the monoidal action of $M_{(\P_n^+)^2}((LG)^2)$ on $M_{\P_n^+}(LG)$.
\end{Emp}





\subsection{Averaging functors}

\begin{Emp} \label{E:adj}
{\bf Notation.}
(a) Consider the map $a:(LG)^2\to LG$  given by the formula $a(x,y):=xyx^{-1}$. As in \re{conv} (a), the map $a$ is admissible, and  we denote by $a_!:M((LG)^2)\to M(LG)$ the corresponding functor. Every $\C{X}\in M(LG)$ gives rise to the functor
$\Ad^{\C{X}}: M(LG)\to  M(LG)$, defined by the rule $\C{F}\mapsto a_!(\C{X}\pp\C{F})$.

(b) More generally, every morphism $f:X\to LG$ of ind-schemes gives rise to the morphism
$a_f:X\times LG\to LG:(x,y)\mapsto f(x)yf(x)^{-1}$. Arguing as in \re{conv} (a), we conclude that for each admissible  $X$, the functor $a_f$ is admissible.  Hence for each $\C{X}\in M(X)$, morphism $f$ induces a functor $M(LG)\to M(LG)$, given by the rule $\C{F}\mapsto (a_f)_!(\C{X}\pp\C{F})$, which we denote by $\Ad^{X;\C{X}}$ or $\Ad^{f;\C{X}}$, depending on the context.

(c) For $\P\in \Par$, we denote by $LG\times^{\P} LG$ the quotient of $LG\times LG$ by $\P$, acting by the action $g(x,y):=(xg^{-1},gyg^{-1})$. Consider diagram
\[
\Fl_{\P}\times\frac{LG}{\P}\overset{\pr}{\lla} LG\times^{\P}LG \overset{\ov{a}}{\lra}LG,
\]
where $\pr$ is the projection $[x,y]\mapsto([x],[y])$, and $\ov{a}$ is the map
$[x,y]\mapsto xyx^{-1}$.

Note that $\ov{a}$ is finitely presented, because it can be written as the composition of the isomorphism
$LG\times^{\P}LG\to \Fl_{\P}\times LG:[x,y]\mapsto ([x],xyx^{-1})$, and the projection
$\Fl_{\P}\times LG\to LG$. Therefore every $\C{X}\in D(\Fl_{\P})=M(\Fl_{\P})$ gives rise to functors \[\Ad_{\P}^{\C{X}}:M(\frac{LG}{\P})\to M(LG):\C{F}\mapsto \ov{a}_!\pr^!(\C{X}\pp\C{F}) \text{ and }\]
\[\Ad^{\C{X},*}_{\P}:D(\frac{LG}{\P})\to D(LG):\C{F}\mapsto \ov{a}_!\pr^*(\C{X}\pp\C{F}).\]

(d) For every $\Q\in\Par$, the diagram of (c) descends to the diagram
\[
(\Q\bs\Fl_{\P})\times\frac{LG}{\P}\lla (\Q\bs LG)\times^{\P}LG\lra\frac{LG}{\Q}.
\]
Therefore for every $\C{X}\in M(\Q\bs \Fl_{\P})$ the formula as in (c) gives rise to the functor
$\Ad_{\P;\Q}^{\C{X}}:M(\frac{LG}{\P})\to M(\frac{LG}{\Q})$.

By the smooth base change, if $\wt{\C{X}}\in M(\Fl_{\P})$ is the $!$-pullback of $\C{X}$ and $\pi$ is the projection $LG\to \frac{LG}{\Q}$, then
$\pi^!\circ\Ad_{\P;\Q}^{\C{X}}$ is naturally isomorphic to $\Ad_{\P}^{\wt{\C{X}}}$.
\end{Emp}

\begin{Lem} \label{L:aver}
(a) Let $f:X\to LG$ be a morphism  in $\AISch_k$ such that the composition
$g:X\overset{f}{\to}LG\overset{\pi}{\to} \Fl_{\P}$ is admissible.  Then for every $\C{X}\in M(X)$, the maps $M(\frac{LG}{\P})\overset{\pi^!}{\lra} M(LG)\overset{\Ad^{X;\C{X}}}{\lra}M(LG)$ and
$\Ad_{\P}^{g_!(\C{X})}:M(\frac{LG}{\P})\to M(LG)$ are isomorphic.

(b) Let $\mu\in M(\frac{LG}{\P})$ be a Haar measure, $\C{X}$ an object of $D(\Fl_{\P})=M(\Fl_{\P})$ and $\C{F}\in D(\frac{LG}{\P})$. Then we have a natural isomorphism
\[\Ad_{\P}^{\C{X},*}(\C{F})\otimes \pi^!(\mu)\cong \Ad_{\P}^{\C{X}}(\C{F}\otimes\mu).\]
\end{Lem}

\begin{proof}
(a) Consider diagram
\[
\begin{CD}
X\times LG @>\wt{f}>> LG\times^{\P} LG @>\ov{a}>> LG\\
@V\Id\times \pi VV   @V\pr VV\\
X\times \frac{LG}{\P} @>g\times\Id>> \Fl_{\P}\times\frac{LG}{\P},
\end{CD}
\]
where $\wt{f}$ is the map $(x,y)\mapsto [f(x),y]$. Since the square is Cartesian,
while $\pi$ and $\pr$ are formally smooth, we have the base change isomorphism
$\wt{f}_!(\Id\times \pi)^!\cong \pr^! (g\times\Id)_!$. Since $\ov{a}\circ\wt{f}=a_f$, we get the desired isomorphism
$(a_f)_!(\C{X}\pp\pi^!(\cdot))\cong \ov{a}_!\pr^!(g_!(\C{X})\pp\cdot)$.

(b) Since $\ov{a}$ is finitely presented, we get an isomorphism
\[
\Ad_{\P}^{\C{X},*}(\C{F})\otimes \pi^!(\mu)=\ov{a}_!\pr^*(\C{X}\pp\C{F})\otimes \pi^!(\mu)\cong
\ov{a}_!(\pr^*(\C{X}\pp\C{F})\otimes \ov{a}^*\pi^!(\mu)),
\]
(see \re{tensor} (c)), while since $\pr$ is formally smooth, we have a natural isomorphism
\[
\Ad_{\P}^{\C{X}}(\C{F}\otimes\mu)=\ov{a}_!\pr^!((\C{X}\pp\C{F})\otimes(1_{\Fl_{\P}}\pp\mu))\cong
 \ov{a}_!(\pr^*(\C{X}\pp\C{F})\otimes\pr^!(1_{\Fl_{\P}}\pp\mu))
\]
(see \re{tensor} (b)). Thus it remains to show that $\pr^!(1_{\Fl_{\P}}\pp\mu)\cong \ov{a}^*\pi^!(\mu)$.

Note that both sides of the last isomorphism are Haar measures (see \re{haarind} (d)), and the morphism $i:LG\to LG\times^{\P}LG:g\mapsto (1,g)$
is finitely-presented. Using \re{haar} (c), it is enough to show an isomorphism $i^*\pr^!(1_{\Fl_{\P}}\pp\mu)\cong i^*\ov{a}^*\pi^!(\mu)$.
The right hand side is $\pi^!(\mu)$, because $\ov{a}\circ i=\Id$. To see that the left hand is $\pi^!(\mu)$, we use the definition
of $i^*$ and the fact that the diagram
\[
\begin{CD}
LG @>i>> LG\times^{\P} LG \\
@V{\pi}VV   @V{\pr}VV\\
[1]\times \frac{LG}{\P} @>>> \Fl_{\P}\times\frac{LG}{\P}
\end{CD}
\]
is Cartesian.
\end{proof}

\begin{Emp} \label{E:aver}
{\bf Particular case.}
(a) Every locally closed subscheme $Y\subset \Fl_{\P}$ gives rise to the constant $\qlbar$-sheaf
$1_Y\in D(Y)\subset D(\Fl_{\P})=M(\Fl_{\P})$. We denote the functor $\Ad_{\P}^{1_Y}:M(\frac{LG}{\P})\to M(LG)$ from \re{adj} (c)  by $\Av^{Y}$
and call it the averaging functor.

(b) Assume now that $Y$ is $\Q$-invariant. Then there exists a unique Haar measure $1_Y\in M(\Q\bs Y)\subset M(\Q\bs \Fl_{\P})$, whose $!$-inverse image to $D(Y)$ is $1_Y$. We denote by $\Av_{\Q}^Y$ the functor $\Ad_{\P;\Q}^{1_Y}:M(\frac{LG}{\P})\to M(\frac{LG}{\Q})$ from \re{adj} (d). By \re{adj} (d),
the composition of $\Av_{\Q}^Y$ and the $!$-pullback $M(\frac{LG}{\Q})\to M(LG)$ is $\Av^Y$.

(c) For every $\P\in \Par$, the subscheme $\P/\I\subset\Fl$ is $\P$-invariant, thus the construction of (b)
gives rise to the functor $\Av_{\P}^{\P/\I}:M(\frac{LG}{\I})\to M(\frac{LG}{\P})$.

(d) For every morphism $f:X\to LG$, where $X\in \Var_k$, we denote by $\Av^X$ the functor
$\Ad^{X;1_X}:M(LG)\to M(LG)$ from \re{adj} (b).

\end{Emp}


\section{The stable center conjecture}

\subsection{Classical theory}

\begin{Emp}{\bf Notation.} (a) Let $F$ be a local non-archimedean
field with residue field $\fq$, let $W_{F}$ be the Weil group of $F$,
and let $W'_{F}$ be the Weil--Deligne group.

(b) Let $G$ be a connected reductive group over $F$, which we for
simplicity assume to be split, and let $\check{G}$ be
the connected Langlands dual group over $\B{C}$.

(c)  Let $R(G)$ be the category of smooth complex
representations of $G(F)$, and let $\Irr(G)$ be the set of
equivalence classes of irreducible objects in $R(G)$.
\end{Emp}

\begin{Emp} \label{E:berncent}
{\bf The  Bernstein center.} (a) Let $ Z_G= Z(R(G))$ be the
Bernstein center of $G(F)$, that is, the algebra of endomorphisms
of the identity functor $\Id_{R(G)}$ (see \cite{Ber}). By definition, $Z_G$ is a commutative algebra over $\B{C}$.

(b) Each $z\in Z_G$ defines an endomorphism $z|_{\pi}\in\End \pi$
for every $\pi\in R(G)$. In particular, by the Schur lemma, each $z\in Z_G$ defines a function
$f_z:\Irr(G)\to \B{C}$ such that $z|_{\pi}=f_z(\pi)\Id_{\pi}$ for all $\pi\in\Irr(G)$. Moreover,
the map $z\mapsto f_z$ is an algebra homomorphism $Z_G\to \Fun(\Irr(G),\B{C})$, which is known to be injective.

(c) Each $z\in Z_G$ defines an endomorphism
$z_{\C{H}}$ of the Hecke algebra $\C{H}(G(F))$, commuting with left and right convolutions.
For every $(\pi,V)\in R(G)$, $v\in V$ and
$h\in\C{H}(G(F))$, we have an equality
$z(h(v))=(z_{\C{H}}(h))(v)$.

(d) The action of $G(F)^2$ on $G(F)$, defined by the formula $(g,h)(x):=gxh^{-1}$ gives to $\C{H}(G(F))$
a structure of an $\C{H}(G(F)^2)$-module. Moreover, the map $z\mapsto z_{\C{H}}$ defines an algebra isomorphism
$Z_G\isom\End_{\C{H}(G(F)^2)}(\C{H}(G(F)))$.


\end{Emp}

\begin{Emp} \label{E:sbc}
{\bf The stable Bernstein center.}
(a) Each $z\in Z_G$ defines an
endomorphism $z_{reg}$ of the regular representation
$C_c^{\infty}(G(F))$ of $G(F)\times G(F)$. Hence $z\in Z_G$ gives rise to an invariant
distribution $\nu_z\in \Dist^{G(F)}(G(F))$ such that
$\nu_z(\phi)=z_{reg}(\iota^*(\phi))(1)$ for all $\phi\in C_c^{\infty}(G(F))$, where $\iota:G(F)\to G(F)$ is the map $g\mapsto g^{-1}$.

(b) For each $z\in Z_G$ the invariant distribution
$\nu_z\in \Dist^{G(F)}(G(F))$ can be characterized by the condition that $\nu_z\ast h=z_{\C{H}}(h)$
for every $h\in\C{H}(G(F))$. Moreover, the map $z\mapsto\nu_z$ identifies
$ Z_G$ with the set of all $\nu\in\Dist^{G(F)}(G(F))$ such that $\nu\ast h\in \C{H}(G(F))$
for every $h\in \C{H}(G(F))$.

(c) We define the {\em stable
center} of $G(F)$ to be the linear subspace
$ Z^{st}_{G}\subset Z_G$ consisting of all of $z\in  Z_G$ such that the
invariant distribution $\nu_z$ on $G(F)$ is stable.
\end{Emp}

\begin{Emp} \label{E:scc}
{\bf The stable center conjecture.} (a) The subspace $ Z^{st}_{G}\subset Z_G$ is a subalgebra.

(b) There exists a bijection $\chi\mapsto\ov{\la}_{\chi}$ between characters $\Hom( Z^{st}_{G},\B{C})$ and the set of
$\check{G}$-conjugacy classes of Frobenius semi-simple continuous homomorphisms $W_F\to\check{G}$.
\end{Emp}

\begin{Emp}
{\bf Remark.} The Lie algebra analogue of conjecture \re{scc} (a) follows from a theorem
of Waldspurger \cite{W} which says that the space of stable distribution on the Lie
algebra $\C{G}$ of $G$ is invariant under the Fourier transform.

Namely, let $\C{H}(\C{G}(F))$ be the space of smooth measures with compact support. Define
the Bernstein center $Z_\C{G}$ of $\C{G}(F)$ to be the space of invariant distributions
$\nu\in\Dist^{G(F)}(\C{G}(F))$ such that $\nu\ast h\in\C{H}(\C{G}(F))$ for all $h\in \C{H}(\C{G}(F))$,
and let $Z^{st}_\C{G}\subset Z_\C{G}$ be the subspace of stable distributions.

To see that $Z^{st}_\C{G}\subset Z_\C{G}$ is a subalgebra, recall that the Fourier transform converts
a convolution of measures into a product of functions. Therefore $\nu\in\Dist^{G(F)}(\C{G}(F))$
belongs to $Z_{\C{G}}$ if and only if its Fourier transform $\C{F}(\nu)$ is locally constant.
By the theorem of Waldspurger, $\nu\in{Z}_{\C{G}}$ is stable if and only if $\C{F}(\nu)$ is constant on
every stable orbit in $\C{G}(F)^{rss}$. Now the assertion follows the obvious remark that the product of constant
functions is constant.
\end{Emp}

\begin{Emp} \label{E:llc}
{\bf Relation to the local Langlands conjecture.}

(a) Recall that {\em the local Langlands conjecture} asserts the
existence of a decomposition $\Irr(G)=\sqcup_{\la}\Pi_{\la}$,
where $\la$ runs over the set of  Langlands parameters $W'_F\to \check{G}$, and
$\Pi_{\la}$ is a finite set, called the $L$-packet, corresponding
to $\la$.

(b) Assume that the decomposition $\Irr(G)=\sqcup_{\la}\Pi_{\la}$ from (a) is known.
Then conjecture  \re{scc} (a) has a more explicit form, saying that
$ Z^{st}_{G}\subset Z_G$ consists of all $z\in  Z_G$ such that the function
$f_z:\Irr(G)\to\B{C}$ is constant on each $L$-packet.

(c) Moreover, the bijection of \re{scc} (b) is supposed to be compatible with the decomposition of (a), that is,
for each $\la:W'_F\to\check{G}$, $\pi\in\Pi_{\la}$  and $\chi\in \Hom( Z^{st}_{G},\B{C})$
such that $\ov{\la}_{\chi}=\la|_{W_F}$ we have an equality $\chi(z)=f_z(\pi)$ for each $z\in Z_G^{st}$.
\end{Emp}






\begin{Emp} \label{E:decomp}
{\bf The decomposition of $R(G)$.}

(a) Let $R(G)^0\subset R(G)$
(resp. $R(G)^{>0}\subset R(G)$) be the full subcategory consisting of
representations $\pi\in R(G)$, all of whose irreducible
subquotients have depth zero (resp. positive depth) (see
\cite{MP,MP2}). Set $ Z_G^0:=Z(R(G)^0)$ and
$ Z_G^{>0}:=Z(R(G)^{>0})$.

(b) It follows from results of Bernstein \cite{Ber} and Moy-Prasad
\cite{MP,MP2} that the category $R(G)$ decomposes as a direct sum
$R(G)=R(G)^0\oplus R(G)^{>0}$. Therefore the Bernstein center
$Z_G$ decomposes as a direct sum $Z_G=Z^0_G\oplus Z^{>0}_G$. Similarly,
the set $\Irr(G)$ decomposes as a disjoint union $\Irr(G)=\Irr(G)^0\sqcup\Irr(G)^{>0}$.

(c) Explicitly, $Z_G^0\subset Z_G$ consists of all $z\in Z_G$ such that $z|_{\pi}=0$ for all
$\pi\in R(G)^{>0}$. Hence the unit $z^0\in Z_G^0\subset Z_G$ is the projector to the depth zero
spectrum. In particular, $z^0$ can be characterized by the condition that $f_{z^0}(\pi)=1$ for all
$\pi\in\Irr(G)^0$ and $f_{z^0}(\pi)=0$ for all $\pi\in\Irr(G)^{>0}$. We also set
$Z^{st,0}_G:= Z_G^0\cap Z^{st}_G$.
\end{Emp}

\begin{Emp} \label{E:depth0rep}
{\bf Depth zero representations.}
(a) It follows from results of Bernstein \cite{Ber}, Moy-Prasad
\cite{MP,MP2} and Deligne--Lusztig \cite{DL} that the category
$R(G)^0$ further decomposes as a direct sum
$R(G)^0=\oplus_{\theta}R(G)^{\theta}$, indexed by the set of
conjugacy classes of semisimple elements $\theta\in\check{G}$ such that
$\theta^q$ is conjugate to $\theta$.

(b) The decomposition of (a) implies a decomposition
$ Z_G^0=\oplus_{\theta} Z_G^{\theta}$ of the Bernstein
center. Each $ Z_G^{\theta}$ is a unital subalgebra, and we set
$Z^{st,\theta}_G:= Z_G^{\theta}\cap Z^{st}_G$.

(c) The decomposition from (a) induces a decomposition
$\Irr(G)^0=\sqcup_{\theta}\Irr(G)^{\theta}$.
In particular, to every
$\pi\in\Irr(G)^0$ one can associate a semisimple conjugacy class
$\theta(\pi)\in\check{G}$ such that $\theta(\pi)^q$ is conjugate to
$\theta(\pi)$ and $\pi\in\Irr(G)^{\theta(\pi)}$.
\end{Emp}

\begin{Emp} \label{E:depth0}
{\bf The depth zero local Langlands correspondence.}
(a) One expects that the local
Langlands correspondence preserves depth. In
particular, a representation $\pi\in\Irr(G)$ is of depth zero if
and only if the corresponding Langlands parameter
$\la:W'_F\to\check{G}$ is tamely ramified, that is,
trivial on the wild inertia subgroup of $W_F$.

(b) We choose a lift $\Fr\in W_F$ of the Frobenius element. This
choice defines a bijection between the set of tamely ramified Langlands
parameters $\la:W'_F\to\check{G}$ and the set of conjugacy classes of pairs
$(s,u)\in\check{G}$ such that $s\in\check{G}$ is semisimple and
$sus^{-1}=u^q$.

(c) To summarize, the depth zero local Langlands correspondence predicts that
$\Irr(G)^0$ has a decomposition $\Irr(G)^0=\sqcup_{\la}\Pi_{(s,u)}$, parameterized
by pairs $(s,u)$ as in (b). In particular, to every
$\pi\in\Irr(G)^0$, one can associate a pair
$(s,u)=(s(\pi),u(\pi))$ (defined up to conjugacy) such that $\pi\in\Pi_{(s,u)}$.
Moreover, the semisimple conjugacy class $\theta(\pi)$ from \re{depth0rep} (c)
is supposed to coincide with the semisimple part of $u(\pi)$.
\end{Emp}

\begin{Emp} \label{E:unipotent}
{\bf Known cases.} In the case when the group $G$ is adjoint,
Lusztig \cite{Lu1,Lu3} parameterized the set $\Irr(G)^1$, (that is,
$\Irr(G)^{\theta}$ with $\theta=1$) of
irreducible unipotent representations, thus verifying a
(more refined version of) the Langlands conjecture in this case. In
particular, for every $\pi\in\Irr(G)^1$, Lusztig associated a
semisimple conjugacy class $s(\pi)\in\check{G}$.

Taking into account recent works of Lusztig and others, it looks like the decomposition
\re{depth0} (c) is within reach. Therefore we assume from now on that the pair
$(s(\pi),\theta(\pi))\in\check{G}$ is defined for every $\pi\in\Irr(G)^0$.
\end{Emp}

Now we can restate the stable center conjecture for depth zero
representations in a more precise form.

\begin{Emp} \label{E:ssc0}
{\bf The depth zero stable center conjecture.}

(a) The subspace $ Z^{st,0}_G\subset Z^{0}_G$ is a unital
subalgebra. In particular, the projector $z^0\in  Z_G$ to the depth zero
spectrum is stable.

(b) An element $z\in Z^{0}_G$ belongs to $ Z^{st,0}_G$ if and
only if $f_z(\pi')=f_z(\pi'')$ for all $\pi',\pi''\in\Irr(G)^0$
such that $s(\pi')=s(\pi'')$ and $\theta(\pi')=\theta(\pi'')$.

(c) There exists a unique bijection $(s,\theta)\mapsto\chi_{(s,\theta)}$ between the set of $\check{G}$-conjugacy classes of pairs
of semisimple elements $(s,\theta)\in\check{G}$ such that $s\theta s^{-1}=\theta^q$ and characters
$\Hom( Z^{st,0}_G,\B{C})$ such that for every $\pi\in\Irr(G)^0$ and $z\in Z^{st,0}_G$ we have an equality
$f_z(\pi)=\chi_{(s(\pi),\theta(\pi))}(z)$.
\end{Emp}


\subsection{Categorification of the Bernstein center} \label{SS:conj1}
In this subsection we are going to construct the categorification of the Bernstein center, the complete details of which will
appear in the forthcoming work \cite{BKV2}. As a first step, we need to upgrade all constructions from previous sections from derived categories
to the corresponding $\infty$-categories.

\begin{Emp} \label{E:infcons}
{\bf Derived $\infty$-categories of constructible sheaves.}

(a) Recall that to every $X\in \Var_k$ or, more generally, $X\in\Art_k$ one can associate the $\infty$-category
$\C{D}(X)$, whose homotopy category is $D(X)$. Moreover, to every morphism $f:X\to Y$ in $\Art_k$, one can associate functors
$f^*,f_*,f^!,f_!$ between $\C{D}(\cdot)$'s, lifting the corresponding functors between $D(\cdot)$'s
(see \cite{LZ1,LZ2}).

Notice that functor
$f^!:\C{D}(Y)\to \C{D}(X)$ (resp. $f^*$) is fully faithful if and only if the corresponding functor $D(Y)\to D(X)$ is fully faithful. Indeed, both conditions are equivalent to the fact that the counit map $f_!f^!\C{F}\to\C{F}$ is an isomorphism in $D(X)$ for all
$\C{F}\in D(X)$.

(b) For every $X\in \ASch_k$ or, more generally, $X\in \ASt_k$, we define $\C{M}(X)$ to the (homotopy) colimit $\colim^!_{(X\to V)\in (X/\cdot)^{op}}\C{D}(V)$, and similarly for $\C{D}(X)$.

As in \re{adm} (b), we can take the colimit over $(X\to V)\in (X/\cdot)^{un}$. In this case, all the transition functors $D(V)\to D(V')$ are fully faithful, thus functors $\C{D}(V)\to\C{D}(V')$ are fully faithful as well (by (a)). It follows that the homotopy category of $\C{M}(X)$ (resp. $\C{D}(X)$) is $M(X)$ (resp. $D(X)$), and that each functor $\C{D}(V)\to\C{M}(X)$ (resp. $\C{D}(V)\to\C{D}(X)$) is fully faithful.

(c) Recall that for every fp-closed embedding $i:X\to Y$ in $\ASch_k$ or $\ASt_k$, the induced map
$i_*:M(X)\to M(Y)$ (resp. $i_*:D(X)\to D(Y)$) is fully faithful. Thus we can mimic \re{consind} and form an $\infty$-category $\C{M}(X)$ (resp. $\C{D}(X)$) for each $X\in \AISch_k$ or $X\in \AISt_k$, whose homotopy category is $M(X)$ (resp. $D(X)$).
\end{Emp}

\begin{Emp} \label{E:naive}
{\bf Naive definition of the categorical Bernstein  center.}

(a) Recall that $LG$ is an admissible ind-scheme. Therefore, by \re{infcons}, we can associate
to it an $\infty$-category $\C{M}(LG)$. Moreover, upgrading the construction of \re{conv}, we get that
$\C{M}(LG)$ has a structure a module $\infty$-category over a monoidal $\infty$-category
$\C{M}((LG)^2)$.

(b) Naively, we would like  to define $\C{Z}(LG)$ to be  $\infty$-category
$\End_{\C{M}((LG)^2)} \C{M}(LG)$ of endomorphisms of  $\C{M}(LG)$, viewed as a module category over
$\C{M}((LG)^2)$. However, in the $\ell$-adic setting we are working in, this definition seems to be wrong, at least a priori. Namely, we would like to consider not abstract endomorphisms of $\infty$-categories, but only those,
"coming from geometry" (see remark \re{BN} below). More formally, instead of using the $(\infty,2)$-category of stable $\infty$-categories,
we have to use an $\infty$-version of the geometric $2$-category, constructed in \re{2cat}.

(c) To make our exposition both conceptual and elementary, we are going to define first
the construction of $\C{Z}(LG)$ using the framework of monoidal $(\infty,2)$-categories, which in currently not well-documented in the literature,
and then to write down an explicit formula, which uses only the framework of $\infty$-categories.
\end{Emp}

\begin{Emp} \label{E:frame}
{\bf General framework.}

(a) Let $\frak{C}$ be an $(\infty,2)$-category. Then to every object $X$ of $\frak{C}$ one can associate the monoidal $\infty$-category
$\End_{\frak{C}}(X)=\Hom_{\frak{C}}(X,X)$.

(b) Assume now that $\frak{C}$ is monoidal. In particular, for every two objects $X,Y$ of $\frak{C}$ we can consider their tensor product
$X\otimes Y$. Let $S\in \frak{C}$ be a semigroup object in $\frak{C}$, that is, we are given a product map $S\otimes S\to S$ together
with various compatibilities between maps $S^{\otimes n}\to S$. Let $X\in\frak{C}$ be an $S$-module, that is, we are given an
action map $S\otimes X\to X$ together with various compatibilities between maps $S^{\otimes n}\otimes X\to X$.

(c) To the data of (b) one can associate the monoidal $\infty$-category $\End^S_{\frak{C}}(X)$ of endomorphisms of $X$, viewed as an $S$-module.
As an $\infty$-category, $\End^S_{\frak{C}}(X)$ is the (homotopy) limit of the following semi-simplicial diagram
\begin{equation} \label{Eq:end}
\Hom_{\frak{C}}(X,X)\rightrightarrows\Hom_{\frak{C}}(S\otimes X,X) \overset{\lra}{\rightrightarrows}\End_{\frak{C}}(S^{\otimes 2}\otimes X,X)\ldots
\end{equation}

(d) Note that for every object $X$ of $\frak{C}$, there is an equivalence $s:X\otimes X\to X\otimes X$, which interchange the factors.  In particular, for every semigroup object $S$, we can form another semigroup object $S^{op}$, which equals to $S$ as an object but the product $S^{op}\otimes S^{op}\to S^{op}$ differs from $S\otimes S\to S$ by $s$.
Then $S\otimes S^{op}$ has a natural structure of a semigroup object of $\frak{C}$, while $S$ has a natural structure of an $S\otimes S^{op}$-module.

The construction of (c) allows us to define  the categorical center of $S$  as the monoidal $(\infty,1)$-category
$\C{Z}_{\frak{C}}(S):=\End^{S\otimes S^{op}}_{\frak{C}}(S)$.
\end{Emp}

\begin{Emp} \label{E:geomcat}
{\bf The geometric $(\infty,2)$-category.}
We claim that one can upgrade the constructions from subsection \ref{SS:2cat} to the $\infty$-setting.

(a) First, mimicking \re{half} and arguing as in \re{infcons}, to every $X,Y\in\AISt_k$ one can associate the $\infty$-category $\C{M}_Y(X\times Y)$,
whose homotopy category is $M_Y(X\times Y)$. Next, we define $\C{Hom}_{geom}(X,Y)$ to be the full $\infty$-subcategory of
$\Pro \C{M}_Y(X\times Y)$, whose objects are $\C{B}\in \Pro \C{M}_Y(X\times Y)$ such that
$\C{B}(\C{F})\in \C{M}(Y)\subset \Pro\C{M}(Y)$ for all $\C{F}\in\C{M}(X)$.
Finally, we denote by $\C{M}_{geom}$ the $(\infty,2)$-category, whose objects are admissible ind-stacks $X\in \AISt_k$, and
for every $X,Y\in\AISt_k$, the $\infty$-category $\Hom_{\C{M}_{geom}}(X,Y)$ is $\C{Hom}_{geom}(X,Y)$.

(b) As in \re{corr}, we have a functor of $(\infty,2)$-categories from the $2$-category of correspondences $\Corr_k$ to $\C{M}_{geom}$.
Notice that the Cartesian product on $\AISt_k$ makes the categories $\Corr_k,\C{M}_{geom}$ and the functor $\Corr_k\to\C{M}_{geom}$ monoidal.
In particular, if $S$ is a semigroup object in $\Corr_k$, then  $S$ is also a semigroup object in
$\C{M}_{geom}$, hence we can consider the monoidal $\infty$-category $\C{Z}_{\C{M}_{geom}}(S)$ (see
\re{frame} (d)).
\end{Emp}

\begin{Emp} \label{E:catcent}
{\bf The categorical Bernstein center.}

(a) Note that $LG$ is a semigroup object in $\Corr_k$  with respect to the product map. Hence by
\re{geomcat} (b) we can consider the monoidal $\infty$-category $\C{Z}_{\C{M}_{geom}}(LG)$.
It is a categorical counterpart of the Bernstein center $Z_G$. As an $\infty$-category, $\C{Z}(LG)$ is a (homotopy) limit of the diagram
\begin{equation} \label{Eq:catbc}
\C{Hom}_{geom}(LG,LG)\rightrightarrows \C{Hom}_{geom}((LG)^2\times LG,LG) \overset{\lra}{\rightrightarrows}
\C{Hom}_{geom}((LG)^4\times LG,LG)\ldots.
\end{equation}

(b) We have a natural forgetful functor $\C{Z}(LG)\to \C{Hom}_{geom}(LG,LG)\to\End\C{M}(LG)$.
\end{Emp}

\begin{Emp} \label{E:BN}
{\bf Remarks.} (a) We have a natural functor from the categorical center $\C{Z}(LG)$ to the naive categorical center
$\End_{\C{M}((LG)^2)}\C{M}(LG)$, but we don't know whether this functor is an equivalence. On the other hand, as it was shown by Ben-Zvi and Nadler in \cite{BN}, an analogous finite-dimensional result for $D$-modules is true.

(b) In principle, we could carry out the construction of the center using the $2$-category $M_{geom}$ instead of its infinity version.
Namely, we could define the monoidal category $Z(LG):=\End^{(LG)^2}_{M_{geom}}(LG)$. Then we have a  natural monoidal functor
$\Ho\C{Z}(LG)\to Z(LG)$, where $\Ho\C{Z}(LG)$ denotes the homotopy category of $\C{Z}(LG)$, which at  least a priori is not an equivalence.
Moreover, it is even unclear whether $Z(LG)$ is a triangulated category.

The definition of the categorical Bernstein center is the first place in this work, where the framework of stable $\infty$-categories
is essential.
\end{Emp}



\begin{Emp} \label{E:catinvdist}
{\bf A categorification of a space of invariant distributions}.

(a) For each $\P\in\Par$, we can form the monoidal $\infty$-category $\C{M}(\frac{LG}{\P})$ (compare \re{equivha}) and the
pro-category $\Pro\C{M}(\frac{LG}{\P})$, which is also monoidal.  Moreover, for every $\P\subset \Q$ in $\Par$, the $!$-pullback induces a monoidal functor $\C{M}(\frac{LG}{\Q})\to \C{M}(\frac{LG}{\P})$, hence a monoidal functor $\Pro\C{M}(\frac{LG}{\Q})\to \Pro\C{M}(\frac{LG}{\P})$.

(b) We define a monoidal $\infty$-category $\Pro\C{M}(\frac{LG}{LG})$ to be the homotopy limit $\lim_{\P}\Pro\C{M}(\frac{LG}{\P})$. It can be thought as a categorical analog of the space of invariant distributions $\Dist^{G(F)}(G(F))$.
Notice that we had to pass to pro-categories before forming a limit over $\P$,
because there are no non-zero smooth invariant distribution with compact support.

(c) The $!$-pullbacks $\C{M}(\frac{LG}{\P})\to \C{M}(LG)$ for all $\P\in\Par$ give rise to a forgetful functor $\om:\Pro\C{M}(\frac{LG}{LG})\to \Pro\C{M}(LG)$,
which is categorification of the inclusion
$\Dist^{G(F)}(G(F))\hra \Dist(G(F))$. In particular, for every $\C{B}\in \Pro\C{M}(\frac{LG}{LG})$
and $\C{F}\in\Pro\C{M}(LG)$ we can form the convolution $\C{B}\ast\C{F}:=\om(\C{B})\ast\C{F}\in\Pro\C{M}(LG)$.

(d) Denote by  $\Pro\C{M}(\frac{LG}{LG})^{fin}\subset \Pro\C{M}(\frac{LG}{LG})$ the full $\infty$-subcategory with objects
$\C{B}\in \Pro\C{M}(\frac{LG}{LG})$ such that $\C{B}\ast\C{F}\in\C{M}(LG)\subset \Pro\C{M}(LG)$ for each $\C{F}\in\C{M}(LG)$.
\end{Emp}

\begin{Emp} \label{E:unitelt}
{\bf Remark.} Note that though monoidal $\infty$-categories $\C{M}(LG)$ and $\C{M}(\frac{LG}{\P})$ are non-unital, the pro-categories $\Pro\C{M}(LG), \Pro\C{M}(\frac{LG}{\P})$ and $\Pro\C{M}(\frac{LG}{LG})$ are unital.

\end{Emp}

\begin{Emp} \label{E:evalmap}
{\bf The evaluation functor.} (a) By \rt{weyl} (a) below, for each $\P\in\Par$ and $n\in\B{N}$, the evaluation functor $\ev_{\dt_{\P_n^+}}:\C{Z}(LG)\to\C{M}(LG):\C{B}\mapsto \C{B}(\dt_{\P_n^+})$ has a natural lift
$\ev^{\P}_{\dt_{\P_n^+}}:\C{Z}(LG)\to\C{M}(\frac{LG}{\P})$.

(b) By the naturality of the lift in (a), the functors $\{\ev^{\P}_{\dt_{\P_n^+}}\}_n$ form a projective system, so we can form a limit
$\ev^{\P}:=\lim_n\ev^{\P}_{\dt_{\P_n^+}}:\C{Z}(LG)\to\Pro\C{M}(\frac{LG}{\P})$.

(c) The evaluation functors $\{\ev^{\P}\}_{\P\in\Par}$ from (b) are compatible with the $!$-pullbacks
$\Pro\C{M}(\frac{LG}{\Q})\to \Pro\C{M}(\frac{LG}{\P})$. Hence we can define a functor
\[
\ev=\{\ev^{\P}\}_{\P}:\C{Z}(LG)\to\Pro\C{M}(\frac{LG}{LG}).
\]
\end{Emp}

The following result is a categorical analog of \re{sbc} (b).

\begin{Thm'} \label{T:emb}
The evaluation functor $\ev$ from \re{evalmap} induces an equivalence of monoidal $\infty$-categories $\C{Z}(LG)\isom\Pro\C{M}(\frac{LG}{LG})^{fin}$.
\end{Thm'}

\begin{Emp} \label{E:thm'}
{\bf Remark.} Here and later we write {\em "Theorem"} instead of {\em Theorem} to indicate that the result appears without a complete proof, and that
the details will appear in \cite{BKV2}.
\end{Emp}

\begin{proof}[Brief sketch of the proof]
Observe that for every $\C{B}\in\C{Z}(LG)$ and $\C{F}\in\C{M}(LG)$ we have
$\ev(\C{B})\ast\C{F}\cong\C{B}(\C{F})\in\C{M}(LG)\subset\Pro\C{M}(LG)$, thus $\ev$ induces a functor
$\C{Z}(LG)\to \Pro\C{M}(\frac{LG}{LG})^{fin}$. Conversely, every  $\C{X}\in \Pro\C{M}(\frac{LG}{LG})^{fin}$
defines an endomorphism $\C{B}$ of $\C{M}(LG)$, given by the formula
$\C{B}(\C{F}):=\C{X}\ast\C{F}$ for each $\C{F}\in\C{M}(LG)$. Now the fact that
$\C{B}$ can be upgraded to an object of $\C{Z}(LG)$ follows from the fact that $\C{X}$ is $\Ad\P$-equivariant for all $\P\in\Par$
(compare the proof of \rt{center} (a) below).
\end{proof}

\subsection{Categorification of the embedding $Z^0_G\hra Z_G$} \label{SS:emb}
In this subsection we construct a categorical analog of the projection $Z_G\to Z_G^0$ and the embedding $Z^0_G\to Z_G$
(see \cite{BKV2} for details). 

\begin{Emp} \label{E:centpn}
{\bf The categorical analog of depth $n$ center.} Fix $\P\in\Par$ and $n\in \B{N}$.

(a) The multiplication map $(LG)^2\to LG$ induces an admissible correspondence
\[
(\P_n^+\bs LG/\P_n^+)\times (\P_n^+\bs LG/\P_n^+)\overset{\pr}{\lla} (\P_n^+\bs LG)\times (LG/\P_n^+)\overset{\ov{m}}{\lra} \P_n^+\bs LG/\P_n^+,
\]
(see \re{corr}), which gives to the admissible ind-stack $\P_n^+\bs LG/\P_n^+$ a structure of a semigroup object in $\Corr_k$. As in \re{geomcat} (b), the construction of \re{frame} defines the monoidal $\infty$-category $\C{Z}_{\P_n^+}(LG):=\C{Z}_{\C{M}_{geom}}(\P_n^+\bs LG/\P_n^+)$.

(b) As an $\infty$-category, $\C{Z}_{\P_n^+}(LG)$ is a homotopy limit of a diagram similar to \form{catbc},
where $LG$ is replaced by $\P_n^+\bs LG/\P_n^+$ at all places.

(c) By definition, we have  a forgetful functor
$\C{Z}_{\P_n^+}(LG)\to \End\C{M}_{\P_n^+}(LG)$, hence we have an evaluation functor
$\ev_{\P_n^+}:\C{Z}_{\P_n^+}(LG)\to\C{M}(LG):\C{B}\mapsto\C{B}(\dt_{\P_n^+})$.
\end{Emp}

\begin{Emp} \label{E:remcentpn}
{\bf Remark.} Note that $\C{M}_{(\P_n^+)^2}((LG)^2)=\C{M}((\P_n^+\bs LG/\P_n^+)^2)$ is a
monoidal $\infty$-category, acting on the $\infty$-category $\C{M}_{\P_n^+}(LG)=\C{M}(\P_n^+\bs LG/\P_n^+)$ (compare \re{biinv} (e)).
Then $\C{Z}_{\P_n^+}(LG)$ can be thought as the $\infty$-category of "geometric" endomorphisms of $\C{M}_{\P_n^+}(LG)$,
viewed as a module category over $\C{M}_{(\P_n^+)^2}((LG)^2)$. Again, "geometric" means "defined using the geometric $(\infty,2)$-category $\C{M}_{geom}$".
\end{Emp}

\begin{Emp} \label{E:restr}
{\bf Restriction functors.}

(a) Recall that $\C{M}_{\I^+}(LG)$ is a full $\infty$-subcategory of $\C{M}(LG)$, spanned by objects
$\C{F}\in \C{M}(LG)$ such that $\dt_{(\I^+)^2}(\C{F})\cong\C{F}$ (compare \re{biinv} (c),(d)).

Since $\dt_{(\I^+)^2}\in \C{M}_{(\P^+)^2}((LG)^2)$ for each $\P\in\Par$,  every
$\C{B}\in \C{Z}_{\P^+}(LG)$ induces an endomorphism of $\C{M}_{\I^+}(LG)$.
Furthermore, the restriction map
\[\C{B}\mapsto\C{B}|_{\C{M}_{\I^+}(LG)}:\C{Z}_{\P^+}(LG)\to \End \C{M}_{\I^+}(LG)\]
naturally upgrades to a monoidal functor
$\C{R}_{\P}:\C{Z}_{\P^+}(LG)\to\C{Z}_{\I^+}(LG)$.

(b) Similarly, the restriction map $\C{B}\mapsto\C{B}|_{\C{M}_{\I^+}(LG)}:\C{Z}(LG)\to \End \C{M}_{\I^+}(LG)$ naturally upgrades to a monoidal functor
$\C{R}:\C{Z}(LG)\to\C{Z}_{\I^+}(LG)$.
\end{Emp}

\begin{Thm'} \label{T:mconj}
(a) For every $\P\in\Par$, the functor
$\C{R}_{\P}:\C{Z}_{\I^+}(LG)\to\C{Z}_{\P^+}(LG)$ is an equivalence of monoidal $\infty$-categories.

(b) The restriction functor $\C{R}:\C{Z}(LG)\to\C{Z}_{\I^+}(LG)$ has a right adjoint
$\C{A}$. Moreover, the functor $\C{A}$ is monoidal and fully faithful.
\end{Thm'}

\begin{Emp}
{\bf Remark.} \rtt{mconj} (a) implies that each $\C{Z}_{\P^+}(LG)$ can be thought as a categorical analog of
the depth zero Bernstein center $Z_G^0$. Moreover, the restriction $\C{R}$ is a categorical analog of the projection
$Z_G\to Z^0_G$, while its adjoint $\C{A}$ is a categorical analog of the
embedding $\C{Z}^0_G\hra  Z_G$.
\end{Emp}

Now we explain the construction of the inverse functor $\C{A}_{\P}$ of $\C{R}_{\P}$.


\begin{Thm'} \label{T:weyl}
(a)  For every $\P\in \Par$ and $n\in\B{N}$, the evaluation functor $\ev_{\dt_{\P_n^+}}$ from \re{centpn} (c) has a natural lift to a monoidal functor
$\ev^{\P}_{\dt_{\P_n^+}}:\C{Z}_{\P_n^+}(LG)\to \C{M}(\frac{LG}{\P})$.

(b) For every $\P\in \Par$, the composition
\[
\Av^{\P/\I}_{\P}\circ\ev^\I_{\dt_{\I^+}}:\C{Z}_{\I^+}(LG)\to \C{M}(\frac{LG}{\I})\to \C{M}(\frac{LG}{\P})
\]
is naturally equipped with an action of the finite Weyl group $W_{\P}\subset\wt{W}$.
\end{Thm'}

\begin{Emp}
{\bf Remark.} The $W_{\P}$-action in \rtt{weyl} (b) is induced from the $W_{\P}$-action on the
Springer sheaf $\C{S}\in D(L_{\P})$, which we normalize so that the sheaf of skew-invariants $\C{S}^{W_{\P},\sgn}$
is the $\dt$-sheaf $\dt_1\in D(L_{\P})$ at the identity.
\end{Emp}

\begin{Emp} \label{E:notaver}
{\bf Notation.} For $\P\in \Par$, we denote by
\[
\bar{\C{A}}_{\P}:=(\Av^{\P/\I}_{\P}\circ\ev^\I_{\dt_{\I^+}})^{W_{\P,\sgn}}:\C{Z}_{\I^+}(LG)\to
\C{M}(\frac{LG}{\P})
\]
the composition of $\Av^{\P/\I}_{\P}\circ\ev^\I_{\dt_{\I^+}}$ with the functor of
skew-invariants.
\end{Emp}

\begin{Thm'} \label{T:inverse}
There exists a natural functor $\C{A}_{\P}:\C{Z}_{\I^+}(LG)\to \C{Z}_{\P^+}(LG)$ such that
$\ev^{\P}_{\dt_{\P^+}}\circ\C{A}_{\P}\cong\bar{\C{A}}_{\P}$. Moreover, $\C{A}_{\P}$ is
the inverse of ${\C{R}}_{\P}$.
\end{Thm'}

Our next goal is to provide an explicit construction of the right adjoint $\C{A}$
of $\C{R}$. By \rtt{emb}, it suffices to construct a functor $\C{A}':\C{Z}_{\I^+}(LG)\to \Pro\C{M}(\frac{LG}{LG})$ and to show that
its image lies in $\Pro\C{M}(\frac{LG}{LG})^{fin}$.

\begin{Emp} \label{E:parah}
{\bf Notation.} (a) Let $\T$ be a category associated to the partially ordered set of non-empty closed $\I$-invariant subschemes $Y\subset \Fl$.
For every $Y\in\T$ we denote by $\wt{Y}\subset LG$ the preimage of $Y$.

(b) For $\P\in\Par$ and $Y\in\T$, we set $Y\P:=\wt{Y}\P/\I\subset \Fl$ and $Y_{\P}:=\wt{Y}\P/\P\subset \Fl_{\P}$.
\end{Emp}

\begin{Emp} \label{E:wpactions}
{\bf Construction of $\C{A}'$.}

(a) Fix $\Q\in\Par$. Recall that if $Y\in \T$ is $\Q$-invariant, then we have an averaging functor
$\Av^Y_{\Q}:\C{M}(\frac{LG}{\I})\to \C{M}(\frac{LG}{\Q})$ (see \re{aver} (b)). Moreover, functors $\{\Av^Y_{\Q}\}_Y$ form a projective system, so we can form a projective limit
\[
\Av^{\Fl}_{\Q}:=\lim_Y \Av^Y_{\Q}:\C{M}(\frac{LG}{\I})\to \Pro\C{M}(\frac{LG}{\Q}).
\]

(b) Functors $\{\Av^{\Fl}_{\Q}\}_{\Q\in\Par}$ are contravariant in $\Q$, thus they give rise to the
 functor
\[
\Av^{\Fl}_{LG}:=\{\Av^{\Fl}_{\Q}\}_{\Q\in\Par}:\C{M}(\frac{LG}{\I})\to \Pro\C{M}(\frac{LG}{LG}).
\]

(c) Fix $\P\in\Par$, and assume that  $Y\in \T$ from (a) is such that $\wt{Y}=\wt{Y}\P\subset LG$.  Then $\Av^Y_{\Q}$ decomposes as
\[
\Av^Y_{\Q}=\Av^{Y_{\P}}_{\Q}\circ \Av^{\P/\I}_{\P}:\C{M}(\frac{LG}{\I})\to \C{M}(\frac{LG}{\P})
\to \C{M}(\frac{LG}{\Q}).
\]
In particular, the composition
$\Av^Y_{\Q}\circ \ev_{\dt_{\I^+}}^{\I}:\C{Z}_{\I^+}(LG)\to \C{M}(\frac{LG}{\Q})$
is equipped with an action of $W_{\P}$ (by \rtt{weyl} (b)).

(d) By (c), the limit
\[
\Av^{\Fl}_{\Q}\circ \ev_{\dt_{\I^+}}^{\I}=\lim_Y(\Av^Y_{\Q}\circ \ev_{\dt_{\I^+}}^{\I}):\C{Z}_{\I^+}(LG)\to\Pro\C{M}(\frac{LG}{\Q})
\]
is equipped with an action of $W_{\P}$. Moreover,
the $W_{\P}$-actions are compatible with the $!$-pullbacks $\Pro\C{M}(\frac{LG}{\Q'})\to \Pro\C{M}(\frac{LG}{\Q})$ for $\Q\subset\Q'$
in $\Par$.

Taking the limit over the $\Q$'s, we get an action of $W_{\P}$ on the functor
\[
\Av^{\Fl}_{LG}\circ \ev_{\dt_{\I^+}}^{\I}:\C{Z}_{\I^+}(LG)\to \Pro\C{M}(\frac{LG}{LG}).
\]
Moreover, by the naturality assertion in \rtt{weyl} (b), for each $\P\subset \P'$, the $W_{\P}$- and $W_{\P'}$-actions on $\Av^{\Fl}_{LG}\circ \ev_{\dt_{\I^+}}^{\I}$ are compatible.

(e) Note that $\wt{W}$ is the "homotopy colimit" of $\{W_{\P}\}_{P\in\Par}$, that is, the classifying space $B(\wt{W})$ is the homotopy colimit of classifying spaces $\{B(W_{\P})\}_{P\in\Par}$. Therefore
the compatible system of $W_{\P}$-actions on $\Av^{\Fl}_{LG}\circ \ev_{\dt_{\I^+}}^{\I}$ define the action of $\wt{W}$.

(f) We define functor $\C{A}':=(\Av^{\Fl}_{LG}\circ \ev_{\dt_{\I^+}}^{\I})^{\wt{W},\sgn}$ to be the composition
\[
\C{A}':\C{Z}_{\I^+}(LG)\to \Pro\C{M}(\frac{LG}{LG})\to \Pro\C{M}(\frac{LG}{LG})
\]
of $\Av^{\Fl}_{LG}\circ \ev_{\dt_{\I^+}}^{\I}$ with the functor of
skew-invariants.
\end{Emp}

\begin{Emp} \label{E:rss}
{\bf Restriction to the regular semisimple locus.}
(a) For  $\gm\in LG$ and $n\in\B{N}$, let
$i_{\gm,n}$ be the inclusion $\gm \I_n^+\hra LG$. Then $i_{\gm,n}$ is finitely presented, and we denote by $i_{\gm,n}^*$ the
pullback functor $i_{\gm,n}^*:\C{M}(LG)\to \C{M}(\gm \I_n^+)$.

(b) We say that the restriction of $\C{F}\in\Pro\C{M}(LG)$ (resp. $\C{F}'\in\Pro\C{M}(\frac{LG}{LG})$)  to $LG^{rss}$ is {\em smooth}, if
for every $\gm\in LG^{rss}$ there exists $n\in\B{N}$ such that $i_{\gm,n}^*(\C{F})\in \C{M}(\gm I_n^+)$ (resp.
$i_{\gm,n}^*(\om\circ\C{F}')\in \C{M}(\gm I_n^+)$ (compare \re{catinvdist} (c)). 
\end{Emp}

\begin{Thm'} \label{T:finite}
For each $\C{B}\in\C{Z}(LG)$, the object $\C{A}'(\C{B})\in \Pro\C{M}(\frac{LG}{LG})$ belongs to
$\Pro\C{M}(\frac{LG}{LG})^{fin}$ and the restriction of $\C{A}'(\C{B})$ to $LG^{rss}$ is smooth.
\end{Thm'}

By "Theorems" \ref{T:emb} and \ref{T:finite}, functor $\C{A}'$ gives rise to a functor
$\C{A}:\C{Z}_{\I^+}(LG)\to\C{Z}(LG)$.

\begin{Thm'} \label{T:adj}
The functor $\C{A}:\C{Z}_{\I^+}(LG)\to\C{Z}(LG)$ is right adjoint to $\C{R}$. Moreover,
$\C{A}$ is fully faithful and monoidal.
\end{Thm'}


\begin{Emp} \label{E:strategy}
{\bf Strategy of the proof of \rtt{finite}.}
(a) Since $\wt{W}$ is the "homotopy colimit" of $\{W_{\P}\}_{P\in\Par}$ (compare \re{wpactions} (e)), the functor
$\C{A}'=(\Av^{\Fl}_{LG}\circ \ev_{\dt_{\I^+}}^{\I})^{\wt{W},\sgn}$ is naturally equivalent to
a limit $\lim_{\P\in\Par}(\Av^{\Fl}_{LG}\circ \ev_{\dt_{\I^+}}^{\I})^{W_{\P},\sgn}$.

(b) By (a) and arguments of \re{wpactions}, there exists a  functor of $\infty$-categories \[
\C{A}_{\star}^{\star}:\Par^{op}\times\T^{op}\times \C{Z}_{\I^+}(LG)\to \C{M}(LG)\] such that the restriction of $\C{A}_{\star}^{\star}$ to $(\P,Y)\in \Par\times\T$ is \[
\C{A}_{\P}^Y:=(\Av^{Y\P}\circ\ev_{\dt_{\I^+}}^\I)^{W_{\P},\sgn}:\C{Z}_{\I^+}(LG)\to \C{M}(LG),
\]
and
\[\om\circ\C{A}'\cong\lim_{\P\in\Par}\lim_{Y\in\T} \C{A}_{\P}^Y:\C{Z}_{\I^+}(LG)\to \Pro \C{M}(LG).\]

(c) Changing the order of limits, we conclude that $\om\circ\C{A}'\cong \lim_{Y\in\T}\lim_{\P\in\Par}\C{A}_{\P}^Y$.

(d) Since the $\infty$-category $\C{M}(LG)$ has finite limits, each
$\C{A}^Y:=\lim_{\P\in\Par}\C{A}_{\P}^Y$ is a functor $\C{Z}_{\I^+}(LG)\to \C{M}(LG)$. Hence the limit
 $\C{A}^{\star}:=\lim_{\P\in\Par}\C{A}^{\star}_{\P}$ is a functor  $\T^{op}\times \C{Z}_{\I^+}(LG)\to \C{M}(LG)$, while by
(c), we get $\om\circ\C{A}'\cong \lim_{Y\in\T}\C{A}^Y$.

(e) By (d), for every $\C{B}\in\C{Z}_{\I^+}(LG)$ and $\C{X}\in \C{M}(LG)$, we obtain
a functor of $\infty$-categories $\C{A}^{\star}(\C{B})\ast\C{X}:\T^{op}\to \C{M}(LG)$ with
$\C{A}'(\C{B})\ast\C{X}\cong\lim_Y \C{A}^Y(\C{B})\ast\C{X}$. In particular,
we get a projective system $\{\C{A}^Y(\C{B})\ast\C{X}\}_{Y\in\T}$ in the homotopy category $M(LG)$.

(f) Similarly, for $\C{B}\in\C{Z}_{\I^+}(LG)$, $\gm\in LG$ and $n\in\B{N}$,
we get a functor of $\infty$-categories $i_{\gm,n}^*(\C{A}^{\star}(\C{B})):\T^{op}\to \C{M}(\gm \I_n^+)$
such that $i^*_{\gm,n}(\om\circ\C{A}'(\C{B}))\cong \lim_Y i_{\gm,n}^*(\C{A}^Y(\C{B}))$. In particular, we get
a projective system $\{i_{\gm,n}^*(\C{A}^Y(\C{B}))\}_{Y\in\T}$ in $M(\gm \I_n^+)$.

(g) By (e) and (f), \rtt{finite} follows from the following more precise assertion.
\end{Emp}




\begin{Thm'} \label{T:mconj2}
(a) For every $\C{B}\in \C{Z}_{\I^+}(LG)$ and $\C{X}\in \C{M}(LG)$,
the projective system $\{\C{A}^Y(\C{B})\ast\C{X}\}_Y$ in $M(LG)$ stabilizes.

(b) For every $\gm\in LG^{rss}$, there exists $n\in\B{N}$ such that the projective system
$\{i_{\gm,n}^\ast\C{A}^Y(\C{B})\}_Y$ in $M(\gm I_n^+)$ stabilizes.
\end{Thm'}

\subsection{Application to the classical (stable) Bernstein center} \label{SS:conj2}

\begin{Emp} \label{E:groth}
{\bf Grothendieck groups.} (a) For every triangulated category $C$,
we denote by $K_0(C)$ the Grothendieck group of $C$, tensored over $\qlbar$, and call it the $K$-group of $C$. In particular, to
each $X\in C$, we associate its class $\lan X\ran\in K_0(C)$. Notice that if $C$ is a monoidal category,
then $K_0(C)$ is a $\qlbar$-algebra.

(b) Every triangulated functor $\phi:C\to C'$ induced a map of $K$-groups
$\lan \phi\ran:K_0(C)\to K_0(C')$. Furthermore, $\lan \phi\ran$ is injective, if there exists a triangulated functor
$\phi':C'\to C$ such that $\phi'\circ\phi\cong\Id$.

(c) For $X\in\ASt_k$, we denote by $K_0(M(X))$ (resp. $K_0(D(X))$) the $K$-group
of $M(X)$ (resp. $D(X)$). Using (b) and \rl{ff}, we see that if $X\in\ASt_k$ has an admissible presentation $X\cong\lim_i X_i$, then $K_0(M(X))\cong \colim_i K_0(M(X_i))$ (resp. $K_0(D(X))\cong \colim_i K_0(D(X_i))$), and all the transition maps are injective.

Similarly, if  $X\in\AISt_k$ has a presentation $X\cong\colim_i X_i$, then we have isomorphisms
$K_0(M(X))\cong \colim_i K_0(M(X_i))$ (resp. $K_0(D(X))\cong \colim_i K_0(D(X_i))$), and all the transition maps are injective.
\end{Emp}

\begin{Emp} \label{E:kgpcenter}
{\bf Grothendieck group version of the center.}

(a) Recall that $M((LG)^2)$ is a monoidal category acting on $M(LG)$ (see \re{conv}).
Therefore $K_0(M((LG)^2))$ is a $\qlbar$-algebra, acting on $K_0(M(LG))$.
We denote by $\frak{Z}(LG)$ 
the algebra of endomorphisms of $K_0(M(LG))$, viewed as a $K_0(M((LG)^2))$-module.

(b) Notice that every $\C{B}\in \C{Z}(LG)$ defines a unique element
$\lan\C{B}\ran\in \frak{Z}(LG)$ such that $\lan \C{B}\ran(\lan\C{F}\ran)=\lan\C{B}(\C{F})\ran$
for every $\C{F}\in \C{M}(LG)$. Indeed, the functor $\C{F}\mapsto \C{B}(\C{F})$ preserves
finite limits in $\C{M}(LG)$, hence it preserves distinguished triangles in $M(LG)$. Thus, $\C{B}$ induces an endomorphism $\lan\C{B}\ran\in \End_{\qlbar} K_0(M(LG)):\lan\C{F}\ran\mapsto\lan\C{B}(\C{F})\ran$. Finally, using the fact that $\C{B}\in \C{Z}(LG)$, we have an isomorphism $\C{B}(\C{X}(\C{F}))\cong \C{X}(\C{B}(\C{F}))$ for every $\C{X}\in \C{M}((LG)^2)$, which implies that $\lan\C{B}\ran\in\frak{Z}(LG)$.

The map $\C{B}\mapsto\lan\C{B}\ran$ defines an algebra homomorphism
$K_0(\C{Z}(LG))\to \frak{Z}(LG)$, which at least a priori is not an isomorphism.
\end{Emp}

\begin{Emp} \label{E:setup}
{\bf Set-up.} (a) Assume now that $k$ is an algebraic closure of the finite
field $\fq$, and that $G$ has an $\fq$-structure. Then all the geometric
objects defined above have $\fq$-structures, and
we denote by $\C{D}^{\Fr}(LG)$, $\C{M}^{\Fr}(LG)$, $\C{Z}^{\Fr}(LG)$ etc. the
corresponding categories of Weil (Frobenius equivariant) objects.

(b) We set $F:=\fq((t))$, and choose a field isomorphism $\qlbar\isom\B{C}$.
For every $\P\in\Par$ and $n\in\B{N}$, we set $P:=\P(\fq)\subset G(F)$ and $P_n^+:=\P_n^+(\fq)\subset G(F)$.
\end{Emp}

\begin{Emp} \label{E:shfun}
{\bf The "sheaf-function correspondence" for loop groups.}

(a) Let $X\in \Var_k$ be defined over $\fq$. Then the classical sheaf-function
correspondence  associates to every $\C{F}\in D^{\Fr}(X)$ a function
$[\C{F}]:="\Tr(\Fr,\C{F})":X(\fq)\to\fqbar$.
Moreover, the correspondence $\C{F}\mapsto [\C{F}]$ defines a surjective homomorphism of vector spaces
$K_0(D^{\Fr}(X))\to \Fun(X(\fq),\qlbar)$.

(b) Note that correspondence $\C{F}\mapsto [\C{F}]$ from (a) commutes with $*$-pullbacks and $!$-pushforwards.
Therefore it extends to a similar correspondence for every $X\in \AISch_k$, defined over $\fq$.

(c) Consider the case $X:=LG$. Then $X(\fq)=G(F)$, and we claim that for every $\C{F}\in  D^{\Fr}(LG)$,
we have $[\C{F}]\in C_c^{\infty}(G(F))$. Indeed, every $\C{F}\in  D^{\Fr}(LG)$ comes from an object of some
$D^{\Fr}(LG^{\leq w}/\I_n^+)$, thus the corresponding function $[\C{F}]$ is supported on $G(F)^{\leq w}$ and is $I_n^+$-invariant from the right. Moreover, the map $\C{F}\mapsto [\C{F}]$ induces a surjective map $K_0(D^{\Fr}(LG))\to C_c^{\infty}(G(F))$,
which follows from the corresponding assertion  in (a).
\end{Emp}

\begin{Emp} \label{E:haarmeas}
{\bf Notation.} For every open compact subset $S\subset G(F)$, we denote by $\mu^S$ the Haar measure on
$G(F)$ normalized by the condition that $\int_{S}\mu^S=1$.
\end{Emp}

\begin{Emp} \label{E:shfun1}
{\bf The "sheaf-function correspondence" for measures.}

(a) For every $\C{F}\in  D^{\Fr}(LG)$ and $\C{X}\in  M^{\Fr}(LG)$, we can form the
tensor product $\C{F}\otimes\C{X}\in M^{\Fr}(LG)$, the $!$-pushforward
$\int_{LG}(\C{F}\otimes\C{X})\in D^{\Fr}(k)$, and the corresponding element $[\int_{LG}(\C{F}\otimes\C{X})]\in \qlbar$.

(b) We claim that for every $\C{X}\in  M^{\Fr}(LG)$, there exists a unique  $[\C{X}]\in\C{H}(G(F))$ such that $[\C{X}]([\C{F}])=[\int_{LG}(\C{F}\otimes\C{X})]$ for every $\C{F}\in  D^{\Fr}(LG)$.

\begin{proof}
The uniqueness follow from the fact that the map $K_0(D^{\Fr}(LG))\to C_c^{\infty}(G(F))$ is surjective
(see \re{shfun} (c)).

To show the existence, we choose $w\in\wt{W}$ and $n\in\B{N}$ such that $\C{X}$ come from an object
of $D^{\Fr}(LG^{\leq w}/\I_n^+)$. Let $[\C{X}]'$ be the corresponding function
$G(F)^{\leq w}/I_n^+\to\qlbar$, and consider element
$[\C{X}]_{n,w}:=[\C{X}]'\mu^{I_n^+}\in \C{H}(G(F))$ (see \re{haarmeas}).

We claim that $[\C{X}]_{n,w}([\C{F}])=[\int_{LG}(\C{F}\otimes\C{X})]$ for every $\C{F}\in  D^{\Fr}(LG)$.
Indeed, note that $[\C{X}]_{n,w}$ is independent of $n$ and $w$. Thus we can increase $w$ and $n$, if necessary,
thus assuming that $\C{F}$ comes from an object of $D^{\Fr}(LG^{\leq w}/\I_n^+)$, and let $[\C{F}]'$ be the corresponding function $G(F)^{\leq w}/I_n^+\to\qlbar$.

Then both $[\C{X}]_{n,w}([\C{F}])$ and $[\int_{LG}(\C{F}\otimes\C{X})]$ are equal to
$\int_{G(F)^{\leq w}/I_n^+}([\C{F}]'\cdot[\C{X}]')$, implying the assertion.
\end{proof}

(c) Note that the map $\C{X}\mapsto[\C{X}]$ defines an algebra homomorphism
$K_0(M^{\Fr}(LG))\to \C{H}(G(F))$. In particular, for every $\C{X},\C{Y}\in {M}^{\Fr}(LG)$, we have
$[\C{X}\ast\C{Y}]=[\C{X}]\ast[\C{Y}]$. Moreover, as in \re{shfun} (c), this homomorphism is surjective.

(d) Similarly, every element $\C{X}\in\wt{M}^{\Fr}(LG)$ defines a smooth measure  $[\C{X}]$ on $G(F)$.
Note that for every fp-closed subscheme $Y\subset LG$, defined over $\fq$, the  Haar measure
$\mu^Y$ (see \re{haarind} (c)) is Frobenius-equivariant, and the corresponding measure
$[\mu^Y]$ on $G(F)$ is $\mu^{Y(\fq)}$ (see \re{haarmeas}).
\end{Emp}

\begin{Emp} \label{E:shfun2}
{\bf The "sheaf-function correspondence" for the center.}

(a) Set $\frak{Z}^{\Fr}(LG):=\End_{K_0(M^{\Fr}((LG)^2))}K_0(M^{\Fr}(LG))$. As in \re{kgpcenter} (b),
there exists an algebra homomorphism $K_0(\C{Z}^{\Fr}(LG))\to\frak{Z}^{\Fr}(LG):\C{B}\mapsto\lan\C{B}\ran$
such that $\lan \C{B}\ran(\lan\C{F}\ran)=\lan\C{B}(\C{F})\ran$ for all $\C{F}\in \C{M}^{\Fr}(LG)$.

(b) We claim that for each $\lan\C{B}\ran\in \frak{Z}^{\Fr}(LG)$
there exists a unique element $[\C{B}]\in Z_{G}$ such that
$[\C{B}]([\C{F}])=[\lan\C{B}\ran(\lan\C{F}\ran)]$ for every
$\C{F}\in M^{\Fr}(LG)$.

\begin{proof}
First we claim that there exists a unique linear endomorphism of $\C{H}(G(F))$
given by the rule $[\C{F}]\mapsto [\lan\C{B}\ran(\lan\C{F}\ran)]$. The uniqueness follows from
the fact that the map $K_0(M^{\Fr}(LG))\to \C{H}(G(F))$ from \re{shfun1} (c) is surjective.

To show the existence and linearity, it suffices to show that  for
every tuple $\C{F}_1,\ldots,\C{F}_r\in{M}^{\Fr}(LG)$ and
$a_1,\ldots,a_r\in\qlbar$ such that $\sum_ia_i[\C{F}_i]=0$, we have
$\sum_ia_i[\lan\C{B}\ran(\lan\C{F}_i\ran)]=0$.


Choose $n\in\B{N}$ such that
$\C{F}_i\cong\dt_{\I_n^+}\ast\C{F}_i\cong (\dt_{\I_n^+}\pp\C{F}_i)(\dt_{\I_n^+})$ for all $i$ (see \re{biinv} (b),(c)).  Then
\[\lan\C{B}\ran(\lan\C{F}_i\ran)=\lan\C{B}\ran(\lan(\dt_{\I_n^+}\pp\C{F}_i)(\dt_{\I_n^+})\ran)
=\lan\dt_{\I_n^+}\pp\C{F}_i\ran(\lan\C{B}\ran(\lan\dt_{\I_n^+}\ran))=\lan\dt_{\I_n^+}\ran\ast\lan\C{B}\ran(\lan\dt_{\I_n^+}\ran)\ast\lan \C{F}_i\ran,\]
hence
$\sum_ia_i[\lan\C{B}\ran(\lan\C{F}_i\ran)]=[\dt_{\I_n^+}]\ast[\lan\C{B}\ran(\lan\dt_{\I_n^+}\ran)]\ast(\sum_ia_i[\C{F}_i])=0$.

Finally, since the homomorphism $K_0(M^{\Fr}((LG)^2))\to \C{H}(G(F)^2)$ is surjective
and $\C{B}\in \frak{Z}^{\Fr}(LG)$, the map $[\C{F}]\mapsto [\lan\C{B}\ran(\lan\C{F}\ran)]$ belongs to
$Z_{G}$.
\end{proof}

(c) Note that the map $\lan\C{B}\ran\to [\C{B}]$ from (b) defines an algebra homomorphism $\frak{Z}^{\Fr}(LG)\to Z_{G}$. Composing it with the algebra homomorphism from (a), we get
an algebra homomorphism $K_0(\C{Z}^{\Fr}(LG))\to Z_{G}:\C{B}\mapsto [\C{B}]$.

(d) Note that the functor $\C{A}$ from \rtt{mconj} is
Frobenius-equivariant, hence it upgrades to a monoidal functor
$\C{A}^{\Fr}:\C{Z}^{\Fr}_{\I^+}(LG)\to\C{Z}^{\Fr}(LG)$.
By (c), the functor $\C{A}^{\Fr}$ define an algebra homomorphism
$[\C{A}]:=[\C{A}^{\Fr}]:K_0(\C{Z}^{\Fr}_{\I^+}(LG))\to Z_{G}$.
We will show in \rt{conjunit} that $[\C{A}](\lan\Id\ran)=z^0$. Therefore
$[\C{A}]$ induces an algebra homomorphism $[\C{A}]:K_0(\C{Z}_{\I^+}^{\Fr}(LG))\to Z^0_{G}$.


\end{Emp}

\begin{Conj} \label{C:surj}
The homomorphism $[\C{A}]:K_0(\C{Z}_{\I^+}^{\Fr}(LG))\to Z^0_{G}$ from \re{shfun2} (d) is surjective.
\end{Conj}




\begin{Emp} \label{E:lcfun}
{\bf Locally constant function on $G^{rss}(F)$.} It follows from
\rtt{finite} that for each $\C{B}\in
\C{Z}^{\Fr}_{\I^+}(LG)$, the invariant distribution
${\nu_{[\C{A}(\C{B})]}}_{|G^{rss}(F)}$ on $G^{rss}(F)$ is locally constant, thus
${\nu_{[\C{A}(\C{B})]}}_{|G^{rss}(F)}=\phi_{[\C{A}(\C{B})]}\mu^{I^+}$ for some
$\phi_{[\C{A}(\C{B})]}\in C^{\infty}(G^{rss}(F))$.
\end{Emp}

Our next goal is to describe function $\phi_{[\C{A}(\C{B})]}$ explicitly.

\begin{Emp}
{\bf Notation.} (a) For every $\gm\in G^{rss}(F)$, we
denote by $\pr_{\gm}:\Fl\to\frac{LG}{\I}$ the map
$[g]\mapsto [g^{-1}\gm g]$. In particular,
$\pr_{\gm}$ defines functor
$\pr_{\gm}^*:\C{D}^{\Fr}(\frac{LG}{\I})\to\C{D}^{\Fr}(\Fl)$. Let $\ev^{\I}_{\dt_{\I^+}}:\C{Z}_{\I^+}^{\Fr}(LG)\to
\C{M}^{\Fr}(\frac{LG}{\I})$ be the evaluation functor from \rtt{weyl} (a), and let
$(\otimes\mu^{\frac{\I^+}{\I}})^{-1}:\C{M}^{\Fr}(\frac{LG}{\I})\cong\C{D}^{\Fr}(\frac{LG}{\I})$
be the inverse of the equivalence $\otimes\mu^{\frac{\I^+}{\I}}:\C{D}^{\Fr}(\frac{LG}{\I})\cong
\C{M}^{\Fr}(\frac{LG}{\I})$ (see \re{haarind} (b)).

(b) We denote by $p_{\gm}^*$ the composition
\[
\C{Z}_{\I^+}^{\Fr}(LG)\overset{\ev^I_{\dt_{\I^+}}}{\lra}\C{M}^{\Fr}(\frac{LG}{\I})
\overset{(\otimes\mu^{\frac{\I^+}{\I}})^{-1}}{\lra}\C{D}^{\Fr}(\frac{LG}{\I})
\overset{\pr_{\gm}^*}{\lra}\C{D}^{\Fr}(\Fl).
\]

(c) It can be deduced from \rtt{weyl} (b) that for each $\C{B}\in \C{Z}^{\Fr}_{\I^+}(LG)$, each homology group
$H_i(\Fl,p_{\gm}^*\C{B})$ has a natural $\wt{W}$-action (see \re{wpactions} (e)).
\end{Emp}

\begin{Conj} \label{C:conj2}
Let $\C{B}\in \C{Z}^{\Fr}_{\I^+}(LG)$ and $\gm\in G^{rss}(F)$.
Then

(a) each $\qlbar[\wt{W}]$-module $H_i(\Fl,p_{\gm}^*\C{B})$ is finitely generated;

(b) we have an equality (see \re{remfd} (a) below)
\begin{equation} \label{Eq:trform}
\phi_{[\C{A}(\C{B})]}(\gm)=\sum_{i,j}(-1)^{i+j}\Tr(\Fr,\Tor_j^{\wt{W}}(H_i(\Fl,p_{\gm}^*
\C{B}),\sgn)),
\end{equation}
where $\sgn$ is the sign character of $\wt{W}$.
\end{Conj}

\begin{Emp} \label{E:remfd}
{\bf Remark.} (a) Note that for every finitely generated $\qlbar[\La]$-module $V$, its homology group
$H_j(\La,V)$ is finite-dimensional. Hence, by \rcn{conj2} (a), each $\qlbar$-vector space
$\Tor_j^{\wt{W}}(H_i(\Fl,p_{\gm}^*
\C{B}),\sgn)=H_j(\La, H_i(\Fl,p_{\gm}^*\C{B}))_{W_,\sgn}$ is finite dimensional.
Therefore the right hand side of \form{trform} is defined.

(b) We believe that once \rcn{conj2} (a) is established, \rcn{conj2} (b) can be proven by the same argument as \rt{tf} below.
\end{Emp}

\subsection{Monodromic case}

\begin{Emp} \label{E:mon}
{\bf Monodromic sheaves.}  Let $\theta$ be a tame rank one local
system on $T$.

(a) We denote by $\C{M}_{\I,\theta}(LG)\subset
\C{M}_{\I^+}(LG)$ the full $\infty$-subcategory, whose objects are
$T$-monodromic with respect to the left and the right actions with
generalized eigenvalues $\theta$. Then
$\C{M}_{\I,\theta}(LG)\subset\C{M}_{\I^+}(LG)$ is a monoidal category
without unit.

(b) Note that the action of the monoidal $\infty$-category $\C{M}_{(\I^+)^2}((LG)^2)$ on $\C{M}_{\I^+}(LG)$
restricts to the action of $\C{M}_{(\I^2,\theta^2)}((LG)^2)\subset\C{M}_{(\I^+)^2}((LG)^2)$ on
$\C{M}_{\I,\theta}(LG)$. We define $\C{Z}_{\I,\theta}(LG)$ to be the $\infty$-category of "geometric" endomorphisms
of $\C{M}_{\I,\theta}(LG)$, viewed as a module category over $\C{M}_{(\I^2,\theta^2)}((LG)^2)$ (compare \re{centpn} and \re{remcentpn}).
\end{Emp}


\begin{Emp} \label{E:conjcl}
{\bf $W$-conjugacy.} (a) We say that two tame local systems $\theta$ and $\theta'$ are $W$-conjugate, if there exists $w\in W$ such that
$w^*(\theta)\cong \theta'$. In this case, $w$ induces equivalences $\C{M}_{\I,\theta}(LG)\isom \C{M}_{\I,w^*(\theta)}(LG)$ and hence $\C{Z}_{\I,\theta}(LG)\isom\C{Z}_{\I,w^*(\theta)}(LG)$.

(b) Notice that if $w^*(\theta)\cong\theta$, then $w^*:\C{M}_{\I,\theta}(LG)\isom \C{M}_{\I,\theta}(LG)$
belongs to $\C{Z}_{\I,\theta}(LG)$, thus the induced equivalence $w^*:\C{Z}_{\I,\theta}(LG)\isom \C{Z}_{\I,\theta}(LG)$
is naturally equivalent to the identity. Therefore if $\theta$ and $\theta'$ are $W$-conjugate, then  $\C{Z}_{\I,\theta}(LG)$ and $\C{Z}_{\I,\theta'}(LG)$ are canonically identified.

(c) For every $W$-conjugacy class $[\theta]$ of tame local systems on $T$, we define $\C{Z}_{\I,[\theta]}(LG)$ to be $\C{Z}_{\I,\theta}(LG)$
for some $\theta\in[\theta]$. This is well-defined by (b).
\end{Emp}

\begin{Emp} \label{E:arith}
{\bf The case of finite fields.}
(a) Assume now that we are in the situation of subsection 3.4 and that the $W$-conjugacy class $[\theta]$ of $\theta$ is Frobenius-invariant. Then the $\infty$-category $\C{Z}_{\I,[\theta]}(LG)$ is equipped with a natural action of Frobenius, thus we can talk about the $\infty$-category of Frobenius equivariant objects $\C{Z}^{\Fr}_{\I,[\theta]}(LG)$.

(b)  Notice that every $\theta$ as in (a) defines a semi-simple conjugacy
class in $\check{G}$ such that $\theta^q\sim\theta$. In particular, $\theta$ defines
a subset $\Irr(G)^\theta\subset\Irr(G)^0$  and a subalgebra $Z^{\theta}_{G}\subset Z_G^0$
(see \re{depth0rep}).
\end{Emp}

\begin{Conj} \label{C:thetafunc}
In the situation of \re{arith}, there is a natural surjective algebra homomorphism
$[\C{A}_{\theta}]:K_0(\C{Z}^{\Fr}_{\I,[\theta]}(LG))\to Z_G^{\theta}$, whose construction
is similar to  the homomorphism $[\C{A}]$ from \re{shfun2} (d)
(compare \rcn{surj}).
\end{Conj}

\begin{Emp} \label{E:denis}
{\bf Gaitsgory's central sheaves.} (a) For each $W$-conjugacy class of tame local system
$\theta$ on $T$, a generalization of Gaitsgory's construction \cite{Ga} gives
a monoidal functor $\Psi_{\theta}:\Rep \check{G}\to
\C{Z}_{\I,[\theta]}(LG)$ (see \cite[3.5]{Be}).

(b) Assume from now on that we are in the situation of \re{arith}.
Then $\Psi_{\theta}$ lifts to a monoidal functor to
$\Psi_{\theta}^{\Fr}:\Rep\check{G}\to\C{Z}^{\Fr}_{\I,[\theta]}(LG)$. Therefore
by \rcn{thetafunc}, $\theta$ defines an algebra homomorphism
$[\Phi_{\theta}]:=[\C{A}_{\theta}]\circ[\Psi^{\Fr}_{\theta}]:K_0(\Rep\check{G})\to
 Z^{\theta}_{G}$.

(c) For each $\theta$ as in (b) and $V\in\Rep\check{G}$, we set
$z_V^\theta:=[\Phi_{\theta}](V)\in{Z}^{\theta}_{G}$.
\end{Emp}


\begin{Conj} \label{C:conj3}
(a) For every $\theta$ as in \re{arith}, $V\in\Rep\check{G}$ and $\pi\in\Irr(G)^\theta$,
we have an equality $f_{z_V^\theta}(\pi)=\Tr(s(\pi),V)$.

(b) For every $\theta$ as in \re{arith}, the image of $[\Phi_{\theta}]:K_0(\Rep\check{G})\to
Z^{\theta}_{G}$ lies in $Z^{st,\theta}_{G}$.

(c) The induced homomorphism $[\Phi_{1}]:K_0(\Rep\check{G})\to Z^{st,1}_{G}$ is an isomorphism.
\end{Conj}

\begin{Emp}
{\bf Remark.} Note that \rcn{conj3} implies the
depth zero stable center conjecture for unipotent representations.
\end{Emp}

\section{Geometric construction and stability of the Bernstein projector}

\subsection{Construction} \label{SS:unit}
In this subsection we carry out the strategy outlined in subsection 3.3 and construct a $K$-group analog
$\lan A\ran\in\frak{Z}(LG)$ of $z^0$ (see \rt{conjunit}).

Notice that $\ev_{\dt_{\I^+}}(\Id)=\dt_{\I^+}\in M(LG)$ has a natural lift
$\dt_{\frac{\I^+}{\I}}\in M(\frac{LG}{\I})$, showing an analog of \rtt{weyl} (a) in this case.
The following assertion is a homotopy analog of  \rtt{weyl} (b).




\begin{Lem} \label{L:weyl}
For every $\P\in {\Par}$, the element
${\Av}^{\P/\I}_{\P}(\dt_{\frac{\I^+}{\I}})\in M(\frac{LG}{\P})$ is equipped with an
action of the finite Weyl group $W_{\P}\subset\wt{W}$. Moreover, the sheaf of
skew-coinvariants ${\Av}^{\P/\I}_{\P}(\dt_{\frac{\I^+}{\I}})^{W_{\P,\sgn}}$ is naturally isomorphic to
$\dt_{\frac{\P^+}{\P}}$.
\end{Lem}

\begin{proof}
Consider the Borel subgroup $B:=\I/\P^+$ of $L:=L_{\P}$ and the unipotent radical $U:=\I^+/\P^+$ of $B$. Then we have a  diagram
\begin{equation} \label{Eq:spr}
L/B\overset{\pr}{\lla} L\times^B U\overset{a}{\lra}L,
\end{equation}
where $a$ is the map $[g,u]\mapsto gug^{-1}$ (compare \re{adj} (c)). Then the classical Springer theory asserts that the Springer sheaf
$a_!\pr^!(1_{L/B})\in D(L)$ is equipped with an action of the Weyl group $W:=W_{\P}$, and that the sheaf of skew-coinvariants $(a_!\pr^!(1_{L/B}))^{W,\sgn}$ is naturally isomorphic to $\dt_1\in D(L)$, the $\dt$-sheaf at $1$.

Note that the diagram \form{spr} is $L$-equivariant with respect to the left multiplication on the first two factors and the adjoint
action on the third one. Moreover, the diagram
$\frac{1}{B}\overset{\ov{\pr}}{\lla} \frac{U}{B}\overset{[a]}{\lra}\frac{L}{L}$, where $[a]$ is induced by the embedding
$U\hra B\hra L$, is obtained by the quotient of \form{spr} by $L$. Let $1_{L/B}\in D(\frac{1}{B})$ be the unique Haar measure whose $!$-pullback to $L/B$ is $1_{L/B}$. Then $[a]_!\ov{\pr}^!(1_{L/B})\in D(\frac{L}{L})$ is equipped with an action of the Weyl group $W$, and  $([a]_!\ov{\pr}^!(1_{L/B}))^{W,\sgn}$ is naturally isomorphic to $\dt_{\frac{1}{L}}\in D(\frac{L}{L})$.

From this the assertion follows. By definition, ${\Av}^{\P/\I}_{\P}(\dt_{\frac{\I^+}{\I}})\in M(\frac{\P}{\P})\subset M(\frac{LG}{\P})$ is simply $[\wt{a}]_!\wt{\pr}^!(1_{\P/\I})$, obtained from the diagram
$\frac{1}{\I}\overset{\wt{\pr}}{\lla} \frac{\I^+}{\I}\overset{[\wt{a}]}{\lra}\frac{\P}{\P}$, where $1_{\P/\I}\in M(\frac{1}{\I})$ is
the unique Haar measure whose $!$-pullback to $\P/\I$ is $1_{\P/\I}$.

Then $1_{\P/\I}\in M(\frac{1}{\I})$ is the $!$-pullback of  $1_{L/B}\in D(\frac{1}{B})$, hence ${\Av}^{\P/\I}_{\P}(\dt_{\frac{\I^+}{\I}})$ is the $!$-pullback of  $[a]_!\ov{\pr}^!(1_{L/B})\in D(\frac{L}{L})$. Therefore ${\Av}^{\P/\I}_{\P}(\dt_{\frac{\I^+}{\I}})$ is equipped with a $W$-action, and its sheaf of  skew-invariants is the $!$-pullback of  $\dt_{\frac{1}{L}}\in D(\frac{L}{L})$, which is  $\dt_{\frac{\P^+}{\P}}$.
\end{proof}

\begin{Emp}
{\bf Notation.} Note that for every $\P\in\Par$ and every locally closed subscheme $Y\subset \Fl$ the image $Y_{\P}\subset \Fl_{\P}$
is locally closed. Hence we can form  $A^Y_{\P}:=\Av^{Y_{\P}}(\dt_{\frac{\P^+}{\P}})\in M(LG)$ (compare \re{aver}).
In particular, for $Y\in\T$ (see \re{parah}), we have $A^Y_{\P}=\C{A}^Y_{\P}(\Id)$ (see \re{strategy} (b)).
\end{Emp}

By \rl{weyl}, the following assertion a homotopy analog of the particular case
of \re{strategy} (b), which is crucial for the whole argument.

\begin{Lem} \label{L:func}
There exists a natural functor $A_{\star}^{\star}:{\Par}^{op}\times\T^{op}\to M(LG)$, whose value at
$(\P,Y)$  is $A^Y_{\P}$.
\end{Lem}

\begin{proof}
We have to show that $A^Y_{\P}=\Av^{Y_{\P}}(\dt_{\frac{\P^+}{\P}})$ is contravariant in $\P\in \Par$ and $Y\in\T$.
Since the functoriality with respect to $Y$ is clear, it remains to construct a natural morphism
$\Av^{Y_{\Q}}(\dt_{\frac{\Q^+}{\Q}})\to \Av^{Y_{\P}}(\dt_{\frac{\P^+}{\P}})$ for each $\P\subset\Q$ in $\Par$
and to show the compatibility with compositions.

Since $\Q^+$ is normalised by $\Q$, thus by $\P$, we can consider
element $\dt_{\frac{\Q^+}{\P}}\in D(\frac{LG}{\P})$. Since $\P\subset\Q$, we have $\Q^+\subset\P^+$. Thus
we have a natural morphism $\dt_{\frac{\Q^+}{\P}}\to\dt_{\frac{\P^+}{\P}}$ (see \re{consind} (b)), hence
a morphism $\Av^{Y_{\P}}(\dt_{\frac{\Q^+}{\P}})\to \Av^{Y_{\P}}(\dt_{\frac{\P^+}{\P}})$.
It remains to construct a morphism $\Av^{Y_{\Q}}(\dt_{\frac{\Q^+}{\Q}})\to \Av^{Y_{\P}}(\dt_{\frac{\Q^+}{\P}})$.

Consider commutative diagram
\[
\begin{CD}
\Fl_{\P}\times\frac{LG}{\Q} @<p_1<< LG\times^{\P}LG @>a^{\P}>> LG\\
@V\pr\times\Id VV @Va^{\P,\Q}VV @|\\
\Fl_{\Q}\times\frac{LG}{\Q} @<p_2<< LG\times^{\Q}LG @>a^{\Q}>> LG.
\end{CD}
\]
By definition, $\Av^{Y_{\Q}}(\dt_{\frac{\Q^+}{\Q}})=a^{\Q}_!p_2^!(1_{Y_{\Q}}\pp \dt_{\frac{\Q^+}{\Q}})$, while $\Av^{Y_{\P}}(\dt_{\frac{\Q^+}{\P}})=a^{\P}_!p_1^!(1_{Y_{\P}}\pp \dt_{\frac{\Q^+}{\Q}})$.
Since $a^{\P}_!\cong a^{\Q}_!a^{\P,\Q}_!$, we have to construct a morphism
\[
p_2^!(1_{Y_{\Q}}\pp \dt_{\frac{\Q^+}{\Q}})\to a^{\P,\Q}_!p_1^!(1_{Y_{\P}}\pp \dt_{\frac{\Q^+}{\Q}}).
\]
Since the left inner square is Cartesian and the $p_i$'s are formally smooth, we have an isomorphism of functors $p_2^!(\pr_!\pp\Id)\cong
p_2^!(\pr\pp\Id)_!\cong a^{\P,\Q}_!p_1^!$, so it remains to construct a morphism $1_{Y_{\Q}}\to\pr_!1_{Y_{\P}}$.

Since $\pr:Y_{\P}\to Y_{\Q}$ is proper, we define $1_{Y_{\Q}}\to\pr_!1_{Y_{\P}}=\pr_*1_{Y_{\P}}=\pr_*\pr^* 1_{Y_{\Q}}$
to be the counit map. The compatibility with compositions is standard.
\end{proof}


\begin{Emp}
{\bf Notation.}
(a) For every $Y\in\T$, we set
\[
\lan{A}^Y\ran:=\sum_{\P\in \Par}(-1)^{\rk G-\rk\P}\lan {A}_{\P}^Y\ran \in
K_0(M(LG)).
\]

(b) Denote by  $\End^r K_0(M(LG))$ the algebra of endomorphisms of $K_0(M(LG))$, viewed as a right $K_0(M(LG))$-module.

\end{Emp}

The part (a) of following result is a $K$-group analog of \rtt{mconj2} (a) for the unit element.
We will prove \rt{stab} in subsection \ref{SS:stab}.

\begin{Thm} \label{T:stab}
(a) For every $\C{X}\in M(LG)$, the system
$\{\lan {A}^Y\ran \ast\lan \C{X}\ran\}_{Y\in\T}$ stabilizes.

(b) For every $\P\in \Par$ and $Y\in\T$, we have an equality
$\lan {A}^Y\ran \ast\lan \dt_{\P^+}\ran=\lan\dt_{\P^+}\ran$.
\end{Thm}

\begin{Cor} \label{C:stab}
(a) There exists a unique element $\lan A\ran\in \End^r K_0(M(LG))$
such that $\lan A\ran (\lan\C{X}\ran)= \lan A^Y\ran\ast\lan \C{X}\ran$
for each $\C{X}\in M(LG)$ and each sufficiently large $Y\in\T$.

(b) For every $\P\in \Par$, we have $\lan A\ran(\lan\dt_{\P^+}\ran)=\lan\dt_{\P^+}\ran$.
\end{Cor}
\begin{proof}
(a) It follows from \rt{stab} (a) that for every $\C{X}\in M(LG)$ we can form an element
$\lan A\ran (\C{X}):=\lim_Y(\lan {A}^Y\ran \ast\lan \C{X}\ran)\in K_0(M(LG))$.

We claim that the map $\C{X}\mapsto \lan A\ran (\C{X})$ defines a group endomorphism of $K_0(M(LG))$.
We have to show that for every distinguished triangle $\C{X}_1\to\C{X}_2\to\C{X}_3\to$ in
$M(LG)$, we have an equality $\lan A\ran (\C{X}_2)=\lan A\ran (\C{X}_1)+\lan A\ran (\C{X}_3)$.
By the definition of $\lan A\ran (\C{X}_i)$, there exists $Y\in\T$ such that $\lan A\ran (\C{X}_i)=\lan {A}^Y\ran \ast\lan \C{X}_i\ran$ for $i=1,2,3$.
Since $\lan A^Y\ran=\sum_{\P}(-1)^{\rk G-\rk{\P}}\lan A^Y_{\P}\ran$, it is enough to show an equality
\[
\lan {A}_{\P}^Y\ran \ast\lan \C{X}_2\ran=\lan {A}_{\P}^Y\ran \ast\lan \C{X}_1\ran+\lan {A}_{\P}^Y\ran \ast\lan \C{X}_3\ran
\]
for all $Y$ and $\P$. But this follows from the fact that
$\C{X}_1\ast{A}_{\P}^Y\to\C{X}_2\ast{A}_{\P}^Y\to\C{X}_3\ast{A}_{\P}^Y\to$ is a distinguished triangle.

Finally, since $\lan A\ran (\lan\C{X}\ran)= \lan A^Y\ran\ast\lan \C{X}\ran$ for sufficiently large $Y\in\T$, the endomorphism
$\lan A\ran$ commutes with the right multiplication.

(b) follows immediately from \rt{stab} (b).
\end{proof}

The following result is a $K$-group analog of \rtt{mconj2} (b) for the unit element.
It is a geometric analog of the fact that the restriction of the distribution $\nu_{[A]}\in \Dist(G(F))$ to $G^{rss}(F)$
is locally constant.

\begin{Thm} \label{T:rss}
For every $\gm\in LG^{rss}$ there exists $n\in\B{N}$ such that the restriction $\{i^*_{\gm,n}(\lan{A}^Y\ran)\}_{Y\in\T}\in K_0(M(\gm \I_n^+))$ stabilizes.
\end{Thm}

\begin{proof}
The proof is based on the following lemma, whose proof will be carried out in subsection \ref{SS:stab}.

\begin{Lem} \label{L:const}
For every $\gm\in LG^{rss}$ there exists
$n\in\B{N}$ such that the restriction $i^*_{\gm,n}({A}_{\P}^Y)\in
M(\gm \I_n^+)$ is constant for every $\P\in \Par$ and every locally closed $\I$-invariant subscheme
$Y\subset \Fl$.
\end{Lem}

Choose $n$ as in \rl{const}. Then the restriction $i^*_{\gm,n}({A}_{\P}^Y)$ is isomorphic to
$i^*_{\gm,n}({A}_{\P}^Y)\ast\dt_{\I_n^+}\cong i^*_{\gm,n}({A}_{\P}^Y\ast\dt_{\I_n^+})$
for every $Y$ and $\P$, hence
$i^*_{\gm,n}(\lan{A}^Y\ran)=i^*_{\gm,n}(\lan{A}^Y\ran\ast\lan \dt_{\I_n^+}\ran)$ for all
$Y$. Thus it  suffices to show that system $\{\lan{A}^Y\ran \ast\lan\dt_{\I_n^+}\ran\}_Y$ stabilizes,
but this follows from \rt{stab} (a).
\end{proof}

Part (a) of following result can be thought as a particular case of \rtt{emb}, while
part (b) can be thought as a particular case of the fact that the functor $\C{A}:\C{Z}_{\I^+}(LG)\to \C{Z}(LG)$ is monoidal.
It will be proven in subsection \ref{SS:cent}.

\begin{Thm} \label{T:center}
(a) Element $\lan A\ran \in\End^r K_0(M(LG))$
belongs to $\frak{Z}(LG)$.

(b) Element $\lan A\ran\in \frak{Z}(LG)$ is an idempotent.
\end{Thm}

\begin{Cor} \label{C:center}
For every $\C{X}\in M(LG)$ and
$\C{F}\in M(LG)$ we have an equality
\[
\lan A\ran(\lan\Ad^{\C{X}}(\C{F})\ran)=\lan\Ad^{\C{X}}\ran(\lan A\ran(\lan\C{F}\ran)).
\]
\end{Cor}


\subsection{Proof of \rt{stab} and \rl{const}} \label{SS:stab} Since $G$ is a direct product of simple groups, we may assume that $G$ is simple.

\begin{Emp} \label{E:not}
{\bf Notation.} (a) We denote by $\wt{\Phi}$ and $\wt{\Dt}$ the set of affine roots of $G$ and the set of simple
affine roots of $G$, respectively. We also set $r:=\rk G=|\wt{\Dt}|-1$.

(b) For every subset $J\subsetneq\wt{\Dt}$, we denote by
$\P_J\in{\Par}$ the parahoric subgroup such that $J$ is the set
of simple roots of $L_{\P}$.

(c) For $w\in\wt{W}$, we denote
by $J_w$ the set of  $\al\in\wt{\Dt}$ such that $w(\al)>0$, that is,
$w(\al)$ is a positive affine root. Note that $J_w\neq\wt{\Dt}$ for every $w\neq 1$.

(d) For $\al\in\wt{\Phi}$, we denote by $U_{\al}\subset LG$ the corresponding root subgroup.
We say that $\al\in \P_n^+$, if $U_{\al}\subset\P_n^+$. 

(e) For each $\P\in \Par$ and $n>0$, let $S(\P_n^+)\subset\wt{W}$ be the union of $\{1\}$ and the set of all
$w\in\wt{W}$ such that $w(\al)\notin \P_n^+$ for every $\al\in\wt{\Dt}$. We set $Y(\P_n^+):=\cup_{w\in S(\P_n^+)} \Fl^{\leq w}$.

(f) Let $N\subset G$ be the normalizer of $T$. For each $w\in\wt{W}=N(K)/T(\C{O})$ we fix a representative $\un{w}\in N(K)\subset LG$.
\end{Emp}

\begin{Lem} \label{L:fin}
(a) For every $\P\in \Par$ and $n\geq 0$, the set $S(\P_n^+)$ is finite.

(b) For every $\P\in \Par$, we have $S(\P^+)=\{1\}$.
\end{Lem}
\begin{proof}
(a) For every $\beta\in\wt{\Phi}$ such that $\beta-(n+1)>0$, we have $\beta\in \I_{n+1}^+\subset\P^+_n$.
Thus for every $w\in S(\P_n^+)$ and $\al\in\wt{\Dt}$ we have $w(\al)-(n+1)<0$.
On the other hand, there exist positive integers $\{n_{\al}\}_{\al\in\wt{\Dt}}$ such that the
linear function $\sum_{\al\in\wt{\Dt}}n_{\al}\al$ is $1$. Thus  $\sum_{\al\in\wt{\Dt}}n_{\al}w(\al)$ is $1$. Hence for each $\al\in\wt{\Dt}$, the set $\{w(\al)\}_{w\in S(\P_n^+)}$ is bounded, which implies that $S(\P_n^+)$ is finite.

(b) We have to show that for every $w\neq 1$ and $\P\in \Par$ there exists
$\al\in \wt{\Dt}$ such that $w(\al)\in\P^+$. We may assume that $\P$ is maximal, that is, $\P=\P_{\wt{\Dt}\sm\beta}$ for some $\beta\in\wt{\Dt}$. We have to show that
there exists $\al\in\wt{\Dt}$ such that $w(\al)\in\P^+_{\wt{\Dt}\sm\beta}$, that is,
the coefficient of $\beta$ in $w(\al)$ is positive. Since
$\sum_{\al\in\wt{\Dt}}n_{\al}w(\al)=1=\sum_{\al\in\wt{\Dt}}n_{\al}\al$, we conclude that the coefficient of
$\beta$ in $\sum_{\al\in\wt{\Dt}}n_{\al}w(\al)$ is $n_{\beta}>0$. Thus the coefficient of $\beta$ in some $w(\al)$ is positive as well.
\end{proof}

The following result is the main step in the proof of \rt{stab}.

\begin{Lem} \label{L:isom}
Let $w\in\wt{W}$, $\al\in\wt{\Dt}$, $\Q\in \Par$, $n\in\B{N}$ and $J\subset J_w\sm\al$ be such that $w(\al)\in\Q_n^+$ and $J\neq\wt{\Dt}\sm\al$. Then the morphism $\dt_{\P^+_{J\cup\al}}\to \dt_{\P^+_J}$
from \re{consind} (b) induces an isomorphism ${A}_{\P_{J\cup \al}}^{Y_w}\ast\dt_{\Q_n^+}\isom {A}_{\P_J}^{Y_w}\ast\dt_{\Q_n^+}$.
\end{Lem}

\begin{proof} We set $J':=J\cup\al$. First we show the equality
\begin{equation} \label{Eq:stab}
\P_{J}^+=\P_{J'}^+\cdot (\P_J^+\cap \un{w}^{-1}\Q_n^+\un{w}).
\end{equation}

For this it suffices to show that for every $\beta\in\wt{\Phi}$ such that $U_{\beta}$ is contained in $\P_{J}^+\sm \P_{J'}^+$ we have
$U_{\beta}\subset\un{w}^{-1}\Q_n^+\un{w}$ or, equivalently, $U_{w(\beta)}=\un{w}U_{\beta}\un{w}^{-1}$ is
contained in  $\Q_n^+$. In other words, we have to show that for each $\beta\in \P_{J}^+\sm\P_{J'}^+$ we have  $w(\beta)\in\Q_n^+$.

Notice that $\beta\in\wt{\Phi}$ belongs to $\P_{J}^+$ if and
only if $\beta=\sum_{\al_i\in\wt{\Dt}}n_i\al_i$ such that $n_i\geq
0$ for all $\al_i$, and $n_i>0$ for some $\al_i\notin J$. Thus
any $\beta\in\P_{J}^+\sm\P_{J'}^+$ has the form $\sum_{\al_i\in J}n_i\al_i+n\al$ such that  $n_i\geq 0$ for all $i$ and $n>0$. But $w(\al)\in\Q_n^+$ (by assumption)
and $w(\al_i)>0$ for every $\al_i\in J$ (because $J\subset J_w$). This implies that
$w(\beta)\in\Q_n^+$, which finishes the proof of equality \form{stab}.

Let $\beta_1,\ldots,\beta_{l(w)}\in\wt{\Phi}$ be all $\beta>0$ such that $w(\beta)<0$,
and set $\I_w:=\prod_i U_{\beta_i}\subset \I^+$.
Consider the closed subscheme $\I_w \un{w}\subset LG$ and the corresponding averaging functor $\Av^{\I_w \un{w}}:M(LG)\to M(LG)$ (see \re{aver} (d)). We claim that  we have a natural isomorphism
\begin{equation} \label{Eq:isomad}
{A}_{\P_J}^{Y_w}\ast\dt_{\Q_n^+}\cong \Av^{\I_w\un{w}}(\dt_{\P_J^+}\ast\dt_{\un{w}^{-1}\Q_n^+\un{w}}),
\end{equation}
and similarly for $J'$. Since $J\subset J_w$, the projection $LG\to \Fl_{\P_J}$ induces an isomorphism
$\I_w\un{w}\isom (Y_w)_{\P_J}$. Thus the composition $\pi:\I_w\un{w}\hra LG\to \Fl_{\P_J}$ satisfies $\pi_!(1_{\I_w\un{w}})\cong 1_{(Y_w)_{\P_J}}$.
Hence, by \rl{aver} (a), ${A}_{\P_J}^{Y_w}=\Av^{(Y_w)_{\P_J}}(\dt_{\P_J^+/\P_J})$ is isomorphic to  $\Av^{\I_w\un{w}}(\dt_{\P_J^+})\cong \Av^{\I_w}(\dt_{\un{w}\P_J^+\un{w}^{-1}})$. Thus the left hand side of \form{isomad} is isomorphic to
$\Av^{\I_w}(\dt_{\un{w}\P_J^+\un{w}^{-1}})\ast\dt_{\Q_n^+}$.

On the other hand, the right hand side of the expression \form{isomad} is isomorphic to
$\Av^{\I_w}(\dt_{\un{w}\P_J^+\un{w}^{-1}}\ast\dt_{\Q_n^+})$. Thus it remains to show that for every
$\C{F}\in M(LG)$ we have a natural isomorphism
$\Av^{\I_w}(\C{F})\ast\dt_{\Q_n^+}\cong \Av^{\I_w}(\C{F}\ast\dt_{\Q_n^+})$.
But this follows from the fact that $\I_w\subset \I^+\subset\Q$ and that $\Q$ normalizes $\Q_n^+$.

Now we are ready to show the assertion of the lemma. By \form{isomad}, it remains to show that the natural morphism $\dt_{\P^+_{J'}}\to \dt_{\P^+_J}$  induces an isomorphism
\[
\dt_{\P_{J'}^+}\ast\dt_{\un{w}^{-1}\Q_n^+\un{w}}\isom \dt_{\P_{J}^+}\ast\dt_{\un{w}^{-1}\Q_n^+\un{w}}.
\]

Note that the multiplication map $m:(LG)^2\to LG$ induces a pro-unipotent morphism
$\P_J^+\times (\un{w}^{-1}\Q_n^+\un{w})\to
\P_J^+\cdot(\un{w}^{-1}\Q_n^+\un{w})$. Hence it induces
an isomorphism $\dt_{\P_J^+}\ast\dt_{\un{w}^{-1}\Q_n^+\un{w}}\cong \dt_{\P_J^+\cdot(\un{w}^{-1}\Q_n^+\un{w})}$, and similarly
for $J'$. Thus it remains to show the equality $\P_J^+\cdot(\un{w}^{-1}\Q_n^+\un{w})=
\P_{J'}^+\cdot(\un{w}^{-1}\Q_n^+\un{w})$, but this follows from \form{stab}.
\end{proof}

For each $\Q\in \Par$ and $n\geq 0$ the set $S(\Q_n^+)$ is finite (by \rl{fin}), hence
$Y(\Q_n^+)\in\T$. The following result is a more precise version of \rt{stab} (a) for
$\C{X}=\dt_{\Q_n^+}$.

\begin{Prop} \label{P:stab1}
Let $\Q\in \Par$, $n\geq 0$, and $Y\in\T$.

(a) If $Y\supset Y(\Q^+_n)$, then we have
an equality $\lan {A}^Y\ran \ast\lan \dt_{\Q_n^+}\ran= \lan{A}^{Y(\Q_n^+)}\ran \ast\lan\dt_{\Q_n^+}\ran$.

(b) We have an equality
$\lan {A}^Y\ran \ast\lan \dt_{\Q^+}\ran=\lan\dt_{\Q^+}\ran$.
\end{Prop}

\begin{proof}
(a) By the induction on the number of orbit $\I$-orbits in $Y\sm Y(\Q^+_n)$, it is enough to show
that for every $Y\in\T$ and $w\in\wt{W}\sm S(\Q_n^+)$ such that $Y_w\subset Y$ is open, we have an equality
$\lan A^Y\ran \ast\lan \dt_{\Q_n^+}\ran=\lan A^{Y\sm Y_w}\ran \ast\lan \dt_{\Q_n^+}\ran$.
Set $Y':=Y\sm Y_w$.

Consider element $\langle{A}^w\rangle:=\sum_{J\subset
J_w}(-1)^{r-|J|}\langle{A}_{\P_J}^{Y_w}\rangle\in
K_0(M(LG))$. We claim that
$\langle{A}^Y\rangle=\langle{A}^{Y'}\rangle+\langle{A}^w\rangle$.
Indeed, let $J\subsetneq\wt{\Dt}$ be a proper subset. Since
$Y_w\subset Y$ is open, we have $Y_{\P_J}\sm Y'_{\P_J}=(Y_w)_{\P_J}$, if
$J\subset J_w$, and ${Y}_{\P_J}=Y'_{\P_J}$, otherwise.
This implies the equality
$\langle{A}^Y\rangle=\langle{A}^{Y'}\rangle+\langle{A}^w\rangle$
in $K_0(M(LG))$. Therefore it is enough to show that
$\langle{A}^w\rangle\ast\langle\dt_{\Q_n^+}\rangle=0$ for every
$w\in\wt{W}\sm S(\Q_n^+)$.

By the definition of $S(\Q_n^+)$, for each $w\in \wt{W}\sm S(\Q_n^+)$
there exists $\al\in\wt{\Dt}$ such that $w(\al)\in\Q_n^+$.
In particular, $\al\in J_w$. Hence, by \rl{isom},
for every $J\subset J_w\sm \al$ we have an equality
$\langle{A}_{\P_J}^{Y_w}\ast\dt_{\Q_n^+}\rangle=\langle{A}_{\P_{J\cup
\al}}^{Y_w}\ast\dt_{\Q_n^+}\rangle$. Since
$\langle{A}^w\rangle\ast\langle\dt_{\Q_n^+}\rangle$ equals
\[
\sum_{J\subset J_w\sm\al}(-1)^{r-|J|}
(\langle{A}_{\P_J}^{Y_w}\ast\dt_{\Q_n^+}\rangle-
\langle{A}_{\P_{J\cup \al}}^{Y_w}\ast\dt_{\Q_n^+}\rangle), \] we conclude
that $\langle{A}^w\rangle\ast\langle\dt_{\Q_n^+}\rangle=0$.

(b) By \rl{fin} (b), we have  $S(\Q^+)=\{1\}$, thus $Y(\Q^+)=Y_1$. Then, by (a),
we have an equality $\lan A^Y\ran \ast\lan \dt_{\Q^+}\ran=\lan{A}^{Y_1}\ran\ast\lan\dt_{\Q^+}\ran$. Choose $\al\in \wt{\Dt}$ such that
$\al\in\Q^+$.
It now follows from \rl{isom} for $w=1$ and $n=0$ that for every subset $J\subsetneq\wt{\Dt}\sm \al$, we have an equality
$\langle{A}_{\P_J}^{Y_1}\ast\dt_{\Q^+}\rangle=\langle{A}_{\P_{J\cup
\al}}^{Y_1}\ast\dt_{\Q^+}\rangle$. Using equality $\lan A^{Y_1}\ran=\lan A^{Y_1}_{\P_{\wt{\Dt}\sm \al}}\ran+
\sum_{J\subsetneq\wt{\Dt}\sm \al} (-1)^{r-|J|}(\lan A^{Y_1}_{\P_J}\ran -\lan A^{Y_1}_{\P_{J\cup\al}}\ran)$,
we deduce that $\lan{A}^{Y_1}\ran\ast\lan\dt_{\Q^+}\ran= \lan{A}^{Y_1}_{\P_{\wt{\Dt}\sm\al}}\ast\dt_{\Q^+}\ran$.
Note that  ${A}^{Y_1}_{\P_{\wt{\Dt}\sm\al}}=\dt_{\P_{\wt{\Dt}\sm\al}^+}$ and that
$\P_{\wt{\Dt}\sm\al}^+\subset\Q^+$, because $\al\in\Q^+$. Therefore
${A}^{Y_1}_{\P_{\wt{\Dt}\sm\al}}\ast\dt_{\Q^+}=\dt_{\P_{\wt{\Dt}\sm\al}^+}\ast\dt_{\Q^+}\cong\dt_{\Q^+}$, implying the assertion.
\end{proof}

Now we are ready to prove \rt{stab}.

\begin{proof}[Proof of \rt{stab}]
By \re{biinv} (b), for each $\C{X}\in M(LG)$ there exists $n\in\B{N}$ such that $\dt_{\I_n^+}\ast\C{X}=\C{X}$. Hence
it is enough to show that the system $\{\lan{A}^Y\ran\ast\lan\dt_{\I_n^+}\ran\}_Y$
stabilizes for each $n$. Thus the assertion follows from \rp{stab1}.
\end{proof}

The following result will be used in the next subsection.

\begin{Cor} \label{C:stab1}
For every $\C{X}\in M((LG)^2)$ the sequence $\{\lan\C{X}\ran(\lan A^Y\ran)\}_Y$
stabilizes.
\end{Cor}

\begin{proof}
Choose $n$ such that $\C{X}\ast\dt_{(\I_n^+)^2}\cong\C{X}$. Then $\lan\C{X}\ran(\lan A^Y\ran)$ equals \[
\lan\C{X}\ran(\lan\dt_{(\I_n^+)^2}\ran(\lan A^Y\ran))=\lan\C{X}\ran(\lan\dt_{\I^+_n}\ran\ast\lan A^Y\ran\ast\lan\dt_{\I^+_n}\ran).
\]
Thus the assertion follows from the fact that the sequence
$\{\lan A^Y\ran\ast\lan\dt_{\I^+_n}\ran\}_Y$ stabilizes.
\end{proof}

We finish this subsection with the proof of \rl{const}.

\begin{Emp} \label{E:pflconst}
\begin{proof}[Proof of \rl{const}]
Let $\mu^{\P^+}\in \wt{M}(LG)$ be as in \re{haarind} (c), and denote by
$\ov{A}_{\P}^Y\in D(LG)$ the unique element such that $\ov{A}_{\P}^Y\otimes \mu^{\P^+}\cong {A}_{\P}^Y$ (see \re{haarind} (b)).
Then  $i^*_{\gm,n}({A}_{\P}^Y)\cong i^*_{\gm,n}(\ov{A}_{\P}^Y)\otimes i^*_{\gm,n}(\mu^{\P^+})$ (see \re{tensor} (c)), thus it remains to show that $i^*_{\gm,n}(\ov{A}_{\P}^Y)\in D(\gm \I_n^+)$ is constant.

Consider the loop group $LG_{\gm}\subset LG$, corresponding to the centralizer $G_{\gm}\subset G$, and set $LG^{rss}_{\gm}:=LG_{\gm}\cap LG^{rss}\subset LG$. First we claim that there exists $m\in\B{N}$ such that $X_m:=LG_{\gm}\cap \gm \I_m^+$ is contained in $LG^{rss}_{\gm}$, and the restriction $\ov{A}_{\P}^Y|_{X_m}$ is constant.

Let $\Spr_{\P}\subset \Fl_{\P}\times LG$ be the closed ind-subscheme consisting of $(y,g)\in \Fl_{\P}\times LG$ such that $ygy^{-1}\in \P^+$, and set $\Spr_{\P}^Y:=(Y_{\P}\times LG)\cap \Spr_{\P}$.  By \rl{aver} (b), we have $\ov{A}_{\P}^Y=\Ad_{\P}^{1_Y,*}(1_{\frac{\P^+}{\P}})$, thus
$\ov{A}_{\P}^Y=(\pr_2)_!(1_{\Spr_{\P}^Y})\in D(LG)$.

It remains to show that there exists $m$ such that the restriction of $\Spr_{\P}$ to $\Fl_{\P}\times X_m$ is constant, that is, there exists $m\in\B{N}$ such that we have an equality of affine Springer fibers
$\Fl_{\P,\gm}=\Fl_{\P,\gm'}$ (see \re{aplasf} (a)) for every $\gm'\in X_m$. In the case of Lie algebras the corresponding assertion was shown in \cite{KL}, reducing to the case when $G_{\gm}$ is split. In the group case, the proof is similar.

Next, we consider the morphism $a:G\times G_{\gm}^{rss}\to G:(x,y)\mapsto xyx^{-1}$ and the induced morphism $La:LG\times LG_{\gm}^{rss}\to LG$. Since $a$ is smooth, it follows from \rl{smooth} (b) that there exist fp-closed neighbourhoods
$X'$ of $(1,\gm)\in LG\times LG_{\gm}^{rss}$ and $X''$ of $\gm\in LG$  such that $La$
restricts  to a pro-unipotent morphism $b:X'\to X''\subset LG$. Moreover, by \rl{nbh}, we can further  assume that $X'\subset \I\times X_m$. Furthermore, since $\{\gm \I_n^+\}_n$ form a basis of fp-closed neighbourhoods of $\gm$ (by \re{basis}),
we have $X''\supset \gm \I_n^+$ for some $n$. We claim that for such an $n$, each $i^*_{\gm,n}(\ov{A}_{\P}^Y)\in D(\gm \I_n^+)$ is constant.

Consider the restriction of $(La)^*(\ov{A}_{\P}^Y)\in D(LG\times LG_{\gm}^{rss})$ to
$\I\times X_m$. Since $Y$ is $\I$-invariant, we get that $\ov{A}_{\P}^Y$ is $\Ad \I$-equivariant,
thus $(La)^*(\ov{A}_{\P}^Y)|_{\I\times X_m}$ is $\I$-equivariant. By our choice of $m$, we conclude that $(La)^*(\ov{A}_{\P}^Y)|_{\{1\}\times X_m}=\ov{A}_{\P}^Y|_{X_m}$ is constant, hence the restriction
$(La)^*(\ov{A}_{\P}^Y)|_{\I\times X_m}$ is constant.

Next, since $X'\subset \I\times X_m$, we conclude that $(La)^*(\ov{A}_{\P}^Y)|_{X'}=b^*(\ov{A}_{\P}^Y|_{X''})$ is constant.
But $b$ is pro-unipotent, thus $b^*$ is fully faithful, which implies that $\ov{A}_{\P}^Y|_{X''}$ is constant. Now the assertion follows from the inclusion $\gm \I_n^+\subset X''$.
\end{proof}
\end{Emp}

\subsection{Proof of \rt{center} and \rco{center}} \label{SS:cent}

\begin{Lem} \label{L:isom2}
Let $\al,\beta:LG\times (LG\times \P\times LG)\to (LG)^3 $ be the maps defined by formulas $\al(x,a,b,c):=(a, xb,c)$ and
$\beta(x,a,b,c):=(a, bx,c)$, and let $\pi$ be the projection $LG\times (LG\times \P\times LG)\to\frac{LG}{\P}\times (LG\times \P\times LG)$.

Then two compositions $\al_!\pi^!,\beta_!\pi^!:M(\frac{LG}{\P}\times (LG\times \P\times LG))
\to M((LG)^3)$ are isomorphic.
\end{Lem}

\begin{proof}
Consider the isomorphism $\gm: LG\times (LG\times \P\times LG)\to LG\times (LG\times \P\times LG)$ defined by the formula $\gm(x,a,b,c)=(bxb^{-1},a,b,c)$. Then $\gm$ satisfies $\pi\circ\gm\cong\pi$ and $\al\circ \gm=\beta$.
Therefore  $\beta_!\pi^!\cong\al_!\gm_!\gm^!\pi^!\cong \al_!\pi^!$.
\end{proof}

\begin{Lem} \label{L:gen}
Let $w,w'\in\wt{W}$ and $\al\in\wt{\Dt}$ be such that $w=s_{\al}w'$, where $s_{\al}\in\wt{W}$ is the simple reflection corresponding to
$\al$, and $w'<w$. Then the map of $K$-groups
\[
m_!:K_0(M(LG\times \P_{\al}\times LG^{\leq w'}\times LG))\to K_0(M(LG\times LG^{\leq w}\times LG)),
\]
induced by the multiplication map $m:\P_{\al}\times LG^{\leq w'}\to LG^{\leq w}$, is surjective.
\end{Lem}

\begin{proof}
Note that for every $X\in \Var_k$, closed embedding $i:Y\hra X$ and an open embedding
$j:U:=X\sm Y\hra X$, the functors $i_!:D(Y)\to D(X)$ and $j_!:D(U)\hra D(X)$ induce  embeddings of $K$-groups $K_0(D(Y))\hra K_0(D(X))$ and $K_0(D(U))\hra K_0(D(X))$. Moreover, for every $\C{F}\in D(X)$ we have a distinguished triangle $j_!j^*\C{F}\to \C{F}\to i_!i^*\C{F}\to $, which implies that $K_0(D(X))$ is generated by (the images of) $K_0(D(Y))$ and $K_0(D(U))$. More generally, if $U_i$ is a finite stratification of $X$ (by locally closed subschemes), then $K_0(D(X))$ is generated by the $K_0(D(U_i))$'s.

Recall that $LG^{\leq w}$ has a finite stratification  $\{\I^+u\I\}_{u\in\wt{W},u\leq w}$.
Thus, by the observations made above, it is enough to show that each $K_0(M(LG\times \I^+u\I\times LG))$ lies in the image of
$m_!$. Note that for each $u\leq w$ we have either $u\leq w'$ or $u=s_{\al}u'$ and $u'\leq w'$ (see \cite[Prop. 5.9]{Hu}).
Thus it remains to show the assertion in two cases $u=w$ and $u=w'$.

In the first case, $m$ induces a pro-unipotent map
\[
m:LG\times \I^+s_{\al}\I^+\times \I^+w'\I\times LG\to LG\times \I^+w\I\times LG.
\]
Thus for every $\C{F}\in M(LG\times \I^+w\I\times LG)$, we have $\C{F}\cong m_!m^!(\C{F})$, hence $\C{F}$ lies in the essential image of $m_!$. In the second case,  $m$ induces a pro-unipotent map
$m:LG\times \I^+\times \I^+w'\I\times LG\to LG\times \I^+w'\I\times LG$, and we conclude similarly.
\end{proof}




\begin{Lem} \label{L:center}
Let $\al,\beta:LG\times (LG)^3\to LG$ be the maps $\al(x,a,b,c):=axbc$ and
$\beta(x,a,b,c):=abxc$, respectively, and let $\al_!$ and $\beta_!$ be the induced functors
$M(LG\times (LG)^3)\to M(LG)$. Then for all
$\C{X}\in  M((LG)^3)$ and all sufficiently large $Y\in\T$, we have an equality
\begin{equation} \label{Eq:isom}
\al_!(\lan A^Y\ran\pp\lan \C{X}\ran)=\beta_!(\lan A^Y\ran \pp\lan\C{X}\ran).
\end{equation}
\end{Lem}

\begin{proof}
Since $M((LG)^3)$ is a colimit of the $M(LG\times LG^{\leq w}\times LG)$'s, we have to show that the equality \form{isom} holds for every $w\in\wt{W}$ and $\C{X}\in  M(LG\times LG^{\leq w}\times LG)$.

We prove the assertion by induction on $l(w)$. Assume that $l(w)\leq 1$. Then $LG^{\leq w}$ is a parahoric $\Q$. It follows from \rco{stab1} that both sides
of \form{isom}  stabilize when $Y$ is sufficiently large. Thus we may assume that $Y$ is $\Q$-invariant. In this case, we claim that we have an equality
\begin{equation} \label{Eq:isomp}
\al_!(\lan A_{\P}^Y\ran\pp\lan \C{X}\ran)=\beta_!(\lan A_{\P}^Y\ran \pp\lan\C{X}\ran)
\end{equation}
for each $\P\in \Par$. Indeed, since $Y_{\P}\subset \Fl_{\P}$ is $\Q$-invariant, $A_{\P}^Y=\Av^{Y_{\P}}(\dt_{\frac{\P^+}{\P}})$ lifts to an object of  $M(\frac{LG}{\Q})$ (by \re{aver} (b)). Thus  $A_{\P}^Y\pp\C{X}$ has a natural lift to an object of $M(\frac{LG}{\Q}\times (LG\times \Q\times LG))$, hence equality
\form{isomp} follows from \rl{isom2}. This implies equality \form{isom} in this case.

Assume now that $l(w)>1$. Choose $\al\in\wt{\Dt}$ and $w'\in\wt{W}$ such that $w=s_{\al}w'$ and
$w'<w$. By \rl{gen}, we can assume that $\C{X}=m_!(\C{X}')$ for some element $\C{X}'\in M(LG\times \P_{\al}\times LG^{\leq w'}\times LG)$.

Consider maps
$\wt{\al},\wt{\beta},\wt{\gm}:LG\times (LG)^4\to LG$ defined by $\wt{\al}(x,a,b',b'',c):=axb'b''c$, $\wt{\beta}(x,a,b',b'',c):=ab'b''xc$ and $\wt{\gm}(x,a,b',b'',c):=ab'xb''c$. Then we have equalities
$\wt{\al}_!(\lan A^Y\ran\pp\lan\C{X}'\ran)= \al_!(\lan A^Y\ran \pp\lan \C{X}\ran )$ and $\wt{\beta}_!(\lan A^Y\ran \pp\lan \C{X}'\ran)= \beta_!(\lan A^Y\ran \pp\lan\C{X}\ran)$. Thus to show the equality \form{isom} it is enough to show the equalities
\begin{equation} \label{Eq:isom2}
\wt{\al}_!(\lan A^Y\ran\pp\lan \C{X}'\ran)=\wt{\gm}_!(\lan A^Y\ran\pp\lan\C{X}'\ran)=
\wt{\beta}_!(\lan A^Y\ran \pp\lan \C{X}'\ran).
\end{equation}
Both these equalities follow from the induction hypothesis. Namely, consider maps
$m',m'':(LG)^4\to (LG)^3$ defined by $m'(a,b,c,d)=(ab,c,d)$ and $m''(a,b,c,d)=(a,b,cd)$.
Then the first equality of \form{isom2} follows from the equality
\form{isom} applied to $m''_!(\C{X}')\in M(LG\times LG^{\leq s_{\al}}\times LG)$, while
the second equality of \form{isom2} follows from the equality
\form{isom} applied to $m'_!(\C{X}')\in M(LG\times LG^{\leq w'}\times LG)$.
\end{proof}

Now we are ready to prove \rt{center} and \rco{center}.

\begin{proof}[Proof of \rt{center} (a)]
We have to show that for every $\C{F}\in M(LG)$ and $\C{X}\in M((LG)^2)$ we have an equality $\lan A\ran (\lan \C{X}(\C{F})\ran)=\lan\C{X}\ran (\lan A\ran (\lan\C{F}\ran))$. By the definition of
$\lan A\ran$, we have to show that for each sufficiently large $Y\in\T$, we have an equality
\begin{equation} \label{Eq:isom3}
\lan A^Y\ran \ast\lan\C{X}(\C{F})\ran=\lan\C{X}\ran(\lan A^Y\ran \ast\lan \C{F}\ran).
\end{equation}

 Let $\gm:(LG)^2\times LG\to LG$ be the map $(a,b,x)\mapsto axb$, $m:(LG)^2\to LG$ be the map $(a,b)\mapsto ab$, and $m':(LG)^2\times LG\to (LG)^2$ be the map $(a,b,x)\mapsto(a,xb^{-1})$. Consider  element $\C{X}':=m'_!(\C{X}\pp\C{F})$. Then equation \form{isom3}
 can be rewritten as $\lan A^Y\ran \ast \lan m_!(\C{X}')\ran=\gm_!(\lan\C{X}'\ran\pp\lan A^Y\ran)$.
 When $Y\in\T$ is sufficiently large, both sides of the last equality stabilize
 (by \rt{stab} and \rco{stab1}).

 Thus there exists $n$ such that
 $\lan\dt_{\I_n^+}\ran\ast \lan A^Y\ran \ast \lan m_!(\C{X}')\ran= \lan A^Y \ran \ast \lan  m_!(\C{X}')\ran$ and
 $\lan\dt_{\I_n^+}\ran \ast  \gm_!(\lan\C{X}'\ran\pp\lan A^Y\ran)=\gm_!(\lan\C{X}'\ran\pp\lan A^Y\ran)$
 for all $Y\in\T$ sufficiently large. Hence it remains to show the equality
 \begin{equation*} \label{Eq:eq3}
 \lan\dt_{\I_n^+}\ran \ast \lan A^Y\ran \ast \lan m_!(\C{X}')\ran=\lan\dt_{\I_n^+}\ran \ast\gm_!(\lan \C{X}'\ran\pp\lan A^Y\ran).
\end{equation*}
In the notation of \rl{center}, the last equality can be written in the form
\[
\al_!(\lan A^Y\ran\pp\lan \dt_{\I_n^+}\pp\C{X}'\ran)=\beta_!(\lan A^Y\ran \pp\lan\dt_{\I_n^+}\pp\C{X}'\ran),
\]
thus the assertion follows from equality \form{isom}.
\end{proof}

\begin{proof}[Proof of \rco{center}]
Let $b:LG\times (LG)^2\to (LG)^2$ be the map defined by $b(g,x,y):=(gx,gy)$, and let $b_!$ be the induced map $M(LG\times (LG)^2)\to M((LG)^2)$. Then for every $\C{X}\in M(LG)$ and $\C{B}\in M((LG)^2)$, we have a natural isomorphism
$\Ad^{\C{X}}(\C{B}(\cdot))\cong b_!(\C{X}\pp\C{B})(\cdot)$ of functors $M(LG)\to M(LG)$.

Choose $n$ such that $\dt_{(\I_n^+)^2}(\C{F})\cong\C{F}$ and $\lan\dt_{(\I_n^+)^2}\ran(\lan\C{A}\ran(\lan\C{F}\ran))=\lan\C{A}\ran(\lan\C{F}\ran)$
(see \re{biinv} (c)), and set $\C{Y}:=b_!(\C{X}\pp\dt_{(\I_n^+)^2})\in M((LG)^2)$.

Then we have an isomorphism $\Ad^{\C{X}}(\C{F})\cong \Ad^{\C{X}}(\dt_{(\I_n^+)^2}(\C{F}))\cong \C{Y}(\C{F})$, and, similarly, an equality
$\lan\Ad^{\C{X}}\ran(\lan A\ran(\lan\C{F}\ran))=\lan \C{Y}\ran(\lan A\ran(\lan\C{F}\ran))$.
Thus our assertion follows from the equality $\lan A\ran(\lan \C{Y}(\C{F})\ran)=\lan \C{Y}\ran(\lan A\ran(\lan\C{F}\ran))$,
shown in \rt{center} (a).
\end{proof}

\begin{proof}[Proof of \rt{center} (b)]
We have to show that $\lan A\ran\circ \lan A\ran=\lan A\ran$, that is,
for every $\C{F}\in M(LG)$ we have $\lan A\ran(\lan A\ran(\lan\C{F}\ran))=\lan A\ran(\lan\C{F}\ran)$.
By the definition of $\lan A\ran$, we have to show that $\lan A\ran(\lan A^Y\ran\ast\lan\C{F}\ran)=\lan A^Y\ran\ast\lan\C{F}\ran$ for each sufficiently large $Y$. Using the equality $\lan A\ran(\lan A^Y\ran\ast\lan\C{F}\ran)=\lan A\ran(\lan A^Y\ran)\ast\lan\C{F}\ran$
(see \rco{stab} (a)), it is enough to show that $\lan A\ran(\lan A^Y\ran)=\lan A^Y\ran$ for all $Y$.

 Recall that $\lan A^Y\ran =\sum_{\P\in\Par}(-1)^{\rk G-\rk\P}\lan A^Y_{\P}\ran $ and that $A^Y_{\P}=\Av^{Y_{\P}}(\dt_{\frac{\P^+}{\P}})$. Thus it remains to show that for every locally closed
$\I$-invariant subscheme $Y\subset \Fl_{\P}$ we have an equality
\begin{equation} \label{Eq:idem}
\lan A\ran(\lan \Av^{Y}(\dt_{\frac{\P^+}{\P}})\ran)=\lan \Av^{Y}(\dt_{\frac{\P^+}{\P}})\ran.
\end{equation}

By additivity, we can assume that $Y$ is one $\I$-orbit $(Y_w)_{\P}=\I w\P/\P$ and, moreover,
that $w$ is the longest element of the coset $w W_{\P}$. Then the projection $\pr:LG\to \Fl_{\P}$
induces an isomorphism between $Y_w^+:=\I^+w\I^+/\I^+\subset LG/\I^+$ and $(Y_w)_{\P}\subset \Fl_{\P}$.

Set $\C{X}_w:=1_{Y^+_w}\in D(LG/\I^+)\subset M(LG)$. Then $\pr_!(\C{X}_w)\cong 1_{(Y_w)_{\P}}$, so it follows from \rl{aver} (a) that
$\Av^{(Y_w)_{\P}}(\dt_{\frac{\P^+}{\P}})\cong \Ad^{\C{X}_w}(\dt_{\P^+})$.
%
Thus the left hand side of \form{idem} is equals to
$\lan A\ran(\lan\Ad^{\C{X}_w}(\dt_{\P^+})\ran)=
\lan\Ad^{\C{X}_w}\ran(\lan A\ran(\lan\dt_{\P^+}\ran))$ (by \rco{center}), thus to
$\lan\Ad^{\C{X}_w}\ran(\lan\dt_{\P^+}\ran)=\lan \Av^{(Y_w)_{\P}}(\dt_{\P^+})\ran$ by \rco{stab} (b).
\end{proof}

\subsection{Application to the classical Bernstein projector}

Assume that we are in the situation of subsection \ref{SS:conj2}.
Note that the element $\lan A\ran\in \frak{Z}(LG)$ from \rt{center}, constructed in \rco{stab},
is Frobenius equivariant. Therefore it defines an element
$[A]:=[\lan A\ran]\in Z_{G}$ of the Bernstein center (see
\re{shfun2}). By construction, for every $\C{X}\in M^{\Fr}(LG)$, we have an equality
$[A]([\C{X}])=[\lan A\ran (\lan\C{X}\ran)]\in \C{H}(G(F))$.


\begin{Thm} \label{T:conjunit}
The element $[A]\in Z_{G}$ equals the projector
$z^0\in Z_G^0\subset Z_{G}$.
\end{Thm}

\begin{proof}
By the definition of $z^0$ (see \re{decomp} (c)), we have to show that $f_{[A]}(\pi)=1$ for each
$\pi\in\Irr(G)^0$ and $f_{[A]}(\pi)=0$ for each $\pi\in\Irr(G)^{>0}$.

Assume first that $(\pi,V)\in \Irr(G)^0$. By definition, there exists $\P\in {\Par}$ such
that $V^{P^+}\neq 0$. We have to show that
$[A](v)=v$ for all $v\in V^{{P}^+}$. By \rco{stab} (b), we conclude that
\[[A](\dt_{{P}^+})=[A]([\dt_{\P^+}])=[\lan A\ran (\lan \dt_{\P^+}\ran)]=[\lan\dt_{\P^+}\ran]=\dt_{{P}^+}.\]

Note that for each  $v\in V^{P^+}$
we have $\dt_{{P}^+}(v)=v$.  Therefore by the observation of
\re{berncent} (c), we conclude that
\[
[A](v)=
[A](\dt_{{P}^+}(v))=([A](\dt_{{P}^+}))(v)=\dt_{{P}^+}(v)=v.
\]

Assume now that $(\pi,V)\in\Irr(G)$ is of depth $r>0$. Then, by the theory of Moy--Prasad \cite{MP,MP2},
there exists a congruence subgroup ${P}_{x,r}\subset G(F)$ and a
non-degenerate character $\xi:{P}_{x,r}/{P}_{x,r^+}\to\qlbar$ such
that the $\xi$-isotypical component
$V_{\xi}:=\Hom_{{P}_{x,r}}(\xi,V)$ is non-zero. It remains to show
that $[A](v)=0$ for every $v\in V_{\xi}$. Observe that $\xi$
defines a smooth measure $h_{\xi}:=\xi^{-1}\dt_{{P}_{x,r}}\in\C{H}(G(F))$, supported on ${P}_{x,r}$.
Then $h_{\xi}(v)=v$ for every $v\in V_{\xi}$, hence arguing as in
the depth zero case, it is enough to show that
$[A](h_{\xi})=0$.

By the definition of $\lan A\ran$ (see \rco{stab} (a)), $[A](h_{\xi})$
equals $[A^Y]\ast h_{\xi}$ for each
sufficiently large $Y\in\T$, and also
$[A^Y]=\sum_{\P\in{\Par}}(-1)^{\rk G-\rk
\P}[A_{\P}^Y]$.

It suffices to show that $[A_{\P}^Y]\ast h_{\xi}=0$ for every
$Y\in\T$ and $\P\in {\Par}$. Since $A_{\P}^Y=\Av^{Y_{\P}}(\dt_{\frac{\P^+}{\P}})$,
it suffices to check that for every $g\in G(F)$ the restriction of
$\xi$ to ${P}_{x,r}\cap g{P}^+g^{-1}$ is non-trivial.

Since the latter assertion is implicit in the argument of Moy--Prasad (see \cite[7.2, Case 2]{MP}), we outline the proof, using the notation of \cite{MP}.

Set $r_i:=r$ and choose a number $r_{i+1}>r_i$ and a point $y$ of the Bruhat--Tits building of $G(F)$ such that ${P}_{x,r_{i+1}}={P}_{x,r^+}$ and ${P}={P}_{y,0}$. Thus
$g{P}^+g^{-1}={P}_{g(y),0^+}$.

Let $\xi$ corresponds to the coset $X+\C{G}^*_{x,-r_{i-1}}\subset \C{G}^*_{x,-r_{i}}$, which does not contain nilpotent elements, because $\xi$ is non-degenerate. Assume that
the restriction of $\xi$ to ${P}_{x,r}\cap {P}_{g(y),0^+}$ is trivial. Then
$X\in \C{G}^*_{x,-r_{i-1}}+\C{G}^*_{g(y),0}$, thus $(X+\C{G}^*_{x,-r_{i-1}})\cap \C{G}^*_{g(y),0}\neq\emptyset$. Since $r_i=r>0$, this contradicts to \cite[Prop. 6.4]{MP}.
\end{proof}
%

Our next goal is to write down an explicit formula for the restriction $\nu_{z^0}|_{G^{rss}(F)}$
and to deduce its stability.

\begin{Emp} \label{E:affspr}
{\bf Homology of the affine Springer fibers.}

(a) For each $\gm\in
G^{rss}(F)$, we denote by $\Fl_{\gm}\subset \Fl$ the corresponding
affine Springer fiber, and by $H_i(\Fl_{\gm})=H_i(\Fl_{\gm},\qlbar)$
its $i$-th homology group. More precisely, for each $Y\in\T$, we set
$\Fl_{\gm}^Y:=\Fl_{\gm}\cap Y$, $H_i(\Fl^Y_{\gm}):=  H^{-i}(\Fl_{\gm}^Y,\B{D}_{\Fl_{\gm}^Y})$, and
we define $H_i(\Fl_{\gm}):=\colim_Y H_i(\Fl_{\gm}^Y)$.


(b) As it was observed by Lusztig (\cite{Lu2}), each
$H_i(\Fl_{\gm})$ is equipped with an action of the affine Weyl
group $\wt{W}$ of $G$ (compare \cite{BV}).

(c) The natural action of the loop ind-group $LG$ on $\Fl$ induces an
action of the centralizer $LG_{\gm}$ on the affine Springer
fiber $\Fl_{\gm}$, hence on its homology groups $H_i(\Fl_{\gm})$.
Moreover, the actions of $\wt{W}$ and $LG_{\gm}$ on
$H_i(\Fl_{\gm})$ commute.

(d) Set $F^{nr}:=\fqbar((t))$, and let $\Gm$ be the Galois group $\Gal(\ov{F}/F^{nr})$. Then $\Gm$ acts
on the group of cocharacters $X_*(G_{\gm})$ of $G_{\gm}$, so we can form the group of coinvariants
$\La_{\gm}:=X_*(G_{\gm})_{\Gm}$. Using Kottwitz lemma \cite[Lem 2.2]{Ko},
we have a natural isomorphism $\La_{\gm}\cong \pi_0(LG_{\gm})$.
Therefore there exists a canonical surjective
homomorphism $LG_{\gm}\to\La_{\gm}$ such that the action of $LG_{\gm}$ on
$H_i(\Fl_{\gm})$ factors through $\La_{\gm}$.

(e) By (b)-(d), the homology groups $H_i(\Fl_{\gm})$ are equipped
with commuting actions  of $\wt{W}$ and $\La_{\gm}$.
\end{Emp}

\begin{Emp} \label{E:canmap}
{\bf The canonical map.} (a) Set $\La:=X_*(T)$, and recall that $\wt{W}$ is the
semidirect product $\wt{W}=\La\rtimes W$.

(b) Since $G_{\gm}\subset G$ is a maximal torus, we have a natural
isomorphism $\varphi:T\isom G_{\gm}$ over $\ov{F}$, defined up to $W$-conjugacy.
Any such $\varphi$ induces a homomorphism
\[
\pi_{\varphi}:\La=X_*(T)\isom X_*(G_{\gm})\to \La_{\gm},
\]
hence an algebra homomorphism
\[
\pi_{\gm}:\qlbar[\La]^{W}\lra
\qlbar[\La]\overset{\pi_{\varphi}}{\lra}\qlbar[\La_{\gm}].
\]
Notice that since $\varphi$ is defined uniquely up to a $W$-conjugacy, the
homomorphism $\pi_{\varphi}$ is unique up to a $W$- conjugacy,
therefore $\pi_{\gm}$ is independent of $\varphi$.

(c) Note that $\Spec \qlbar[\La]$ is canonically the dual torus $\check{T}$ of $T$, $\Spec \qlbar[\La_{\gm}]$ is $(\check{G}_{\gm})^{\Gm}$, and
$\Spec \qlbar[\La]^W=\check{T}/W$ is the Chevalley space $c_{\check{G}}$ of $\check{G}$. Then
$\pi_{\gm}$ corresponds to the canonical morphism $\check{\pi}_{\gm}:(\check{G}_{\gm})^{\Gm}\hra\check{G}_{\gm}\to  c_{\check{G}}$.
\end{Emp}

The following result asserts that the actions of $\wt{W}$ and
$\La_{\gm}$ from \re{affspr} are "compatible". Assume that $char(k)>2h$,
where $h$ is Coxeter number (see \cite[1.11]{Yun}).

\begin{Thm} \label{T:action}
There exists a finite $\wt{W}\times\La_{\gm}$-equivariant
filtration $\{F^j H_i(\Fl_{\gm})\}_j$ of $H_i(\Fl_{\gm})$ such that
the action of $\qlbar[\La]^{W}\subset \qlbar[\La]$ on each graded
piece $\gr^j H_i(\Fl_{\gm})$ is induced from the action of
$\qlbar[\La_{\gm}]$ via homomorphism $\pi_{\gm}$.
\end{Thm}

\begin{proof}[Sketch of a proof]
The main ingredient of the proof is an analogous result of Yun
(\cite[Thm 1.3]{Yun}) for Lie algebras.

The argument of Yun is very involved. First he treats the case when the reduction
$\ov{x}\in \Lie I/\Lie I^{+}$ of $x\in \Lie I\cap (\Lie G)^{rss}(F)$ is regular.
In this case, the affine Springer fiber $\Fl_{x}$ is discrete, and Lusztig's action
of $\wt{W}$ on $H_i(\Fl_{x})$ comes from an action of $\wt{W}$ on $\Fl_{x}$.
Moreover, the restriction of this action to $\La\subset\wt{W}$ coincides with the geometric
action of $\La_{x}=\La$.

To show the result in general, Yun uses global method, extending
some of the results of Bao Chau Ngo (\cite {Ngo}).
To deduce the assertion for groups, we use quasi-logarithms and
the topological Jordan decomposition (see \cite{BV}).
\end{proof}

\begin{Cor} \label{C:fg}
For each $\gm\in G^{rss}(F)$ and $i\in\B{Z}$, the homology group $H_i(\Fl_{\gm})$
is a finitely generated $\qlbar[\wt{W}]$-module.
\end{Cor}

\begin{proof}
First we claim that  $H_i(\Fl_{\gm})$ is a finitely generated $\qlbar[\La_{\gm}]$-module.
Indeed, set $\ov{\La}_{\gm}:=X_*(G_{\gm})^{\Gm}$. Then we have a natural embedding
$\ov{\La}_{\gm}\hra LG_{\gm}$, the group $\ov{\La}_{\gm}$ acts on $\Fl_{\gm}$ discretely, and
the quotient $\ov{\La}_{\gm}\bs \Fl_{\gm}$ is a projective scheme (by \cite{KL}). This implies that
$H_i(\Fl_{\gm})$ is a finitely generated $\qlbar[\ov{\La}_{\gm}]$-module, thus a finitely generated
$\qlbar[\La_{\gm}]$-module (see \cite{BV} for details).

Now the assertion follows from \rt{action}. Indeed, since $\qlbar[\La_{\gm}]$ is Noetherian, each graded piece
$\gr^j H_i(\Fl_{\gm})$ is a finitely generated $\qlbar[\La_{\gm}]$-module, hence a finitely generated $\qlbar[\La]^W$-module
(by \rt{action}). Thus  $H_i(\Fl_{\gm})$ is a finitely generated $\qlbar[\La]^W$-module, hence a finitely generated
$\qlbar[\wt{W}]$-module.
\end{proof}

\begin{Emp}
{\bf Remark.} This statement appears also as Conjecture 3.6 in \cite{Lu4}.
It is also mentioned in {\em loc. cit.} that the statement should follow
from the result of \cite{Yun}.
\end{Emp}
\begin{Emp} \label{E:fg}
{\bf Observations.} (a) It follows from Theorems \ref{T:rss} and \ref{T:conjunit} that
the restriction of $\nu_{z^0}=\nu_{[A]}$ to $G^{rss}(F)$ has
the form $\phi_{z^0}\mu^{I^+}$ for some $\phi_{z^0}\in
C^{\infty}(G^{rss}(F))$, where $\mu^{I^+}$ was defined in \re{haarmeas}.

(b) Using  \rco{fg} and arguing as in \re{remfd}, we see that
$\Tor_j^{\wt{W}}(H_i(\Fl_{\gm}),\sgn)$ is a finite-dimensional $\qlbar$-vector space,
equipped with an action of $\Gal(\fqbar/\fq)$.  In particular, we can consider trace
$\Tr(\Fr,\Tor_j^{\wt{W}}(H_i(\Fl_{\gm}),\sgn))\in\qlbar$.
\end{Emp}

Taking into account \rt{conjunit}, the following result is
\rcn{conj2} (b) for the unit element. Its proof will occupy Section \ref{S:tf}.

\begin{Thm} \label{T:tf}
For each topologically unipotent $\gm\in G^{rss}(F)$,  we have an
equality
\begin{equation} \label{Eq:equal}
\phi_{z^0}(\gm)=\sum_{i,j}(-1)^{i+j}
\Tr(\Fr,\Tor_j^{\wt{W}}(H_i(\Fl_{\gm}),\sgn)).
\end{equation}
\end{Thm}

The following result is a particular case of \rcn{conj3} (b).

\begin{Thm} \label{T:stable}
The restriction $\nu_{z^0}|_{G^{rss}(F)}$ is stable.
\end{Thm}

\begin{proof}
We have to show that $\phi_{z^0}(\gm)=\phi_{z^0}(\gm')$ for every pair of stably conjugate elements
$\gm,\gm'\in G^{rss}(F)=LG^{rss}(\fq)$.

First we claim that $\nu_{z_0}=\nu_{[A]}$ is supported on topologically unipotent elements.
Indeed, this follows from the fact that $[A^Y]=\sum_{\P\in{\Par}}(-1)^{\rk G-\rk
\P}[A_{\P}^Y]$ and that  $A_{\P}^Y={\Av}^{Y_{\P}}(\dt_{\frac{\P^+}{\P}})$ is supported
on elements, conjugate to $\P^+$. Therefore  $\phi_{z^0}(\gm)$ and $\phi_{z^0}(\gm')$ are vanish, unless
$\gm$ and $\gm'$ are topologically unipotent.

Assume now that $\gm$ and $\gm'$ are topologically unipotent. Then, by \rt{tf}, it suffices to show that for each $i,j$ we have an equality
\[
\Tr(\Fr,\Tor_j^{\wt{W}}(H_i(\Fl_{\gm}),\sgn))=\Tr(\Fr,\Tor_j^{\wt{W}}(H_i(\Fl_{\gm'}),\sgn)).
\]
Since $\gm,\gm'\in LG^{rss}(\fq)=G^{rss}(F)$ are stably conjugate and $H^1(F^{nr},G_{\gm})=1$, there exists
$g\in LG(\fqbar)=G(F^{nr})$ such that $g\gm g^{-1}=\gm'$. Then $h:=g^{-1}{}^{\Fr}g$ belongs to $LG_{\gm}$, and element $g$ induces an isomorphism $\Fl_{\gm}\isom \Fl_{\gm'}$, hence an isomorphism $\Tor_j^{\wt{W}}(H_i(\Fl_{\gm}),\sgn))\isom \Tor_j^{\wt{W}}(H_i(\Fl_{\gm'}),\sgn))$. Using the identity $g^{-1}\circ\Fr\circ g=h\circ\Fr$,
we conclude that
\[
\Tr(\Fr,\Tor_j^{\wt{W}}(H_i(\Fl_{\gm'}),\sgn))= \Tr(h\circ \Fr,\Tor_j^{\wt{W}}(H_i(\Fl_{\gm}),\sgn))
\]
(see \cite{BV} for more details). Hence it suffices to show an equality
\[
\Tr(h\circ \Fr,\Tor_j^{\wt{W}}(H_i(\Fl_{\gm}),\sgn))=\Tr(\Fr,\Tor_j^{\wt{W}}(H_i(\Fl_{\gm}),\sgn)).
\]

It is therefore enough to show that there exists a $\lan\Fr\ran\ltimes LG_{\gm}$-equivariant filtration
on $\Tor_j^{\wt{W}}(H_i(\Fl_{\gm}),\sgn)=H_j(\La,H_i(\Fl_{\gm}))_{W,\sgn}$ such that
$\pi_0(LG_{\gm})\cong\La_{\gm}$ acts trivially on each graded piece. It suffices to show that
each element of $\La_{\gm}$ acts on $H_j(\La,H_i(\Fl_{\gm}))$ unipotently.
Geometrically this means that the coherent sheaf $\C{F}$ on $(\check{G}_{\gm})^{\Gm}$,
corresponding to the  $\qlbar[\La_{\gm}]$-module $H_j(\La,H_i(\Fl_{\gm}))$ (see \re{canmap} (c)),
is supported at $1\in  (\check{G}_{\gm})^{\Gm}$.

By definition, the group $\La$ acts trivially on $H_j(\La,H_i(\Fl_{\gm}))$.
Geometrically this means that the coherent sheaf $\C{F}'$ on $\check{T}$, corresponding to $H_j(\La,H_i(\Fl_{\gm}))$, is a direct sum of $\dt$-sheaves at the identity. The restriction of $H_j(\La,H_i(\Fl_{\gm}))$ to $\qlbar[\La]^W$ corresponds to the pushforward $\pr_*\C{F}'$, where
$\pr$ is the projection $\check{T}\to\check{T}/W=c_{\check{G}}$. In particular, $\pr_*\C{F}'$ is supported at $[1]\in c_{\check{G}}$.

Since the actions of $\qlbar[\La]^W$ and $\La_{\gm}$ of $H_i(Fl_{\gm})$ are compatible (\rt{action}),
the induced actions on $H_j(\La,H_i(\Fl_{\gm}))$ are compatible as well. Geometrically this means that
$\C{F}$ and $\check{\pi}_{\gm}^*(\pr_*\C{F}')$ have filtrations with isomorphic graded pieces. In particular, the support
of $\C{F}$ equals the support of $\check{\pi}_{\gm}^*(\pr_*\C{F}')$, which is
$\check{\pi}_{\gm}^{-1}([1])=\{1\}$.
\end{proof}
\begin{Rem}
We believe that using an extension of a theorem of Harish-Chandra, the stability of
$\nu_{z^0}$ formally follows from the stability of $\nu_{z^0}|_{G^{rss}(F)}$.
\end{Rem}

\section{Proof of \rt{tf}} \label{S:tf}

\subsection{Reformulation of the problem} In this subsection we will reduce \rt{tf} to a certain
equality in a $K$-group (see \rco{reduction}). First we need to introduce several constructions.
As before, we can assume that $G$ is simple.


\begin{Emp} \label{E:verspr}
{\bf Springer theory revisited.}

(a) For $\P\in \Par$, we set $n_{\P}:=\dim \I^+/\P^+$ and $\ov{\dt}_{\P^+}:=1_{\frac{\P^+}{\P}}[2n_P](n_P)\in D(\frac{\P}{\P})$. Every embedding $\eta:\P\hra\Q$ in $\Par$ induces a map $[\eta]:\frac{\P}{\P}\to\frac{\Q}{\Q}$
and a closed embedding $\ov{\eta}:\frac{\Q^+}{\P}\hra\frac{\P^+}{\P}$. Note that $[\eta]^*\ov{\dt}_{\Q^+}\cong 1_{\frac{\Q^+}{\P}}[2n_Q](n_Q)\in D(\frac{\P}{\P})$, thus
we have natural isomorphisms $([\eta]^*\ov{\dt}_{\Q^+})|_{\frac{\Q^+}{\P}}\cong\ov{\eta}^!(\ov{\dt}_{\P^+}|_{\frac{\P^+}{\P}})$
and $([\eta]^*\ov{\dt}_{\Q^+})|_{\frac{\P^+}{\P}}\cong \ov{\eta}_!\ov{\eta}^!(\ov{\dt}_{\P^+}|_{\frac{\P^+}{\P}})$, hence a bijection
\[
\Hom_{D(\frac{\P}{\P})}([\eta]^*\ov{\dt}_{\Q^+},\ov{\dt}_{\P^+})\cong \Hom_{D(\frac{\P^+}{\P})}(\ov{\eta}_!\ov{\eta}^!(\ov{\dt}_{\P^+}|_{\frac{\P^+}{\P}}),\ov{\dt}_{\P^+}|_{\frac{\P^+}{\P}}).
\]
In particular, we have a natural morphism $[\eta]^*\ov{\dt}_{\Q^+}\to\ov{\dt}_{\P^+}$, corresponding to the counit map $c:\ov{\eta}_!\ov{\eta}^!\to\Id$.

(b) For every $\eta:\P\hra\Q$ in $\Par$ as in (a), we have isomorphisms
\[
\RHom_{D(\frac{\P}{\P})}([\eta]^*\ov{\dt}_{\Q^+},\ov{\dt}_{\P^+})\cong \RHom_{D(\frac{\P^+}{\P})}(\ov{\eta}_!\ov{\eta}^!(\ov{\dt}_{\P^+}|_{\frac{\P^+}{\P}}),\ov{\dt}_{\P^+}|_{\frac{\P^+}{\P}})\cong
\]\[
\cong \RHom_{D(\frac{\Q^+}{\P})}(\ov{\eta}^!(\ov{\dt}_{\P^+}|_{\frac{\P^+}{\P}}),\ov{\eta}^!(\ov{\dt}_{\P^+}|_{\frac{\P^+}{\P}}))\cong \RHom_{D(\frac{\Q^+}{\P})}(1_{\frac{\Q^+}{\P}},1_{\frac{\Q^+}{\P}}).
\]
Note that the projection $p:\frac{\Q^+}{\P}\to \frac{1}{L_{\P}}$ is pro-unipotent (use \rl{unipstack} (a)), thus the functor
$p^*$ is fully faithful (by \rl{ff}), hence we have an isomorphism.
\[
\RHom_{D(\frac{\P}{\P})}([\eta]^*\ov{\dt}_{\Q^+},\ov{\dt}_{\P^+})\cong \RHom_{D(\frac{1}{L_{\P}})}
(1_{\frac{1}{L_{\P}}},1_{\frac{1}{L_{\P}}}).
\]
 In particular, we have
$\dim_{\qlbar}\Hom([\eta]^*\ov{\dt}_{\Q^+},\ov{\dt}_{\P^+})=\dim_{\qlbar}\Hom
(1_{\frac{1}{L_{\P}}},1_{\frac{1}{L_{\P}}})=1$ and \\
$\Ext^{-j}([\eta]^*\ov{\dt}_{\Q^+},\ov{\dt}_{\P^+})=\Ext^{-j}(1_{\frac{1}{L_{\P}}},1_{\frac{1}{L_{\P}}})=0$ for all $j>0$.

(c) Let $\eta$ be the embedding $\I\hra \P$. Then classical Springer theory implies that
$[\eta]_*\ov{\dt}_{\I^+}\in D(\frac{\P}{\P})$ is equipped with $W_{\P}$-action, and
$([\eta]_*\ov{\dt}_{\I^+})^{W_{\P},\sgn}\cong \ov{\dt}_{\P^+}$ (see the proof of \rl{weyl}).
Moreover, since $\Hom(\ov{\dt}_{\P^+},[\eta]_*\ov{\dt}_{\I^+})\cong \Hom([\eta]^*\ov{\dt}_{\P^+},\ov{\dt}_{\I^+})$ is one-dimensional (by (b)), we can normalize the isomorphism $\ov{\dt}_{\P^+}\cong ([\eta]_*\ov{\dt}_{\I^+})^{W_{\P},\sgn}$ such that the composition $\ov{\dt}_{\P^+}\isom ([\eta]_*\ov{\dt}_{\I^+})^{W_{\P},\sgn}\hra [\eta]_*\ov{\dt}_{\I^+}$ corresponds by adjointness to the morphism $[\eta]^*\ov{\dt}_{\P^+}\to\ov{\dt}_{\I^+}$, defined in (a).
\end{Emp}

\begin{Emp} \label{E:aplasf}
{\bf Application to the affine Springer fibers.} Let $\gm\in G^{rss}(F)=LG^{rss}(\fq)$ be a topologically unipotent element.

(a) For $\P\in \Par$, we let $\Fl_{\gm,\P}\subset\Fl_{\P}$ be the ind-scheme of $\gm$-fixed points.
Denote by $\pr_{\gm,\P}:\Fl_{\gm,\P}\to \frac{\P}{\P}$ the map $[g]\mapsto [g^{-1}\gm g]$, and set
$\wt{\dt}_{\P^+}:=\pr_{\gm,\P}^*(\ov{\dt}_{\P^+})\in D(\Fl_{\gm,\P})$. Notice that since $\gm$ is topologically unipotent, and $\ov{\dt}_{\I^+}=1_{\frac{\I^+}{\I}}$, we have  $\wt{\dt}_{\I^+}=1_{\Fl_{\gm,\P}}$.
For each $Y\in\T$, we set $\Fl^Y_{\gm,\P}:=\Fl_{\gm,\P}\cap Y_{\P}\subset \Fl_{\P}$.
We denote the $*$-restriction of $\wt{\dt}_{\P^+}$ to $\Fl^Y_{\gm,\P}$ simply by $\wt{\dt}_{\P^+}$.

(b) We claim that there exists a natural functor  $C^\star_{\star}(\gm):\T\times\Par\to D(k)$, which sends pair
$(Y,\P)$ to $C^Y_{\P}(\gm):=R\Gm(\Fl^Y_{\gm,\P},\B{D}(\wt{\dt}_{\P^+}))\cong \B{D}(R\Gm(\Fl^Y_{\gm,\P},\wt{\dt}_{\P^+}))$,
where $\B{D}$ denotes the Verdier duality.

Indeed, for every $Y\subset Y'$ in $\T$, we have a closed embedding  $\Fl^Y_{\gm,\P}\hra \Fl^{Y'}_{\gm,\P}$. We define the map $C^Y_{\P}(\gm)\to C^{Y'}_{\P}(\gm)$ to be the Verdier dual of the
restriction map $R\Gm(\Fl^{Y'}_{\gm,\P},\wt{\dt}_{\P^+})\to R\Gm(\Fl^{Y}_{\gm,\P},\wt{\dt}_{\P^+})$.

Next, every inclusion $\eta:\P\hra\Q$ in $\Par$ induces a projection $\wt{\eta}:\Fl^Y_{\gm,\P}\to \Fl^Y_{\gm,\Q}$, which gives rise to a map $\wt{\eta}^*\wt{\dt}_{\Q^+}\to \wt{\dt}_{\P^+}$, induced by the map  $[\eta]^*\ov{\dt}_{\Q^+}\to \ov{\dt}_{\P^+}$ from \re{verspr} (a). Therefore $\eta$ induces a map
\[
R\Gm(\Fl^Y_{\gm,\Q},\wt{\dt}_{\Q^+})\overset{\wt{\eta}^*}{\lra} R\Gm(\Fl^Y_{\gm,\P},\wt{\eta}^*\wt{\dt}_{\Q^+})\to
R\Gm(\Fl^Y_{\gm,\P},\wt{\dt}_{\P^+}),
\]
which by duality induces a map $C^Y_{\P}(\gm)\to C^Y_{\Q}(\gm)$.
Finally, we notice that the above maps are compatible with compositions.

(c)  Consider the inductive limit  $C_{\star}(\gm):=\indlim_{Y\in\T} C_{\star}^{Y}(\gm):\Par\to D(k)$.
It sends $\P$ to $C_{\P}(\gm)=R\Gm(\Fl_{\gm,\P}, \B{D}(\wt{\dt}_{\P^+}))$. In particular, $C_{\I}(\gm)=R\Gm(\Fl_{\gm},\B{D}_{\Fl_{\gm}})$.

(d) We claim that for every $\P\in \Par$ element $C_{\I}(\gm)\in D(k)$ is equipped with the action of $W_{\P}$, and the morphism $C_{\I}(\gm)\to C_{\P}(\gm)$ from (c) induces an isomorphism
$C_{\I}(\gm)_{W_{\P},\sgn}\isom C_{\P}(\gm)$.

By the Verdier duality, it suffices to show that $R\Gm(\Fl_{\gm},\wt{\dt}_{\I^+})$ is equipped with the action of $W_{\P}$, and the morphism $R\Gm(\Fl_{\gm,\P},\wt{\dt}_{\P^+})\to R\Gm(\Fl_{\gm},\wt{\dt}_{\I^+})$ induces an isomorphism $R\Gm(\Fl_{\gm,\P},\wt{\dt}_{\P^+})\isom R\Gm(\Fl_{\gm},\wt{\dt}_{\I^+})^{W_{\P},\sgn}$. Note that the following diagram is Cartesian
\[
\begin{CD}
\Fl_{\gm} @>\pr_{\gm}>> \frac{\I}{\I}\\
@V\wt{\eta}VV @V[\eta]VV\\
\Fl_{\gm,\P} @>\pr_{\gm,\P}>> \frac{\P}{\P},
\end{CD}
\]
thus the assertion follows from the Springer theory (\re{verspr} (c)) and the base change isomorphism.
\end{Emp}

\begin{Emp} \label{E:remasf}
{\bf A morphism.} Let $\gm\in LG^{rss}(\fq)$ be as in \re{aplasf}. For every $Y\subset Y'$ in $\T$ and $\P\in \Par$, we set
$\Fl^{Y';Y}_{\gm,\P}:=\Fl^{Y'}_{\gm,\P}\sm \Fl^{Y}_{\gm,\P}$.

For every $\P\subset\Q$ in $\Par$, the projection $\wt{\eta}:\Fl^{Y'}_{\gm,\P}\to \Fl^{Y'}_{\gm,\Q}$ is proper and satisfies
$\wt{\eta}(\Fl^{Y}_{\gm,\P})\subset \Fl^{Y}_{\gm,\Q}$. Thus we have an open embedding $\wt{\eta}^{-1}(\Fl^{Y';Y}_{\gm,\Q})\hra \Fl^{Y';Y}_{\gm,\P}$,
which together with a map $\wt{\eta}^*\wt{\dt}_{\Q^+}\to \wt{\dt}_{\P^+}$ from \re{aplasf} (b) induces a morphism
\begin{equation} \label{Eq:morph}
R\Gm_c(\Fl^{Y';Y}_{\gm,\Q},\wt{\dt}_{\Q^+})\overset{\wt{\eta}^*}{\lra} R\Gm_c(\wt{\eta}^{-1}(\Fl^{Y';Y}_{\gm,\Q}),\wt{\eta}^*\wt{\dt}_{\Q^+})
\to R\Gm_c(\Fl^{Y';Y}_{\gm,\P},\wt{\dt}_{\P^+}).
\end{equation}
\end{Emp}

\begin{Emp} \label{E:homol}
{\bf Cochain complexes.} Let $Sh(k)$ be the abelian category of \'etale $\qlbar$-sheaves on $\Spec k$, let $\Ch^*(Sh(k))$ be the category of cochain complexes over $Sh(k)$, let $Sh_c(k)\subset Sh(k)$ be the subcategory of constructible, that is finite-dimensional, $\qlbar$-modules, and let $\Ch^{*,b}_c(Sh(k))\subset \Ch^*(Sh(k))$ be a subcategory of bounded complexes with constructible cohomologies.

Fix $i\in\B{Z}$, and let $h_{i}:\Ch^*(Sh(k))\to Sh(k)$ be the $(-i)$-th cohomology functor.

(a) Consider functor $C_{i,\star}^{\star}(\gm):=h_i\circ C^{\star}_{\star}(\gm):\T\times\Par\to Sh(k)$, where $C^{\star}_{\star}(\gm)$ was constructed in \re{aplasf} (b). This functor sends $\P\in \Par$ and $Y\in \T$ to \\
$C^Y_{i,\P}(\gm)=H^{-i}(\Fl^Y_{\gm,\P}, \B{D}(\wt{\dt}_{\P^+}))\cong H_i(\Fl^Y_{\gm,\P}, \wt{\dt}_{\P^+})$.
By functoriality, we have morphisms $f^Y_{\P,\Q}:C^Y_{i,\P}(\gm)\to C^Y_{i,\Q}(\gm)$ for all $\P\subset\Q$ in $\Par$
and $Y\in\T$.

(b) We choose a bijection between $\wt{\Dt}$ and $[r]:=\{0,\ldots,r\}$. For every
$J\subsetneq\wt{\Dt}=[r]$  and $m\in [r]\sm J$, we define $\sgn(J,m):=(-1)^{|j\in [r]\sm J,j<m|}$.

(c) Consider functor $C_{i}^{\star}(\gm):\T\to \Ch^*(Sh(k))$, which associates to each $Y\in\T$ a cochain complex
\[
C^Y_i(\gm):=(0\to C^{Y,-r}_{i}(\gm)\to \ldots \to C^{Y,0}_{i}(\gm)\to 0),
\]
where $C^{Y,-j}_{i}(\gm):=\oplus_{J\subset\wt{\Dt},|J|=r-j}C^Y_{i,\P_J}(\gm)$ and the differentials
$C^{Y,-j}_{i}(\gm)\to C^{Y,-j+1}_{i}(\gm)$ are the direct sums of maps
$\sgn(J,m)f^Y_{J,J\cup\{m\}}:C^Y_{i,\P_J}(\gm)\to C^Y_{i,\P_{J\cup m}}(\gm)$, taken over subsets $J\subset \wt{\Dt}$ with  $|J|=r-j$ and $m\in [r]\sm J$. Notice that our choice of signs in (b) implies that each $C^Y_i(\gm)$ is indeed a complex.

(d) Consider functor $C_{i,\star}(\gm):=\indlim_{Y\in\T} C_{i,\star}^{Y}(\gm)\cong h_i\circ  C_{\star}(\gm):\Par\to Sh(k)$. Explicitly,
$C_{i,\P}(\gm)=H_i(\Fl_{\gm},\wt{\dt}_{{P}^+})$, in particular, $C_{i,\I}(\gm)=H_i(\Fl_{\gm})$.
By \re{aplasf} (d), each $C_{i,\I}(\gm)$ is equipped with a $W_{\P}$-action, and the morphism
$C_{i,\I}(\gm)\to C_{i,\P}(\gm)$ induces an isomorphism $C_{i,\I}(\gm)_{W_\P,\sgn}\cong C_{i,\P}(\gm)$.

(e) Set $C_{i}(\gm):=\indlim_{Y\in\T} C^{\star}_{i}(\gm)\in \Ch^*(Sh(k))$. Explicitly, $C_{i}(\gm)$
is a complex
\[
C_i(\gm):=(0\to C^{-r}_{i}(\gm)\to \ldots \to C^0_{i}(\gm)\to 0),
\]
where $C_i^{-j}(\gm):=\oplus_{J\subset\wt{\Dt},|J|=r-j}C_{i,\P_J}(\gm)$, and the differentials are defined as in (c).
\end{Emp}

\begin{Lem} \label{L:hom}
We have a natural isomorphism $\Tor_j^{\wt{W}}(H_i(\Fl_{\gm}),\sgn)\cong H^{-j}(C_i(\gm))$ for each $i,j$. Therefore $C_i(\gm)\in \Ch^{*,b}_c(Sh(k))$ for all $i$.
\end{Lem}

\begin{proof}
For every $J\subset\wt{\Dt}$ and $m\in[r]\sm J$, we define the map
\[
g_{\P,m}:\qlbar[\wt{W}/W_{\P_J}]\to \qlbar[\wt{W}/W_{\P_{J\cup\{m\}}}]
\]
to be $\sgn(J,m)$ times the natural projection.

Consider the cochain complex $D_{\star}$ of $\qlbar[\wt{W}]$-modules
$0\to D^{-r}\to\ldots \to D^0\to 0$, where $D^{-j}=\oplus_{\P,\rk\P=r-j}\qlbar[\wt{W}/W_{\P}]$, and the differentials are the direct sum of the $g_{\P,m}$'s.

Note that $D_{\star}$ is a resolution of the trivial module $\qlbar$. Indeed, this follows from the fact that $D_{\star}$ is the simplicial chain complex of the affine space $\La\otimes_{\B{Z}}\B{R}$, triangulated by alcoves. Moreover, each $\qlbar[\wt{W}/W_{\P}]=\qlbar[\wt{W}]_{W_{\P}}$ is a direct summand of the free $\qlbar[\wt{W}]$-module $\qlbar[\wt{W}]$, thus $D_{\star}$ is a projective resolution of $\qlbar$.

By the definition of homology, for every $\qlbar[\wt{W}]$-module $V$, we have an isomorphism
$H_j(\wt{W},V)\cong H_j(D_{\star}\otimes_{\qlbar[\wt{W}]}V)$.
Since $\qlbar[\wt{W}/W_{\P}]\otimes_{\qlbar[\wt{W}]}V\cong V_{W_{\P}}$,  the tensor product
$D_{\star}\otimes_{\qlbar[\wt{W}]}V$ is simply the complex with terms
$\oplus_{\P,\rk\P=r-j}V_{W_{\P}}$.

 We claim that if $V=H_i(\Fl_{\gm})\otimes\sgn$, then the complex $D_{\star}\otimes_{\qlbar[\wt{W}]}V$ is isomorphic to $C_i(\gm)$.
 Indeed, for each $\P\in\Par$, we have a natural isomorphism $V_{W_{\P}}\cong H_i(\Fl_{\gm})_{W_{\P},\sgn}\cong(C_{i,\I})_{W_{\P},\sgn}\cong C_{i,\P}(\gm)$ (see \re{homol} (d)), and for every $\P\subset\Q$ the induced map $V_{W_{\P}}\to V_{W_{\Q}}$ is the map $f_{\P,\Q}$ from \re{homol}
(a). This implies the first assertion. The second assertion follows from the first one and \re{fg} (b).
\end{proof}

\begin{Emp}
{\bf Notation.}
(a) We denote by $Sh^{\Fr}(k)$, $D^{\Fr}(k)$ etc. the categories of Weil (Frobenius equivariant) objects.
Using the sheaf-function correspondence, to every $\lan A\ran\in K_0(D^{\Fr}(k))=K_0(Sh^{\Fr}(k))$
we can associate $[A]:=\Tr(\Fr,\lan A\ran)\in\qlbar$.

(b) Notice that since  $\gm\in G(F)$, all objects, defined in \re{aplasf} and \re{homol} are Frobenius equivariant. Using \rl{hom}, we have
$C_i(\gm),C_i^Y(\gm)\in \Ch^{*,b}_c(Sh^{\Fr}(k))$ for all $i$ and $Y\in\T$. In particular, we can form  $\lan C^Y_i(\gm) \ran\in K_0(D^{\Fr}(k))$ and
\[
\lan C_i(\gm) \ran=\sum_{j=0}^r(-1)^j\lan h_j(C_i(\gm))\ran\in K_0(D^{\Fr}(k)).\]
\end{Emp}

\begin{Emp} \label{E:rem}
{\bf Remark.} For every $\P\in \Par$ and every locally closed subscheme $Y\subset \Fl$, defined over $\fq$, let $\wt{A}_{\P}^Y\in D^{\Fr}(LG)$ be the unique object such that $\wt{A}_{\P}^Y\otimes\mu^{\I^+}=A_{\P}^Y$ (see \re{tensor}), and let $\wt{A}_{\P}^Y(\gm)\in D^{\Fr}(k)$ be the $*$-pullback of $\wt{A}_{\P}^Y$ to $\{\gm\}$. Arguing as in \re{pflconst}, and using the equality $\mu^{\P^+}=\mu^{\I^+}[2n_{\P}](n_{\P})$ of Haar measures, we get that $\wt{A}^Y_{\P}(\gm)=R\Gm_c(\Fl^Y_{\gm,\P},\wt{\dt}_{\P^+})\cong\B{D}(C_{\P}^Y(\gm))$.
\end{Emp}

\begin{Lem} \label{L:trfr}
(a) For every $i\in\B{Z}$, we have an equality
\[ [C_i(\gm)]=\sum_{j}(-1)^j\Tr(\Fr,\Tor_j^{\wt{W}}(H_i(\Fl_{\gm}),\sgn)).
\]

(b) For every $Y\in\T$ sufficiently large, we have an equality
$\phi_{z^0}(\gm)=\sum_i (-1)^i[C^Y_i(\gm)]$.
\end{Lem}

\begin{proof}
(a) follows immediately from \rl{hom}.

(b) For every $Y\in\T$ and $\P\in \Par$, let  $\wt{A}_{\P}^Y(\gm)\in D^{\Fr}(k)$ be as in \re{rem} and let $[\wt{A}_{\P}^Y(\gm)]$ be the corresponding element of $\qlbar$. By construction and \rt{conjunit}, for each sufficiently large $Y\in\T$, we have an equality  $\phi_{z^0}(\gm)=\sum_{\P}(-1)^{r-\rk\P}[\wt{A}^Y_{\P}(\gm)]$.

On the other hand, each $[C^Y_i(\gm)]$ equals $\sum_{\P}(-1)^{r-\rk\P} [H_i(\Fl^Y_{\gm,\P},\wt{\dt}_{\P^+})]$.
Therefore it remains to show that $[\wt{A}^Y_{\P}(\gm)]=\sum_i (-1)^i [H_i(\Fl^Y_{\gm,\P},\wt{\dt}_{\P^+})]$.
Using \re{rem}, we get $\wt{A}^Y_{\P}(\gm)\cong R\Gm(\Fl^Y_{\gm,\P},\wt{\dt}_{\P^+})$, hence the assertion follows from
the fact that $H_i(\Fl^Y_{\gm,\P},\wt{\dt}_{\P^+})$ is the dual of $H^i(\Fl^Y_{\gm,\P},\wt{\dt}_{\P^+})$, thus
$[H_i(\Fl^Y_{\gm,\P},\wt{\dt}_{\P^+})]=[H^i(\Fl^Y_{\gm,\P},\wt{\dt}_{\P^+})]$.
\end{proof}

\begin{Cor} \label{C:reduction}
In order to prove \rt{tf}, it is enough to show that for every sufficiently large $Y\in\T$ we have an equality in $K_0(D^{\Fr}(k))$
\begin{equation} \label{Eq:eq}
\sum_i (-1)^i\lan C_i(\gm) \ran=\sum_i (-1)^i\lan C^Y_i(\gm) \ran.
\end{equation}
\end{Cor}

\begin{proof}
The equation \form{eq} would imply the equality
\[
\sum_i (-1)^i[C_i(\gm)]=\sum_i (-1)^i[C^Y_i(\gm)]
\]
in $\qlbar$, whose right hand side equals
$\phi_{z^0}(\gm)$ (by \rl{trfr} (b)), and left hand side equals
$\sum_{i,j}(-1)^{i+j}\Tr(\Fr,\Tor_j^{\wt{W}}(H_i(\Fl_{\gm}),\sgn))$ (by \rl{trfr} (a)).
\end{proof}

\subsection{Proof of the theorem} \label{SS:compl}
In this subsection we prove \rt{tf}, assuming \rcl{functor}. Abusing the notation, we omit superscript  $\Fr$, so we will simply write
$Sh(k),D(k),\C{D}(k)$ etc. instead of  $Sh^{\Fr}(k),D^{\Fr}(k),\C{D}^{\Fr}(k)$, respectively.

\begin{Emp} \label{E:obs}
{\bf Observation.}
Let $[1]$ be the partially ordered set $\{0<1\}$, let $[1]^n$ be its $n$'th power, let $\ov{1}\in [1]^n$ be the maximal element, and set $([1]^n)':=[1]^n-\{\ov{1}\}$. Then
\[
([1]^n)'=(\{0\}\times [1]^{n-1})\bigsqcup_{(\{0\}\times([1]^{n-1})')}([1]\times ([1]^{n-1})').
\]

Let $\C{D}$ be a stable $\infty$-category, and $F:([1]^n)'\to \C{D}$ a functor of $\infty$-categories. Then we have an  equivalence
\[
\left(\colim_{\ov{x}\in ([1]^n)'}F(\ov{x})\right)\cong F(0,\ov{1})\bigsqcup_{\colim_{\ov{x}\in\{0\}\times([1]^{n-1})'}F(\ov{x})}\colim_{\ov{x}\in \{1\}\times ([1]^{n-1})'}F(\ov{x}).
\]
\end{Emp}

\begin{Lem} \label{L:colim}
Let $\C{D}$ be a stable $\infty$-category, and let $F:([1]^n)'\to \C{D}$ be a functor of $\infty$-categories such that
$F(0,\ov{1})\cong 0$, and for every $\ov{x}\in ([1]^{n-1})'$ the induced map
$F(0,\ov{x})\to F(1,\ov{x})$ is an equivalence.
Then $\colim_{\ov{x}\in ([1]^n)'}F(\ov{x})\cong 0$.
\end{Lem}

\begin{proof}
Using \re{obs} and the assumption $F(0,\ov{1})\cong 0$,
we have to show that the map $\colim_{\ov{x}\in([1]^{n-1})'}F(0,\ov{x})\to\colim_{\ov{x}\in([1]^{n-1})'}F(1,\ov{x})$ is an equivalence. It is enough to show that the map $F(0,\cdot)\to F(1,\cdot)$ in $\C{D}^{([1]^{(n-1)})'}$ is an equivalence. But this follows from the assumption that $F(0,\ov{x})\to F(1,\ov{x})$ is an equivalence for all $\ov{x}$.
\end{proof}

\begin{Emp} \label{E:cone}
{\bf Notation.} Let $A$ be an abelian category, let $\Ch^*(A)$ be the category of cochain complexes in $A$, and let $\C{D}(A)$ be
the corresponding derived $\infty$-category.  Then we have a natural functor of $\infty$-categories $\Ch^*(A)\to\C{D}(A)$,
which is a bijection on objects. In particular, every functor $F:J\to\Ch^*(A)$ of ordinary categories
gives rise to a functor $F:J\to\C{D}(A)$ of $\infty$-categories.

(a) Note that if $J$ is filtered, then the inductive limit $\indlim_J F\in\Ch^*(A)$ represents the (homotopy) colimit $\colim_J F\in\C{D}(A)$.

(b) For every $\ov{x}=(x_1,\ldots,x_n)\in [1]^n$, set $|\ov{x}|=\sum_{i} x_i$. For every $i=1,\ldots, n$, let $e_i\in[1]^n$ be the element $(\dt_{1,i},\ldots\dt_{n,i})$. For $\ov{x}$ and $i$ such that $x_i=0$, we set $\sgn(\ov{x},i)=(-1)^{|j<i, x_j=0|}$.

(c) Let $F:([1]^n)'\to \Ch^*(A)$ be a functor.
Consider the cochain complex $\wt{F}$ in $\Ch^*(A)$, defined as follow. As a graded module, $\wt{F}=\oplus_{\ov{x}\in  ([1]^n)'}F(\ov{x})[n-1-|\ov{x}|]$,
where $[\cdot]$ is the cohomological shift. Define the differential $d_{\wt{F}}:\wt{F}\to\wt{F}[1]$ to be the direct sum of differentials of the $F(\ov{x})[n-1-|\ov{x}|]$'s and the sum of \[
\sgn(\ov{x},i)F_{\ov{x},\ov{x}+e_i}:F(\ov{x})[n-1-|\ov{x}|]\to F(\ov{x}+e_i)[n-1-|\ov{x}|],
\] taken over all
$\ov{x}$ and $i$ such that $x_i=0$.
\end{Emp}

\begin{Lem} \label{L:cone}
In the notation of \re{cone} (c), the cochain complex $\wt{F}\in \Ch^*(A)$ represents the (homotopy) colimit
$\colim_{([1]^n)'}F\in \C{D}(A)$.
\end{Lem}

\begin{proof}
Recall that a (homotopy) pushout  $B_0\sqcup_A B_1$ of the diagram $B_0\overset{f_0}{\lla} A\overset{f_1}{\lra}B_1$ is equivalent to the pushout of the diagram
$0\lla A\overset{(-f_0,f_1)}{\lra}B_0\oplus B_1$, thus it is represented by the cone of
$(-f_0,f_1):A\to B_0\oplus B_1$. In other words, it is represented by a complex $B_0\oplus B_1\oplus A[1]$ with differential $d_{B_0}+d_{B_1}+d_{A[1]}-f_0[1]+f_1[1]$.

From this the assertion follows from observation \re{obs} by induction on $n$.
Indeed, for $i=0,1$, let $F|_i:([1]^{n-1})'\to\Ch^*(Sh(k))$ be the functor $F|_i(\ov{x})=F(i,\ov{x})$.
Then, by the induction hypothesis and \re{obs}, we have to show that the
homotopy pushout $F(0,\bar{1})\sqcup_{\wt{F|_0}} \wt{F|_1}$ is represented by $\wt{F}$.
But this follows from the observation in the beginning of the proof.
\end{proof}

\begin{Emp} \label{E:equivc}
{\bf Truncations.}
(a)  For every $-\infty\leq n\leq m\leq\infty$ we denote by $\C{D}^{[n,m]}(A)\subset\C{D}(A)$ the full
$\infty$-subcategory, whose objects are all $\C{F}$ in the homotopy category $D(A)$ such that $h^i(\C{F})=0$ for all $i>m$ and $i<n$, where
$h^i$ denotes the $i$'s cohomology. We also set $\C{D}^{[0]}(A):=\C{D}^{[0,0]}(A)\subset\C{D}(A)$.

(b) For $i\in\B{Z}$ we have truncation functors $\tau^{\leq i}: \C{D}(A)\to \C{D}^{\leq i}(A):=\C{D}^{(-\infty,i]}(A)$
and $\tau^{\geq i}: \C{D}(A)\to \C{D}^{\geq i}(A):=\C{D}^{[i,\infty)}(A)$ (see \cite[1.2.1.7]{Lur2} but note that we use here a cohomological notation
rather than homological). Moreover, we have a natural equivalence $\tau^{\geq i}\circ \tau^{\leq i}[i]\cong \tau^{\geq i}\circ \tau^{\leq i}[i]:\C{D}(A)\to \C{D}^{[0]}(A)$ (see \cite[Prop 1.2.1.10]{Lur2}), and we denote the resulting functor by $h^i$.

(c) Let $\C{C}\subset \C{D}(A)$ be a full $\infty$-subcategory, stable under translations, cofibers and truncation functors $\tau^{\leq i}$ and
$\tau^{\geq i}$. Then $\C{C}$ is a stable  $\infty$-category (see \cite[Lem 1.1.3.3]{Lur2}). Consider full $\infty$-subcategories
$\C{C}^{[n,m]}:=\C{C}\cap \C{D}^{[n,m]}(A)\subset \C{C}$ and $\C{C}^{[0]}:=\C{C}\cap \C{D}^{[0]}(A)\subset \C{C}$. Then
$\tau^{\leq i},\tau^{\geq i}$ and $h^i$ defines functors   $\tau^{\leq i}:\C{C}\to \C{C}^{\leq i},\tau^{\geq i}:\C{C}\to \C{C}^{\geq i}$ and
$h^i:\C{C}\to \C{C}^{[0]}$.

(d) {\bf Basic example.} Let $A=Sh(k)$ and $\C{C}=\C{D}(k)=\C{D}^b_c(Sh(k))\subset\C{D}(Sh(k))$. Then $\C{C}$ satisfies all the assumptions of (c).
\end{Emp}



\begin{Lem} \label{L:cohom}
Let $L$ be an ordinary category, $n\leq m$ be integers, and let $\C{C}\subset \C{D}(A)$ be a full $\infty$-subcategory satisfying
assumptions of \re{equivc} (c).

Let $F:L\to \C{C}^{[n,m]}$ be a functor of $\infty$-categories such that
for each $i\in\B{Z}$ the composition $h^i\circ F:L\to\C{C}^{[0]}\subset \C{C}$ has a colimit
$\colim_L (h^i\circ F)\in \C{C}$. Then there exists a colimit $\colim_L F\in \C{C}$, and we have an equality in $K_0(\C{C})$
\begin{equation} \label{Eq:kgroups}
\lan \colim_L F \ran=\sum_{i=n}^m(-1)^i \lan\colim_L (h^i\circ F)\ran.
\end{equation}
\end{Lem}

\begin{proof}
The proof goes by induction on $m-n$. If $m-n=0$, then $n=m$, hence $F=(h^n\circ F)[-n]$,
thus $\colim_L F=(\colim_L (h^n\circ F))[-n]$, and the assertion is clear.

Assume now that $m>n$. Then, by the induction hypothesis, we can assume that the lemma holds for
$\tau^{\leq m-1}\circ F:L\to\C{C}^{[n,m-1]}$. In other words, there exists $\colim_L(\tau^{\leq m-1}\circ F)\in\C{C}$,
which satisfies
\begin{equation} \label{Eq:kgroups2}
\lan \colim_L (\tau^{\leq m-1}\circ F) \ran=\sum_{i=n}^{m-1}(-1)^i \lan\colim_L (h^i\circ F)\ran.
\end{equation}

Since $\tau^{\geq m}\circ F=(h^m\circ F)[-m]$, we have a fiber sequence
\[
\tau^{\leq m-1}\circ F\to F\to (h^m\circ F)[-m]
\]
in $\C{C}^L$ (see \cite[Rem. 1.2.1.8]{Lur2}). Thus $F$ is a cofiber of $(h^m\circ F)[-m-1]\to \tau^{\leq m-1}\circ F$,
hence $\colim_L F\in\C{C}$ is a cofiber of  $\colim_L(h^m\circ F)[-m-1]\to \colim_L(\tau^{\leq m-1}\circ F)$ (compare  \cite[Prop. 1.1.4.1]{Lur2}). Since $\C{C}$ is stable, we conclude that $\colim_L F\in \C{C}$ exists and satisfies
\[
\lan\colim_L F\ran=\lan\colim_L (\tau^{\leq m-1}\circ F)\ran+(-1)^m\lan\colim_L(h^m\circ F)\ran.
\]
Thus equality \form{kgroups} follows from \form{kgroups2}.
\end{proof}

\begin{Emp} \label{E:discrete}
{\bf Discrete $\infty$-categories.}
(a) We call an $\infty$-category $X$ {\em discrete}, if for every two objects $x,y\in \Ob X$ the mapping space
$\map_X(x,y)$ satisfies $\pi_j(\map_X(x,y))=0$ for every $j>0$. Note that $X$ is discrete if and only if the natural functor
$X\to \Ho X$ from $X$ to its homotopy category is an equivalence of $\infty$-categories.

(b) Assume that $X$ is a full $\infty$-subcategory of a stable $\infty$-category $Y$. Then
\[
\pi_j(\map_X(x,y))\cong \pi_0(\map_Y(x[j],y))=\Ext^{-j}_{\Ho Y}(x,y)
\]
(see \cite[remark after 1.1.2.8, and 1.1.2.17]{Lur2}). Therefore $X$ is discrete, if and only if $\Ext^{-j}_{\Ho Y}(x,y)=0$
for every $x,y\in \Ob X$ and $j>0$. In particular, the full subcategory $\C{D}^{[0]}(A)\subset\C{D}(A)$
is discrete, thus the natural functor $\C{D}^{[0]}(A)\to\Ho \C{D}^{[0]}(A)=A$ is an equivalence.
\end{Emp}

\begin{Emp} \label{E:geomlift}
{\bf A geometric lift.}
(a) Let $C^\star_\star(\gm):\T\times\Par\to D(k)$ be the functor, constructed in \re{aplasf} (b). By a {\em lift} of $C^\star_\star(\gm)$,
we mean a functor of $\infty$-categories $\C{C}^\star_\star(\gm):\T\times\Par\to\C{D}(k)$ such that the induced functor on homotopy categories is   $C^\star_\star(\gm)$.

(b) Let $\C{C}^\star_\star(\gm)$ be any lift of $C^\star_\star(\gm)$. For $Y\in\T$, we denote by $\C{C}_{\star}^Y(\gm):\Par\to \C{D}(k)$
the restriction of $\C{C}^\star_\star(\gm)$  to $Y\in\T$. For every $Y\subset Y'$ in $\T$, we have a morphism $\C{C}_{\star}^Y(\gm)\to\C{C}_{\star}^{Y'}(\gm)$,
defined uniquely up to an equivalence.

(c) Let $\C{C}_{\star}^{Y';Y}(\gm)\in \C{D}(k)^{\Par}$ be the cofiber of $\C{C}^Y_{\star}(\gm)\to\C{C}^{Y'}_{\star}(\gm)$.
Then $\C{C}_{\star}^{Y';Y}(\gm)$ defines a functor $C_{\star}^{Y';Y}(\gm):\Par\to D(k)$ of usual categories.
In particular, for every $\P\subset\Q$ in $\Par$, we have a well-defined morphism ${C}_{\P}^{Y';Y}(\gm)\to{C}_{\Q}^{Y';Y}(\gm)$ in the homotopy category $D(k)$.

(d) Since $\C{C}^{Y}_{\P}(\gm)=\B{D}(R\Gm(\Fl_{\gm,\P}^Y,\wt{\dt}_{\P^+}))$, and $\C{C}^{Y';Y}_{\P}(\gm)$ is a cofiber of
$\C{C}^{Y}_{\P}(\gm)\to \C{C}^{Y'}_{\P}(\gm)$, we have a natural isomorphism ${C}^{Y';Y}_{\P}(\gm)\cong \B{D}(R\Gm_c(\Fl_{\gm,\P}^{Y';Y},\wt{\dt}_{\P^+}))$ (see \re{remasf}).

(e) We call a lift  $\C{C}^\star_\star(\gm)$ of $C^\star_\star(\gm)$ {\em geometric}, if for every $Y\subset Y'$ in $\T$ and every
$\P\subset \Q$ in $\Par$, the isomorphism of (d) identifies the morphism ${C}_{\P}^{Y';Y}(\gm)\to{C}_{\Q}^{Y';Y}(\gm)$ constructed in (c)
with the Verdier dual of the morphism \form{morph} from \re{remasf}.

(f) For each $Y\in\T$, we set $\C{C}^Y(\gm):=\colim_{\P\in\Par}\C{C}_{\star}^Y(\gm)\in \C{D}(k)$.
The the morphism from (b) induces the morphism  $\C{C}^Y(\gm)\to\C{C}^{Y'}(\gm)$, defined up to an equivalence.
\end{Emp}

The following claim will be proven in the next subsection.

\begin{Cl} \label{C:functor}
There exists a geometric lift $\C{C}^\star_\star(\gm)$ of $C^\star_\star(\gm)$.
\end{Cl}

\begin{Emp}
{\bf Remark.} We do not know whether a geometric lift in our sense is unique.
\end{Emp}

\begin{Prop} \label{P:stab}
Let $\C{C}^\star_\star(\gm)$ be a geometric lift of $C^\star_\star(\gm)$. Then
for every sufficiently large $Y\in\T$ and every $Y'\supset Y$ in $\T$, the map $\C{C}^Y(\gm)\to\C{C}^{Y'}(\gm)$ (see \re{geomlift} (f)) is an equivalence.
\end{Prop}

\begin{proof}
Choose $n\in\B{Z}_{>0}$ as in \rl{const}. We claim that for every $Y'\supset Y\supset Y(\I_n^+)$ in $\T$ (see notation \re{not} (e)), the map
$\C{C}^Y(\gm)\to\C{C}^{Y'}(\gm)$ is an equivalence. By the induction on the number of $\I$-orbits in $Y'\sm Y$, we can assume that $Y'\sm Y=Y_w=Iw$, and $w\notin S(\I_n^+)$.
Let $\C{C}^{Y';Y}(\gm)$ be the cofiber of $\C{C}^Y(\gm)\to\C{C}^{Y'}(\gm)$. We want to show that  $\C{C}^{Y';Y}(\gm)\cong 0$.

Let $\C{C}_{\star}^{Y';Y}(\gm)\in \C{D}(k)^{\Par}$ be the cofiber of
$\C{C}^Y_{\star}(\gm)\to\C{C}^{Y'}_{\star}(\gm)$. Since cofibers commute with colimits,
we have  $\C{C}^{Y';Y}(\gm)\cong\colim_{\P\in\Par} \C{C}^{Y';Y}_{\P}(\gm)$. Thus it remains to show that
$\colim_{\P\in\Par} \C{C}^{Y';Y}_{\P}(\gm)\cong 0$.

Since $w\notin S(\I_n^+)$, there exists $\al\in\wt{\Dt}$ such that $w(\al)\in \I_n^+$.
Fix a bijection $\wt{\Dt}\isom\{1,\ldots, r+1\}$, which maps $\al$ to $r+1$, and the
induced bijection $\Par\isom ([1]^{r+1})'$. It remans to show that
the functor $([1]^{r+1})'\to\C{D}(k)$, corresponding to $\C{C}^{Y';Y}_{\star}(\gm)$,
satisfies  the assumptions of \rl{colim}. In other words, we have to show that
$\C{C}^{Y';Y}_{\P_{\wt{\Dt}\sm\al}}(\gm)\cong 0$, and that for every
$J\subsetneq\wt{\Dt}\sm\al$ the map
\begin{equation} \label{Eq:equiv}
\C{C}^{Y';Y}_{\P_J}(\gm)\to \C{C}^{Y';Y}_{\P_{J\cup\al}}(\gm)
\end{equation}
is an equivalence.

For every $J'\nsubseteq J_w$, we have $Y_{\P_{J'}}=Y'_{\P_{J'}}$, hence $\Fl_{\gm,\P_{J'}}^{Y';Y}=\emptyset$, thus $\C{C}^{Y';Y}_{\P_{J'}}(\gm)\cong 0$ (see \re{geomlift} (d)). This implies that $\C{C}^{Y';Y}_{\P_{\wt{\Dt}\sm\al}}(\gm)\cong 0$
(because $\wt{\Dt}\sm \al\nsubseteq J_w$) and that the map \form{equiv} is an equivalence when $J\nsubseteq J_w\sm\al$. Assume now that $J\subset J_w\sm\al$.

Since a lift $\C{C}^{\star}_{\star}(\gm)$ is geometric,
the map \form{equiv} is the Verdier dual of the map
\begin{equation} \label{Eq:equiv1}
R\Gm_c(\Fl_{\gm,\P_{J\cup\al}}^{Y';Y},\wt{\dt}_{\P_{J\cup\al}^+})\to R\Gm_c(\Fl_{\gm,\P_J}^{Y';Y},\wt{\dt}_{\P_J^+}),
\end{equation}
constructed in \re{remasf}. It remains to show that the map \form{equiv1} is an equivalence.

Since $J\subset J_w\sm\al$,  the projection $\wt{\eta}:\Fl^{Y'}_{\gm,\P_J}\to \Fl^{Y'}_{\gm,\P_{J\cup\al}}$ from \re{remasf} satisfies
$\wt{\eta}^{-1}(\Fl_{\gm,\P_{J\cup\al}}^{Y';Y})=\Fl_{\gm,\P_J}^{Y';Y}$, and the pullback map
\begin{equation*} \label{Eq:eta}
\wt{\eta}^*:R\Gm_c(\Fl_{\gm,\P_{J\cup\al}}^{Y';Y},\wt{\dt}_{\P_{J\cup\al}^+})\to R\Gm_c(\Fl_{\gm,\P_J}^{Y';Y},\wt{\eta}^*\wt{\dt}_{\P_{J\cup\al}^+})
\end{equation*}
is an equivalence. Hence it remains to show that the morphism
\begin{equation} \label{Eq:can}
R\Gm_c(\Fl_{\gm,\P_J}^{Y';Y},\wt{\eta}^*\wt{\dt}_{\P_{J\cup\al}^+})\to R\Gm_c(\Fl_{\gm,\P_J}^{Y';Y},\wt{\dt}_{\P_J^+}),
\end{equation}
induced by the morphism  ${\dt}_{\P_{J\cup\al}^+}\to{\dt}_{\P_{J}^+}$ from \re{consind} (b), is an equivalence.

As in the proof of \rt{rss}, we get equivalences $i^*_{\gm,n}(A^{Y_w}_{\P_{J}}\ast\dt_{\I_n^+})\cong  i^*_{\gm,n}(A^{Y_w}_{\P_{J}})$ and $i^*_{\gm,n}(A^{Y_w}_{\P_{J\cup\al}})\cong i^*_{\gm,n}(A^{Y_w}_{\P_{J\cup\al}}\ast\dt_{\I_n^+})$. Therefore,
by \rl{isom}, the morphism ${\dt}_{\P_{J\cup\al}^+}\to{\dt}_{\P_{J}^+}$ from \re{consind} (b)
induces an equivalence $i^*_{\gm,n}(A^{Y_w}_{\P_{J\cup\al}})\isom i^*_{\gm,n}(A^{Y_w}_{\P_{J}})$,
hence an equivalence $\wt{A}^{Y_w}_{\P_{J\cup\al}}(\gm)\isom \wt{A}^{Y_w}_{\P_{J}}(\gm)$ (see \re{rem}).

Finally, using equality $\Fl_{\gm,\P_{J}}^{Y';Y}=\Fl_{\gm,\P_{J}}^{Y_w}$, and arguing as in \re{rem}, the last equivalence coincides with \form{can}.
\end{proof}

\begin{Emp} \label{E:outline}
\begin{proof}[Proof of \rt{tf}]
By \rco{reduction}, we have to show equality \form{eq}.

Let  $\C{C}^\star_\star(\gm)$ be a geometric lift of ${C}^\star_\star(\gm)$. Since the dimension of $\Fl_{\gm}$ is finite, the dimensions of
$\{\Fl^Y_{\gm,\P}\}_{Y,\P}$ are bounded. Thus there exist $n\leq m$ such that the image of $\C{C}^\star_\star(\gm)$ lies in $\C{D}^{[n,m]}(k)$. We are going to apply \rl{cohom}.

By construction, the equivalence $\C{D}^{[0]}(k)\to Sh_c(k)$ from \re{discrete} (b) identifies functor $h^{-i}\circ \C{C}^\star_\star(\gm): \T\times\Par\to\C{D}^{[0]}(k)$ with $C^\star_{i,\star}(\gm):\T\times\Par\to Sh(k)$. Therefore it follows from \rl{cone} that  $C^\star_{i}(\gm):\T\to\Ch^{*,b}_c(Sh(k))\subset\Ch^*(Sh(k))$ represents
$\colim_{\P\in \Par} (h^{-i}\circ \C{C}^\star_\P(\gm)):\T\to \C{D}(k)$.

Next, by \rl{hom}, the inductive limit
$C_{i}(\gm)=\indlim_Y C^Y_{i}(\gm)\in \Ch^*(Sh(k))$ belongs $\Ch^{*,b}_c(Sh(k))$.
Since $\T$ is filtered, we conclude from \re{cone} (a) that $C_{i}(\gm)$
 represents
$\colim_Y C^Y_{i}(\gm)=\colim_ Y\colim_{\P}(h^{-i}\circ \C{C}^Y_\P(\gm))\in\C{D}(k)$.

Thus, by \rl{cohom}, there exists a colimit $\C{C}(\gm):=\colim_Y \colim_{\P}\C{C}^{Y}_{\P}(\gm)\in \C{D}(k)$, and we have an equality in
$K_0(\C{D}(k))$
\begin{equation} \label{Eq:1}
\lan\C{C}(\gm)\ran=\sum_i (-1)^i \lan C_i(\gm)\ran.
\end{equation}
Similarly, applying the same argument to the functor $\C{C}^Y_\star(\gm):\Par\to\C{D}^{[n,m]}(k)$
we conclude from \rl{cohom} that for every $Y\in\T$ we have an equality
\begin{equation} \label{Eq:2}
\lan\C{C}^Y(\gm)\ran=\sum_i (-1)^i \lan C^Y_i(\gm)\ran.
\end{equation}

Recall that $\C{C}(\gm)\cong\colim_{Y\in\T} \C{C}^Y(\gm)$, the transition maps
$\C{C}^Y(\gm)\to \C{C}^{Y'}(\gm)$ are equivalences for each sufficiently large $Y$ (by \rp{stab}),
and $\T$ is filtered. Therefore the map $\C{C}^Y(\gm)\to \C{C}(\gm)$ is an equivalence for each sufficiently large $Y\in\T$.
Hence $\lan\C{C}^Y(\gm)\ran=\lan\C{C}(\gm)\ran$, thus equality \form{eq} follows from \form{1} and \form{2}.
\end{proof}
\end{Emp}

\subsection{Construction of a geometric lift}
In this subsection we construct a geometric lift $\C{C}_{\star}^{\star}(\gm)$ of ${C}_{\star}^{\star}(\gm)$.
As in subsection \ref{SS:compl}, we omit $\Fr$ from the notation.

\begin{Emp} \label{E:dec}
{\bf Decomposition of functor $C^{\star}_{\star}(\gm)$.} First we decompose $C^{\star}_{\star}(\gm)$ as a composition of four simpler functors.

(a) Let $\ASt_k$ be the category of admissible stacks over $k$ (see \re{stacks}). Denote by
$D(\ASt_k,*)$ category over $\ASt_k^{op}$, whose objects are pairs $(X\in \ASt_k,\C{A}\in D(X))$, and for every
$(X,\C{A}),(X',\C{A}')\in \Ob D(\ASt_k,*)$, the set of morphisms $(X,\C{A})\to (X',\C{A}')$ is the disjoint union $\sqcup_{f\in\Hom(X',X)}\Hom_{D(X')}(f^*\C{A},\C{A}')$. We also denote by $D(\Var_k,*)$ the full subcategory of
$D(\ASt_k,*)$, whose objects are pairs $(X,\C{A})$ with $X\in \Var_k$.

(b) Let $\ASt_k^{[1]}$ be the category of morphisms in $\ASt_k$, thus $\ASt_k^{[1]}\times_{\ASt_k}D(\ASt_k,*)^{op}$
(resp. $\Var_k\times_{\ASt_k}\ASt_k^{[1]}\times_{\ASt_k}D(\ASt_k,*)^{op}$) is the category of pairs $(f,\C{A})$, where $f:X\to Y$ is a morphism in $\ASt_k$ (resp. with $X\in \Var_k$) and $\C{A}\in D(Y)$.

We have a natural functor
$\ASt_k^{[1]}\times_{\ASt_k}D(\ASt_k,*)^{op}\to D(\ASt_k,*)^{op}$, which maps a pair $(f:X\to Y, \C{A})$ to
$(X,f^*\C{A})$. This functor restricts to a functor
\begin{equation} \label{Eq:(b)}
\Var_k\times_{\ASt_k}\ASt_k^{[1]}\times_{\ASt_k}D(\ASt_k,*)^{op}\to D(\Var_k,*)^{op}.
\end{equation}

(c)  Denote by $R\Gm$ the functor $D(\Var_k,*)\to D(k)$, which sends a pair $(X,\C{A})$ to $R\Gm(X,\C{A})$ and sends a morphism
$f:X'\to X$, $\phi:f^*\C{A}\to\C{A}'$ to the composition
\[R\Gm(X,\C{A})\overset{f^*}{\lra} R\Gm(X',f^*\C{A})\overset{\phi}{\lra} R\Gm(X',\C{A}').\]

(d) Consider a functor $\Par\to D(\ASt_k,*)^{op}$, which maps $\P$ to the pair $(\frac{\P}{\P},\ov{\dt}_{\P^+})$,
and maps an embedding $\eta:\P\hra\Q$ to the pair $([\eta]:\frac{\P}{\P}\to \frac{\Q}{\Q},[\eta]^*\ov{\dt}_{\Q^+}\to \ov{\dt}_{\P^+})$, defined in \re{verspr} (a). This functor induces the functor
\begin{equation} \label{Eq:(d)}
\T\times \Par\to \Var_k\times_{\ASt_k}\ASt_k^{[1]}\times_{\ASt_k}D(\ASt_k,*)^{op},
\end{equation}
which maps $(Y,\P)$ to the pair $(\pr_{\gm,\P}:\Fl_{\gm,\P}^Y\to \frac{\P}{\P},\ov{\dt}_{\P^+})$.

(e) Finally, let
$\B{D}:D(k)^{op}\to D(k)$ be the Verdier duality functor $\C{F}\mapsto \B{D}\C{F}$. Then functor $C^{\star}_{\star}(\gm)$
decomposes as a composition
\[
\T\times \Par\overset{\form{(d)}}{\lra} \Var_k\times_{\ASt_k}\ASt_k^{[1]}\times_{\ASt_k}D(\ASt_k,*)^{op}\overset{\form{(b)}}{\lra} D(\Var_k,*)^{op}\overset{R\Gm}{\to} D(k)^{op}\overset{\B{D}}{\to} D(k).
\]
\end{Emp}

\begin{Emp} \label{E:constr}
{\bf Main construction.}
Let $\Cat_{\infty}$ be the $\infty$-category of (small) $\infty$-categories (see, for example, \cite[3.1]{Lur}).
We claim that there exists a "natural" functor of $\infty$-categories $F:\ASt_k^{op}\to \Cat_{\infty}$ such that $F(X)=\C{D}(X)$ for all
$X\in \ASt_k$ and $F(f)=f^*$ for all $f:X\to Y$. Though this fact seems  to be well-known to specialists, we sketch this construction for the convenience of the reader. The construction consists of two main parts.

{\bf Part 1.} Let $L/\ql$ be a finite extension, $m\subset\B{O}_L$ be the maximal ideal, and  $n\in\B{N}$. We claim that there exists a natural functor of $\infty$-categories $F'_{L,n}:\Var_k^{op}\to \Cat_{\infty}$ such that
$F'_{L,n}(X)=\C{D}^b_c(X,\C{O}_L/m^n)$ for all $X\in \Var_k$, and $F'_{L,n}(f)=f^*$ for all $f:X\to Y$.

By \cite[3.2]{Lur}, to define a functor $F'_{L,n}$ one needs to construct a coCartesian fibration $\C{D}^b_c(\Var_k,\C{O}_L/m^n)\to \Var_k^{op}$.
For each $X\in \Var_k$, let $\Ch^{*,b}_c(X,\C{O}_L/m^n)$ be the category of bounded cochain complexes of $\C{O}_L/m^n$-modules
over $X$ with constructible cohomologies. Recall that $\C{D}^b_c(X,\C{O}_L/m^n)$ is by definition is the $\infty$-category, obtained
by the $\infty$-localization of the $\infty$-category $\Ch^{*,b}_c(X,\C{O}_L/m^n)$ by quasi-isomorphisms.

To construct $\C{D}^b_c(\Var_k,\C{O}_L/m^n)$, we first consider category $C:=\Ch^{*,b}_c(\Var_k,\C{O}_L/m^n)$ over $\Var_k^{op}$,
defined in the same way as $D(\Var_k,*)$ (see \re{dec} (a)), where the derived category $D(X)$ is replaced by  $\Ch^{*,b}_c(X,\C{O}_L/m^n)$.
Next, we denote by $W$ be the set of morphisms $(f:X'\to X,\phi:f^*\C{A}\to\C{A}')$ in $C$, such that $f$ is an isomorphism, and $\phi$ is a quasi-isomorphism. We denote by  $\C{D}^b_c(\Var_k,\C{O}_L/m^n)$ the $\infty$-category, obtained
by the $\infty$-localization of $C$ by $W$. For example, in the model of complete Segal spaces, this localization is simply the fibrant
replacement of the Rezk nerve $N(C,W)$ (see \cite{Re}).

{\bf Part 2.} The rest of the construction is a formal consequence of the fact that the $\infty$-category $\Cat_{\infty}$ has all small limits and colimits. We divide it in five steps.

(I) We claim that $F'_{L,n}$ naturally extends to the functor $F_{L,n}:(\Art_k)^{op}\to \Cat_{\infty}$ (see \re{stacks}), and $F_{L,n}(X)=\C{D}^b_c(X,\C{O}_L/m^n)$ for each $X\in \Art_k$. Explicitly,
\[
F_{L,n}(X)=\lim_{(V\to X)\in \Var_k/X}F'_{L,n}(V)\in \Cat_{\infty}.
\]
Formally, $F_{L,n}$ is the right Kan extension of $F'_{L,n}$.

(II) Functors $\{F_{L,n}\}_n$ form a projective system, therefore we can form a limit
$F_{\C{O}_L}:=\lim_n F_{L,n}:(\Art_k)^{op}\to \Cat_{\infty}$. In particular, $F_{\C{O}_L}(X)=\C{D}^b_c(X,\C{O}_L)$ for each $X\in \Art_k$.

(III) Functor $F_{\C{O}_L}$ naturally defines a functor $F_{L}:=F_{\C{O}_L}[\frac{1}{l}]:(\Art_k)^{op}\to \Cat_{\infty}$ such that $F_{L}(X)=\C{D}^b_c(X,L)$.
Namely, the multiplication by $l$ induces an endomorphism of $F_{\C{O}_L}$, so we can set
$F_{L}:=\colim(F_{\C{O}_L}\overset{\cdot l}{\lra}F_{\C{O}_L}\overset{\cdot l}{\lra}\ldots)$.

(IV) Functors $\{F_{L}\}_L$ form an inductive system, so we can form a colimit
\[
\wt{F}:=\colim_L F_{L}:(\Art_k)^{op}\to \Cat_{\infty}.
\]
 In particular, we have $\wt{F}(X)=\C{D}(X)$ for each $X\in \Art_k$.

(V) Finally, functor $\wt{F}:(\Art_k)^{op}\to \Cat_{\infty}$ extends to a functor $F:\ASt_k^{op}\to \Cat_{\infty}$ such that
$F(X)=\C{D}(X):=\colim_{(V\to X)\in (X/\Art_k)^{op}}\C{D}(V)$. Explicitly, $F$ is the left Kan extension of $\wt{F}$.
\end{Emp}

\begin{Emp}
{\bf Remark.} To construct functor $F:\ASt_k\to \Cat_{\infty}$, we used a classical way of defining derived categories of constructible $\qlbar$-sheaves.
Alternatively, we could use a recent approach of Bhatt--Scholze \cite{BS}, which would make the process shorter.
\end{Emp}


To prove \rcl{functor}, we are first going to construct a lift $\C{C}_{\star}^{\star}(\gm)$ of ${C}_{\star}^{\star}(\gm)$ (see \re{lift}) and then to show that this lift is geometric (see \rcl{geom}).

\begin{Emp} \label{E:lift}
{\bf Construction of a lift.} To construct of a lift $\C{C}_{\star}^{\star}(\gm)$ of ${C}_{\star}^{\star}(\gm)$, we are going to lift all categories and functors defined in \re{dec} to $\infty$-categories.

(a) Let $\C{D}(\ASt_k,*)\to \ASt_k^{op}$ be the coCartesian fibration, corresponding to the functor  $\ASt_k^{op}\to \Cat_{\infty}$, constructed in \re{constr}. By construction, objects of $\C{D}(\ASt_k,*)$ are pairs
$(X\in \ASt_k, \C{A}\in \C{D}(X))$, and for every $(X,\C{A}),(X',\C{A}')\in \Ob\C{D}(\ASt_k,*)$, the mapping space
$\map_{\C{D}(\ASt_k,*)}((X,\C{A}),(X',\C{A}'))$ is $\sqcup_{f:X'\to X}\map_{\C{D}(X')}(f^*\C{A},\C{A}')$.
In particular, the homotopy category of $\C{D}(\ASt_k,*)$ is $D(\ASt_k,*)$.

(b) We claim that there exists a functor
\begin{equation} \label{Eq:b}
\ASt_k^{[1]}\times_{\ASt_k}\C{D}(\ASt_k,*)^{op}\to \C{D}(\ASt_k,*)^{op},
\end{equation}
which maps $(f:X\to Y,\C{A})$ to $(X,f^*\C{A})$ and lifts  the functor from \re{dec} (b).

Let $\frak{S}$ be the $\infty$-category of spaces (see, for example, \cite{KV}). By the Yoneda lemma, it is enough to construct a functor
\begin{equation} \label{Eq:F}
(\ASt_k^{[1]}\times_{\ASt_k}\C{D}(\ASt_k,*)^{op})\times\C{D}(\ASt_k,*)\to\frak{S},
\end{equation}
which sends  a pair
$(f:X\to Y,\C{A}), (Z,\C{B})$ to $\sqcup_{g:Z\to X}\map_{\C{D}(Z)}(g^*f^*\C{A},\C{B})$.
Note that $\Hom_{\ASt_k}(Z,X)\cong\Hom_{\ASt_k^{[1]}}(\Id_Z, f)$, thus
\[
\sqcup_{g:Z\to X}\map_{\C{D}(Z)}(g^*f^*\C{A},\C{B})\cong\map_{(\ASt^{op}_k)^{[1]}\times_{\ASt^{op}_k}\C{D}(\ASt_k,*)}((f,\C{A}),(\Id_Z,\C{B})),
\]
so we can define the map \form{F} to be the composition of the diagonal embedding
\[\C{D}(\ASt_k,*)\to (\ASt_k^{op})^{[1]}\times_{\ASt_k^{op}}\C{D}(\ASt_k,*):(Z,\C{B})\mapsto (\Id_Z,\C{B})\]
and the evaluation map
\[
\map(\cdot,\cdot): ((\ASt^{op}_k)^{[1]}\times_{\ASt_k^{op}}\C{D}(\ASt_k,*))^{op}\times((\ASt^{op}_k)^{[1]}\times_{\ASt_k^{op}}\C{D}(\ASt_k,*))
\to \frak{S}.
\]

(c) Let $\C{D}(\Var_k,*)\subset \C{D}(\ASt_k,*)$ be the full subcategory, whose homotopy category is
$D(\Var_k,*)\subset D(\ASt_k,*)$. Then the functor \form{b} induces a functor of $\infty$-categories
\[
F_1:\Var_k\times_{\ASt_k}\ASt_k^{[1]}\times_{\ASt_k}\C{D}(\ASt_k,*)^{op}\to \C{D}(\Var_k,*)^{op}.
\]

(d) We claim that there is a functor of $\infty$-categories $R\Gm:\C{D}(\Var_k,*)\to \C{D}(k)$, which maps a pair $(X,\C{A})$ to $R\Gm(X,\C{A})$
and lifts the functor from \re{dec} (c). By the Yoneda lemma, we have to construct a functor
\begin{equation} \label{Eq:F'}
\C{D}(k)^{op}\times\C{D}(\Var_k,*)\to\frak{S},
\end{equation}
which maps a pair $(\C{F},(X,\C{A}))$ to $\map_{\C{D}(k)}(\C{F},R\Gm(X,\C{A}))$. Let $p:X\to\Spec k$ be the projection. Then
by adjointness we have natural equivalences
\[
\map_{\C{D}(k)}(\C{F},R\Gm(X,\C{A}))\cong \map_{\C{D}(X)}(p^*\C{F},\C{A})\cong\map_{\C{D}(\Var_k,*)}((\Spec k,\C{F}),(X,\C{A})),
\]
thus we can define \form{F'} to be a composition of the embedding
\[\C{D}(k)^{op}\hra \C{D}(\Var_k,*)^{op}:\C{F}\mapsto (\Spec k,\C{F})\]
and the evaluation map $\C{D}(\Var_k,*)^{op}\times \C{D}(\Var_k,*)\to\frak{S}$.

(e) Consider the full $\infty$-subcategory $A\subset \C{D}(\ASt_k,*)$ with objects $\{(\frac{\P}{\P},\ov{\dt}_{\P^+})\}_{\P\in \Par}$.
By \re{verspr} (b), we know that $\Ext^{-j}([\eta]^*\ov{\dt}_{\Q^+},\ov{\dt}_{\P^+})=0$
for every $j>0$ and every embedding $\eta:\P\hra\Q$ in $\Par$. Thus $A$ is discrete by \re{discrete} (b), hence the natural morphism $A\to \Ho A$ is an equivalence by \re{discrete} (a).
Therefore the functor of categories $\Par\to (\Ho A)^{op}\subset D(\ASt_k,*)^{op}$ from \re{dec} (d) naturally lifts to a functor
of $\infty$-categories $\Par\to A^{op}\subset \C{D}(\ASt_k,*)^{op}$. In particular, it gives rise to a functor
\[
F_2:\T\times \Par\to \Var_k\times_{\ASt_k}\ASt_k^{[1]}\times_{\ASt_k}\C{D}(\ASt_k,*)^{op},
\]
lifting \re{dec} (d).

(f)  Let $\B{D}:\C{D}(k)^{op}\to \C{D}(k)$ be the Verdier duality functor, and take $\C{C}^{\star}_{\star}(\gm)$ to be the composition $\B{D}\circ R\Gm\circ F_1\circ F_2$.
By construction, it is a lift of $C_{\star}^{\star}(\gm)$.
\end{Emp}

\begin{Emp} \label{E:obs}
{\bf Observations.}
(a) Recall that in \re{lift} (d) we constructed a functor of $\infty$-categories $R\Gm:\C{D}(\Var_k)\to\C{D}(k):(X,\C{A})\mapsto R\Gm(X,\C{A})$. By the same arguments, for every $Y\in \Var_k$, one can construct a functor of $\infty$-categories
\[(p_Y)_*:\C{D}(\Var_k)\times_{\Var^{op}_k}(\Var_k/Y)^{op} \to \C{D}(Y):(f:X\to Y,\C{A})\mapsto f_*\C{A}\]
such that the composition
\[
\C{D}(\Var_k)^{op}\times_{\Var_k}(\Var_k/Y)\to  \C{D}(\Var_k)^{op}\overset{R\Gm}{\lra}\C{D}(k)^{op}
\]
naturally decomposes as $R\Gm(Y,\cdot)\circ(p_Y)_*$.

(b) Let $X\in \Var_k$, $i:Z\hra X$ be a closed embedding, and $j:U\hra X$ be an open embedding such that $j(U)=X\sm i(Z)$. We claim that the natural composition  $j_!j^*\to\Id_{\C{D}(X)}\to i_*i^*$ is a fiber sequence in $\C{D}(X)^{\C{D}(X)}$.

This formally follows from the easy fact that for every $\C{F}\in\C{D}(X)$ and a geometric point $\ov{x}$ of $X$ the induced sequence $(j_!j^*\C{F})_{\ov{x}}\to\C{F}_{\ov{x}}\to (i_*i^*\C{F})_{\ov{x}}$ is a fiber sequence.

Namely, let $\C{A}$ be the fiber of the unit map $\Id_{\C{D}(X)}\to i_*i^*$  in $\C{D}(X)^{\C{D}(X)}$. Then $\C{A}(\C{F})$ is the fiber of $\C{F}\to i_*i^*\C{F}$ for every $\C{F}\in\C{D}(X)$,
hence $\C{A}(\C{F})_{\ov{x}}$ is the fiber of $\C{F}_{\ov{x}}\to (i_*i^*\C{F})_{\ov{x}}$ for every  geometric point $\ov{x}$ of $X$.

Since the composition $j_!j^*\to\Id_{\C{D}(X)}\to i_*i^*$ is the zero morphism, it induces a canonical
morphism $\nu:j_!j^*\to\C{A}$, and it remains to show that $\nu$ is an equivalence.

Note that $\nu$ induces a morphism  $\nu(\C{F})_{\ov{x}}:(j_!j^*\C{F})_{\ov{x}}\to\C{A}(\C{F})_{\ov{x}}$ for all $\C{F}$ and $\ov{x}$.
 As both $(j_!j^*\C{F})_{\ov{x}}\to\C{F}_{\ov{x}}\to (i_*i^*\C{F})_{\ov{x}}$ and $\C{A}(\C{F})_{\ov{x}}\to\C{F}_{\ov{x}}\to (i_*i^*\C{F})_{\ov{x}}$ are fiber sequences, we conclude that  $\nu(\C{F})_{\ov{x}}$ is an equivalence for all $\C{F}$ and $\ov{x}$. Hence $\nu(\C{F}):j_!j^*\C{F}\to\C{A}(\C{F})$ is an equivalence for all $\C{F}$, thus $\nu$ is an equivalence.
\end{Emp}

\begin{Cl} \label{C:geom}
The lift $\C{C}^{\star}_{\star}(\gm)$, constructed in \re{lift}, is geometric.
\end{Cl}

\begin{proof}
We divide the proof into steps.

 {\bf Step 1.} First, we reformulate the assertion in more concrete terms. In the notation of \re{lift}, consider the functor  $\wt{\C{C}}^{\star}_{\star}(\gm):=R\Gm\circ F_1\circ F_2:\T\times\Par\to\C{D}(k)^{op}$.
In other words, $\wt{\C{C}}^{\star}_{\star}(\gm)=\B{D}\circ{\C{C}}^{\star}_{\star}(\gm)$, thus
$\wt{\C{C}}^{Y}_{\P}(\gm)=R\Gm(\Fl_{\gm,\P}^Y,\wt{\dt}_{\P^+})$ for each $Y\in\T$ and $\P\in\Par$.

We fix $Y\subset Y'$ in $\T$ and $\P\subset\Q$ in $\Par$ as in \re{geomlift} (e), set $Y_0:=Y'$, $Y_1:=Y$, $\P_0:=\Q$, $\P_1:=P$, and consider functor $\epsilon:[1]\times[1]\to\T^{op}\times\Par^{op}:(i,j)\mapsto(Y_i,\P_j)$.

Let $\C{E}:[1]\times[1]\to\C{D}(k)$ be the composition   $\wt{\C{C}}^{\star}_{\star}(\gm)^{op}\circ\epsilon$, and let $\C{E}'$ be the corresponding functor
$[1]\to\C{D}(k)^{[1]}:i\mapsto [j\mapsto \C{E}(i,j)]$. Then $\C{E}'$ is a morphism in the $\infty$-category $\C{D}(k)^{[1]}$, thus we can consider its fiber $\Fib\C{E}'\in \C{D}(k)^{[1]}$.

 Thus $\Fib\C{E}'$ is a morphism in the $\infty$-category $\C{D}(k)$, and, by definition, the assertion that $\C{C}^{\star}_{\star}(\gm)$ is geometric means that the induced  morphism  $[\Fib\C{E}']$ in $D(k)$ is naturally isomorphic to the morphism \form{morph} from \re{remasf}.

{\bf Step 2.} Next, we rewrite the morphism \form{morph} in a more convenient form.
In the notation of \re{remasf}, we consider a commutative diagram
\begin{equation*} \label{Eq:open}
\begin{CD}
\wt{\eta}^{-1}(\Fl_{\gm,\P_0}^{Y_0;Y_1}) @>j''>> \Fl_{\gm,\P_1}^{Y_0;Y_1} @>j_1>>\Fl_{\gm,\P_1}^{Y_0} \\
@V\wt{\eta}'VV @VVV @V\wt{\eta}VV\\
\Fl_{\gm,\P_0}^{Y_0;Y_1} @= \Fl_{\gm,\P_0}^{Y_0;Y_1} @>j_0>>  \Fl_{\gm,\P_0}^{Y_0},
\end{CD}
\end{equation*}
and set $j':=j_1\circ j'':\wt{\eta}^{-1}(\Fl_{\gm,\P_0}^{Y_0;Y_1})\hra \Fl_{\gm,\P_1}^{Y_0}$.  Consider the composition
\begin{equation} \label{Eq:map}
j_{0!}j_0^*\wt{\dt}_{\P_0^+}\overset{\iota}{\lra} j_{0!}j_0^*\wt{\eta}_*\wt{\dt}_{\P_1^+}\overset{BC}{\lra}  j_{0!}\wt{\eta}'_* j'^*\wt{\dt}_{\P_1^+}=\wt{\eta}_*j'_{!}j'^*\wt{\dt}_{\P_1^+}\overset{c}{\lra}\wt{\eta}_*j_{1!}j^*_1\wt{\dt}_{\P_1^+}
\end{equation}
in $D(\Fl_{\gm,\P_0}^{Y_0})$, where $\iota$ is induced by the morphism $\iota:\wt{\dt}_{\P_0^+}\to \wt{\eta}_*\wt{\dt}_{\P_1^+}$, adjoint to the map $\wt{\eta}^*\wt{\dt}_{\P_0^+}\to\wt{\dt}_{\P_1^+}$ from \re{aplasf} (b), $BC$ stands for the base change isomorphism, and $c$ is induced by the counit map $j'_{!}j'^*\cong j_{1!}j''_{!}j''^*j_1^*\to j_{1!}j^*_1$.

Then the morphism \form{morph} from \re{remasf} is naturally isomorphic to the morphism
$R\Gm(\Fl_{\gm,\P_0}^{Y_0},j_{0!}j_0^*\wt{\dt}_{\P_0^+})\to R\Gm(\Fl_{\gm,\P_0}^{Y_0},\wt{\eta}_*j_{1!}j^*_1\wt{\dt}_{\P_1^+})$, induced by \form{map}.

{\bf Step 3.} Now, we reduce the problem to a question about an isomorphism of morphisms in $D(\Fl_{\gm,\P_0}^{Y_0})$. Note that the composition \[
F_1\circ F_2\circ \epsilon:[1]\times[1]\to \C{D}(\Var_k):(i,j)\mapsto (\Fl_{\gm,\P_j}^{Y_i},\wt{\dt}_{\P_j^+})
\]
 naturally lifts to a functor
 $F_3:[1]\times[1]\to \C{D}(\Var_k)\times_{\Var^{op}_k}(\Var_k/\Fl_{\gm,\P_0}^{Y_0})^{op}$. Then, by \re{obs} (a), functor $\C{E}$ from Step 1  decomposes
 as a composition of
 \[
 \wt{\C{E}}:=(p_{\Fl_{\gm,\P_0}^{Y_0}})_*\circ F_3:[1]\times[1]\to \C{D}(\Fl_{\gm,\P_0}^{Y_0})
 \]
 and $R\Gm(\Fl_{\gm,\P_0}^{Y_0},\cdot):\C{D}(\Fl_{\gm,\P_0}^{Y_0})\to\C{D}(k)$.

As in Step 1, $\wt{\C{E}}$ gives rise to the morphism $\wt{\C{E}}'$ in the $\infty$-category $\C{D}(\Fl_{\gm,\P_0}^{Y_0})^{[1]}$ such that
$\C{E}'\cong R\Gm(\Fl_{\gm,\P_0}^{Y_0},\cdot)\circ \wt{\C{E}}'$. Hence, since functor $R\Gm(\Fl_{\gm,\P_0}^{Y_0},\cdot)$ commutes with all limits, we have  $\Fib\C{E}'\cong R\Gm(\Fl_{\gm,\P_0}^{Y_0},\cdot)\circ \Fib\wt{\C{E}}'$. Therefore, by Step 2, it suffices to
show that the induced morphism $[\Fib\wt{\C{E}}']$ in the homotopy category $D(\Fl_{\gm,\P_0}^{Y_0})$ is naturally isomorphic to the map \form{map}.

{\bf Step 4.} Consider commutative diagram
\begin{equation*} \label{Eq:closed}
\begin{CD}
\Fl_{\gm,\P_1}^{Y_1} @>i''>> \wt{\eta}^{-1}(\Fl_{\gm,\P_0}^{Y_1}) @>i'>>\Fl_{\gm,\P_1}^{Y_0} \\
@VVV @VVV @V\wt{\eta}VV\\
\Fl_{\gm,\P_0}^{Y_1} @= \Fl_{\gm,\P_0}^{Y_1} @>i_0>>  \Fl_{\gm,\P_0}^{Y_0},
\end{CD}
\end{equation*}
and set $i_1:=i'\circ i'':\Fl_{\gm,\P_1}^{Y_1}\hra \Fl_{\gm,\P_1}^{Y_0}$.

By construction, the functor  $\wt{\C{E}}:[1]\times[1]\to \C{D}(\Fl_{\gm,\P_0}^{Y_0})$ is represented by a commutative square in
$\C{D}(\Fl_{\gm,\P_0}^{Y_0})$, corresponding to the exterior square of the diagram
\begin{equation} \label{Eq:e}
\begin{CD}
\wt{\dt}_{\P_0^+} @>\iota>> \wt{\eta}_*\wt{\dt}_{\P_1^+} @=\wt{\eta}_*\wt{\dt}_{\P_1^+} @=\wt{\eta}_*\wt{\dt}_{\P_1^+}\\
 @VuVV @VuVV @VuVV @VuVV \\
i_{0*}i_0^*\wt{\dt}_{\P_0^+} @>\iota>> i_{0*}i_0^*\wt{\eta}_*\wt{\dt}_{\P_1^+} @>BC >>\wt{\eta}_*i'_{*}i'^*\wt{\dt}_{\P_1^+} @>u'>>\wt{\eta}_*i_{1*}i^*_1\wt{\dt}_{\P_1^+},
\end{CD}
\end{equation}
where $u$ stands for the unit map, $BC$ and $\iota$ have the same meaning as in \form{map}, and $u'$ is induced by the unit map $i'_{*}i'^*\to i'_{*}i''_{*}i''^*i'^*\cong i_{1*}i^*_1$.

Therefore, morphism $\wt{\C{E}}'$ in $\C{D}(\Fl_{\gm,\P_0}^{Y_0})^{[1]}$ can be viewed as a morphism from
the top row in \form{e} to the bottom one.

{\bf Step 5.} Note that the left inner square of \form{e} is commutative by functoriality, while the remaining inner squares are commutative  by a straightforward diagram chase. Hence, by \re{obs} (b), the fiber  $[\Fib \wt{\C{E}}']$ is naturally isomorphic to the map \form{map}. By Step 3, this completes the proof of the claim.
\end{proof}

\end{document}